\documentclass[11pt]{article}

\usepackage[T1]{fontenc}
\usepackage[utf8]{inputenc}
\usepackage{lmodern}
\usepackage{microtype}
\usepackage{amsmath,amssymb,amsthm,mathtools,mathrsfs}
\usepackage{bm}
\usepackage{enumitem}
\usepackage{xcolor}
\usepackage[top=1in,bottom=1in,left=1in,right=1in]{geometry}
\usepackage[colorlinks=true,linkcolor=blue!55!black,citecolor=blue!55!black,urlcolor=blue!55!black]{hyperref}

\def\E{\mathbb{E}}

\def\P{\mathbb{P}}
\def\Q{\mathbb{Q}}
\def\R{\mathbb{R}}

\def\Ac{{\mathcal A}}

\def\Dc{{\mathcal D}}

\def\Fc{{\mathcal F}}

\def\Mc{{\mathcal M}}
\def\Nc{{\mathcal N}}
\def\Oc{{\mathcal O}}

\def\Uc{{\mathcal U}}

\def\Xc{{\mathcal X}}

\def\id{\mathbf{1}}

\def\hOm{\widehat{\Omega}^{\,N}}
\def\hX{\widehat{\Xc}^{\,N}}
\newcommand{\dd}{\mathrm{d}}
\newcommand{\Leb}{\mathcal L}
\newcommand{\HH}{\mathcal H}

\newcommand{\dist}{\operatorname{dist}}

\newcommand{\Ent}{\operatorname{Ent}}
\newcommand{\tr}{\operatorname{tr}}
\newcommand{\argmin}{\operatorname*{argmin}}

\newcommand{\restr}{\mathbin{\vrule height 1.4ex depth -0.3ex width 0.08ex\vrule height 0.08ex depth -0.02ex width 0.7ex}}
\newcommand{\one}{\mathbf 1}
\newcommand{\norm}[1]{\left\lVert #1\right\rVert}

\newcommand{\Wtwo}{W_2}
\newcommand{\Kant}{\mathscr K}
\newcommand{\Coll}{\mathscr C}

\newcommand{\Record}{\widehat X}
\newcommand{\pmin}{p_{\min}}
\newcommand{\pmax}{p_{\max}}

\newtheorem{theorem}{Theorem}[section]
\newtheorem{lemma}[theorem]{Lemma}

\newtheorem{proposition}[theorem]{Proposition}

\theoremstyle{definition}
\newtheorem{assumption}{Assumption}

\newtheorem{algorithm}[theorem]{Algorithm}
\newtheorem{example}[theorem]{Example}
\theoremstyle{remark}
\newtheorem{remark}[theorem]{Remark}

\numberwithin{equation}{section}

\title{Semi-discrete quadratic Wasserstein energy and state-dependent Langevin exploration}

\author{%
  Ran Gu\thanks{NITFID, School of Statistics and Data Science,
  Nankai University. \href{mailto:rgu@nankai.edu.cn}{rgu@nankai.edu.cn}}
  \and
  Gaoyue Guo\thanks{MICS, CentraleSup\'elec, Universit\'e Paris-Saclay.
  \href{mailto:gaoyue.guo@centralesupelec.fr}{gaoyue.guo@centralesupelec.fr}. This work was supported by the grant ANR-25-CE40-0714 (MATH-SPA). }
  \and
  Kelvin Shuangjian Zhang\thanks{School of Mathematical Sciences,
  Fudan University. \href{mailto:ksjzhang@fudan.edu.cn}{ksjzhang@fudan.edu.cn}}
}
\date{}

\hypersetup{
  pdftitle={Semi-discrete quadratic Wasserstein energy and state-dependent Langevin exploration},
  pdfauthor={Ran Gu, Gaoyue Guo, Kelvin Shuangjian Zhang},
  pdfkeywords={semi-discrete optimal transport, quadratic Wasserstein energy, Laguerre cells, Lloyd algorithm, exploratory HJB equation, Langevin diffusion, geometric ergodicity, global optimization}
}

\begin{document}
\maketitle

\begin{abstract}
Let $\Omega\subset\mathbb R^d$ be compact and convex with nonempty interior, and let $\nu$ be a
probability measure on $\Omega$ with density
\[
 \rho\in C(\Omega)\cap W^{1,1}(\Omega^\circ),
 \qquad
 0<\underline\rho\le\rho\le\overline\rho<\infty.
\]
For prescribed masses $p_1,\ldots,p_N>0$, we study the semi-discrete quadratic Wasserstein energy
\[
 E(X)=\frac12W_2^2\!\left(\nu,\sum_{i=1}^Np_i\delta_{x_i}\right),
 \qquad X=(x_1,\ldots,x_N).
\]
The energy is nonsmooth at collisions of sites.  We prove local Lipschitz continuity on the full
configuration space, together with global semiconcavity, coercivity, and dissipativity; show that every
global minimizer is interior and collision free; and establish $C^2$ regularity on the collision-free
configuration space.  The gradient is expressed through the barycenters of the balanced Laguerre
cells, while the Hessian is given by an explicit facet formula and satisfies a global one-sided bound.
We also solve the one-dimensional problem explicitly in each ordering chamber and give a two-site
example on the unit square with non-minimizing Lloyd fixed points.

For $d\ge2$, we then formulate an entropy-regularized relaxed control of the Langevin temperature.
The controlled dynamics is strongly well posed, nonexplosive, and collision free.  Its value function
is a classical interior solution of the exploratory Hamilton--Jacobi--Bellman equation; the Laplacian
of the value function is locally $C^1$, which yields a locally Lipschitz optimal temperature feedback.
Independently of this optimal-control result, for every $d\ge1$, every fixed Borel temperature rule
bounded away from zero, and every sufficiently small step size, the associated Gaussian Euler chain
is geometrically ergodic with a full-support invariant law.  The raw iterates do not converge, whereas
the best-so-far energy converges almost surely to the global minimum and the running record approaches
the set of global minimizers.

\medskip
\noindent\textbf{Keywords.}
Semi-discrete Wasserstein distance; Laguerre cells; Lloyd algorithm; exploratory HJB  equation; Langevin diffusion;
geometric ergodicity. 
\end{abstract}

\section{Introduction}\label{sec:intro}

Approximating a continuous probability measure by a finitely supported measure with prescribed
atomic masses is a basic problem in semi-discrete optimal transport.  In capacitated location and
urban-planning models, the target measure describes the spatial distribution of demand, the sites
represent facilities, and the prescribed masses represent their capacities.  Closely related
formulations arise in vector quantization, centroidal Voronoi and Laguerre tessellations, point-cloud
approximation, and numerical integration; see
\cite{ButtazzoSantambrogio2009,DuFaberGunzburger1999,GrafLuschgy2000,
BourneRoper2015,MerigotSantambrogioSarrazin2021}.  A natural deterministic method in the present
fixed-mass setting is the weighted Lloyd iteration: one first constructs the Laguerre cells having the
prescribed masses and then moves each site to the barycenter of its cell.  Convergence properties of
Lloyd-type iterations have been studied in
\cite{DuEmelianenkoJu2006,EmelianenkoJuRand2008,BourneRoper2015,
PortalesCazellesPauwels2025}.  These results concern convergence toward fixed or critical
configurations under their respective assumptions; they do not by themselves provide a
global-optimization guarantee, because non-minimizing critical points can occur.  An explicit example
is given in Proposition~\ref{prop:square-critical-points}.

This paper develops a global and local regularity theory for the corresponding semi-discrete
Wasserstein energy and then uses it as the potential of a state-dependent Langevin exploration.  The
main analytical obstruction is the collision set, where two or more labeled sites coincide and the
energy is generally not differentiable.  We therefore use two complementary descriptions.  A
splitting formulation remains valid on the whole configuration space and yields collision-valid
semiconcavity, growth, and dissipativity estimates.  On the open collision-free set, the optimal
splitting is induced by balanced Laguerre cells and can be differentiated classically.  In dimensions
$d\ge2$, the collision diagonals are polar for the uniformly nondegenerate diffusions considered
below.  This permits a stochastic-control analysis on the punctured state space without imposing
artificial boundary data on the collision set.

The continuous-time and discrete-time parts answer two distinct questions.  The first is an optimal
control problem: among randomized temperatures, which state-dependent law minimizes a specified
discounted entropy-regularized criterion?  The second is an asymptotic exploration problem: for
which discrete temperature rules does the best energy observed up to time $n$ converge to the
global minimum?

\subsection{Problem formulation}\label{subsec:problem-formulation}

Write $\Xc:=\mathbb R^d$ and fix an integer $N\ge2$.  We work under the following standing assumption.

\begin{assumption}[Target measure and prescribed masses]\label{ass:geometry}
{\rm (i)} The set $\Omega\subset\R^d$ is compact and convex with nonempty interior $\Omega^\circ$.

\medskip

\noindent {\rm (ii)} The probability measure $\nu$ has the form
\begin{align}
 \nu(\dd y)&=\rho(y)\id_{\Omega}(y)\dd y,
 \label{eq:density-assumption-main}\\
 \rho&\in C(\Omega)\cap W^{1,1}(\Omega^\circ),
 \qquad
 0<\underline\rho:=\min_{y\in\Omega}\rho(y)
 \le \max_{y\in\Omega}\rho(y)=:\overline\rho<\infty,
 \notag\\
 \int_\Omega\rho(y)\dd y&=1.
 \notag
\end{align}

\noindent {\rm (iii)} The prescribed masses $p_1,\ldots,p_N$ satisfy
\begin{equation}\label{eq:weights-main}
 \sum_{i=1}^Np_i=1,
 \qquad
 0<\pmin:=\min_{1\le i\le N}p_i,
 \qquad
 \pmax:=\max_{1\le i\le N}p_i.
\end{equation}
\end{assumption}

For $X=(x_1,\ldots,x_N)\in\Xc^N$, define
\begin{equation}\label{eq:muX-intro}
 \mu_X:=\sum_{i=1}^Np_i\delta_{x_i},
\end{equation}
and the semi-discrete quadratic Wasserstein energy
\begin{equation}\label{eq:E-intro}
 E(X):=\frac12W_2^2(\nu,\mu_X).
\end{equation}
Our optimization problem is
\begin{equation}\label{pb:min}
 E_*:=\inf_{X\in\Xc^N}E(X),
\end{equation}
and the set of global minimizers is
\begin{equation}\label{eq:minimizer-set}
 \Mc_*:=\argmin_{X\in\Xc^N}E(X).
\end{equation}
The collision set and the collision-free configuration spaces are
\begin{equation}\label{eq:collision-spaces}
 \Coll:=\bigcup_{1\le i<j\le N}\{X\in\Xc^N:x_i=x_j\},
 \qquad
 \hX:=\Xc^N\setminus\Coll,
 \qquad
 \hOm:=\hX\cap\Omega^N.
\end{equation}
All derivatives below are ordinary Euclidean derivatives on the open set $\hX\subset\Xc^N$.

Let $\mathcal M_+(\Omega)$ denote the set of finite nonnegative Borel measures on $\Omega$, and
define the set of labeled splittings of $\nu$ with the prescribed masses by
\begin{equation}\label{eq:splittings}
 \Ac_p(\nu):=
 \left\{
 \boldsymbol\nu=(\nu_1,\ldots,\nu_N):
 \nu_i\in\mathcal M_+(\Omega),\quad
 \sum_{i=1}^N\nu_i=\nu,\quad
 \nu_i(\Omega)=p_i
 \right\}.
\end{equation}
Every $\boldsymbol\nu\in\Ac_p(\nu)$ induces a transport plan
\[
 \gamma_{\boldsymbol\nu}(\dd y,\dd x)
 :=\sum_{i=1}^N\nu_i(\dd y)\delta_{x_i}(\dd x)
 \in\Pi(\nu,\mu_X).
\]
Conversely, every plan in $\Pi(\nu,\mu_X)$ admits at least one such labeled decomposition; when
sites collide, the decomposition need not be unique.  Consequently,
\begin{equation}\label{eq:E-splitting}
 E(X)=\inf_{\boldsymbol\nu\in\Ac_p(\nu)}
 \frac12\sum_{i=1}^N\int_\Omega|x_i-y|^2\nu_i(\dd y).
\end{equation}
The infimum is attained because $\Ac_p(\nu)$ is narrowly compact.  We write
\begin{equation}\label{eq:first-moments-splitting}
 m_i(\boldsymbol\nu):=\int_\Omega y\,\nu_i(\dd y),
 \qquad
 m_2(\nu):=\int_\Omega|y|^2\nu(\dd y).
\end{equation}
Representation~\eqref{eq:E-splitting} is valid at collisions and is the source of the global
estimates in Section~\ref{sec:objective}.  On $\hX$, the optimal splitting is induced by a balanced
Laguerre partition and admits a differential analysis.

\subsection{Relation to previous work and contributions}
\label{subsec:literature-contributions}

The paper combines semi-discrete optimal transport, entropy-regularized
exploratory control, and Harris-chain analysis of Langevin discretizations.  We state explicitly
which ingredients are imported and which conclusions are specific to the present problem.

\begin{enumerate}[label=\textup{(\roman*)},leftmargin=2.8em]

\item \emph{Semi-discrete energy.}
The semi-dual formulation and balanced Laguerre representation are standard in semi-discrete
optimal transport; see, for example, \cite{Villani2009,KitagawaMerigotThibert2019}.  The strict concavity and weighted graph-Laplacian structure used to balance the cells are
closely related to the analysis in that reference.  De~Gournay, Kahn, and Lebrat
\cite{deGournayKahnLebrat2019} establish joint second-order differentiation of semi-discrete
transport functionals with respect to the parameters of the discrete measure and provide the
moving-cell formulas used in Appendix~\ref{app:parametric}.

The contributions specific to the present energy are the collision-valid splitting analysis on all of
$\Xc^N$, including global semiconcavity, explicit supergradients, coercivity, and one-sided
dissipativity; the proof that every global minimizer is interior and collision free; and the
organization of the second variation into the symmetric facet form
\[
 D^2E=P-\mathcal Q_X,
 \qquad \mathcal Q_X\ge0,
\]
which yields the global one-sided Hessian bound needed in the stochastic analysis.  We also give an
explicit prescribed-mass solution in every one-dimensional ordering chamber and a complete
classification of the two-site critical configurations for the uniform measure on the unit square.
The one-dimensional result should be contrasted with the classical quantization literature, where
the atom masses are generally induced by the Voronoi cells rather than fixed in advance; see
\cite{GrafLuschgy2000}.

\item \emph{Exploratory temperature control.}
The continuous-time exploratory-control formulation originates in the entropy-regularized relaxed
control framework of Wang, Zariphopoulou, and Zhou
\cite{WangZariphopoulouZhou2020}.  Tang, Zhang, and Zhou
\cite{TangZhangZhou2022} establish well-posedness and regularity for the viscosity solution of the
associated exploratory HJB equation, together with a vanishing-exploration limit.  Gao, Xu, and Zhou
\cite{GaoXuZhou2022} apply this framework to state-dependent temperature control for Langevin
diffusions and derive the truncated exponential, or Gibbs, temperature law.  Accordingly, the
entropy relaxation, the scalar log-partition Hamiltonian, and the pointwise Gibbs minimizer appearing
below are specializations of that framework rather than new constructions.

What is specific here is the singular semi-discrete potential.  The state space is the punctured open
set $\hX$, the drift $-\nabla E$ is only locally Lipschitz there and has linear rather than bounded
growth, and no boundary value is imposed on the collision diagonals.  The global dissipativity from
the first part replaces a global bounded-gradient hypothesis, while polarity in dimension $d\ge2$
replaces boundary conditions on the internal singular set.  We prove global strong well-posedness and
collision avoidance for every relaxed control, establish the dynamic programming principle on the
raw canonical space, identify the value function as an interior viscosity solution, obtain its local
$C^{2,\beta}$ regularity, and then use the HJB identity itself to bootstrap
$\Delta v_\lambda$ to $C^1_{\mathrm{loc}}$.  This yields a locally Lipschitz optimal feedback and a
strong verification theorem.  Recent numerical work on exploratory HJB temperature control is
discussed in \cite{WangLiWangZhang2026}; the present paper does not analyze a particular HJB
solver.

\item \emph{Euler exploration and the running record.}
Classical global-optimization results for Langevin dynamics typically use a temperature that decreases
to zero; see, among many references,
\cite{ChiangHwangSheu1987,HolleyKusuokaStroock1989,GelfandMitter1991}.  Fixed-temperature Langevin and stochastic-gradient
Langevin methods instead lead naturally to recurrent or stationary sampling behavior; for a
nonasymptotic optimization perspective, see \cite{RaginskyRakhlinTelgarsky2017}.  Our Euler scheme
keeps a strictly positive temperature floor, so the raw iterates are not expected to converge to a
minimizer.  The relevant object is the running record.

The geometric-ergodicity proof uses standard Foster--Lyapunov, minorization, and Harris-recurrence
theory from \cite{MeynTweedie2009}.  The contribution is to verify these hypotheses for the exact
semi-discrete drift, despite its collision singularities, and for an arbitrary fixed Borel
state-dependent temperature rule with values in $[a,1]$.  Once full support and positive Harris
recurrence are established, record convergence follows from infinitely many visits to every open
sublevel set above $E_*$.  Thus the long-time record theorem is deliberately temperature-robust: it
covers the exact HJB feedback, any fixed clipped numerical approximation of it, and constant positive
temperatures.  It does not compare their finite-time performance.

\end{enumerate}

\subsection{Main results}\label{subsec:main-results}

The main results follow the three parts just described.

\medskip

\noindent\textbf{Regularity and geometry of the semi-discrete energy.}
For $X\in\hX$ and $\Phi=(\Phi_1,\ldots,\Phi_N)\in\mathbb R^N$, define the Laguerre cells
\[
 V_i(X,\Phi)
 :=
 \left\{
 y\in\Omega:
 \frac12|y-x_i|^2-\Phi_i
 \le
 \frac12|y-x_j|^2-\Phi_j
 \ \text{for every }j=1,\ldots,N
 \right\}.
\]
There is a unique normalized vector
\[
 \Phi^*(X)\in
 U:=\left\{\Phi\in\mathbb R^N:\sum_{i=1}^N\Phi_i=0\right\}
\]
such that
\[
 \nu\bigl(V_i(X,\Phi^*(X))\bigr)=p_i,
 \qquad i=1,\ldots,N.
\]
We write
\[
 V_i(X):=V_i(X,\Phi^*(X)),
 \qquad
 c_i(X):=\frac1{p_i}\int_{V_i(X)}y\rho(y)\dd y,
\]
and set
\[
 P:=\operatorname{diag}(p_1I_d,\ldots,p_NI_d).
\]
Thus, for $\Xi=(\xi_1,\ldots,\xi_N)$ and
$\Theta=(\theta_1,\ldots,\theta_N)$ in $\Xc^N$,
\[
 \Xi\cdot P\Theta=\sum_{i=1}^Np_i\,\xi_i\cdot\theta_i.
\]

\begin{theorem}\label{thm:C2}
Under Assumption~\ref{ass:geometry},
\begin{equation}\label{eq:E-C2}
 E\in C^2(\hX),
 \qquad
 \Phi^*\in C^1(\hX;U).
\end{equation}
For every $X\in\hX$,
\begin{equation}\label{eq:DE-C2-theorem-main}
 \nabla_{x_i}E(X)=p_i\bigl(x_i-c_i(X)\bigr),
 \qquad i=1,\ldots,N.
\end{equation}
Equivalently, for every $\Xi=(\xi_1,\ldots,\xi_N)\in\Xc^N$,
\[
 DE(X)[\Xi]
 =\sum_{i=1}^Np_i\bigl(x_i-c_i(X)\bigr)\cdot\xi_i.
\]
Moreover, there exists a continuous family
$(\mathcal Q_X)_{X\in\hX}$ of symmetric nonnegative bilinear forms on $\Xc^N$ such that
\begin{equation}\label{eq:hessian-structural-main}
 D^2E(X)[\Xi,\Theta]
 =\Xi\cdot P\Theta-\mathcal Q_X[\Xi,\Theta],
 \qquad \Xi,\Theta\in\Xc^N.
\end{equation}
The bilinear form $\mathcal Q_X$ is given by the facet representation
\eqref{eq:hessian-facet-main}.  In particular,
\begin{equation}\label{eq:hessian-upper-P}
 D^2E(X)[\Xi,\Xi]
 \le\sum_{i=1}^Np_i|\xi_i|^2,
 \qquad \Xi=(\xi_1,\ldots,\xi_N)\in\Xc^N.
\end{equation}
On every compact set $K\Subset\hX$, the maps $D^2E$ and $D\Phi^*$ are bounded and uniformly
continuous.
\end{theorem}

The local result is complemented by global estimates on $\Xc^N$: the energy is locally Lipschitz,
semiconcave, coercive, and dissipative.  Moreover, $\Mc_*$ is nonempty and compact, and
\[
 \Mc_*\subset\hX\cap(\Omega^\circ)^N.
\]
In dimension one, Theorem~\ref{thm:one-dimensional} gives the unique minimizer in each ordering
chamber explicitly in terms of the quantiles of $\nu$.

\medskip

\noindent\textbf{Exploratory temperature control.}
Fix
\[
 0<a<1,
 \qquad
 \vartheta>0,
 \qquad
 \lambda>0.
\]
For $\varpi\in\mathcal P([a,1])$, define
\[
 \bar u(\varpi):=\int_a^1u\,\varpi(\dd u)
\]
and the entropy relative to Lebesgue measure on $[a,1]$ by
\[
 \Ent_{[a,1]}(\varpi)
 :=
 \begin{cases}
 \displaystyle\int_a^1f(u)\log f(u)\dd u,
 &\text{if }\varpi(\dd u)=f(u)\dd u,\\[1.2ex]
 +\infty,
 &\text{otherwise}.
 \end{cases}
\]
A relaxed temperature control is a progressively measurable process
$\pi=(\pi_t)_{t\ge0}$ with values in $\mathcal P([a,1])$.  It drives
\begin{equation}\label{eq:controlled-SDE-main-results}
 \dd X_t^\pi
 =-\nabla E(X_t^\pi)\dd t
 +\sqrt{2\bar u(\pi_t)}\,\dd W_t,
 \qquad
 X_0^\pi=X\in\hX,
\end{equation}
where $W$ is a standard Brownian motion in $\Xc^N$.  The discounted entropy-regularized criterion is
\begin{equation}\label{eq:cost-main-results}
 J_\lambda(X;\pi)
 :=\mathbb E_X\left[
 \int_0^\infty e^{-\vartheta t}
 \left(E(X_t^\pi)+\lambda\Ent_{[a,1]}(\pi_t)\right)\dd t
 \right],
\end{equation}
and
\[
 v_\lambda(X):=\inf_\pi J_\lambda(X;\pi),
\]
where admissibility is specified in Section~3.1.

For $q\in\mathbb R$, introduce
\begin{align}
 \mathscr H_\lambda(q)
 &:=
 \inf_{\varpi\in\mathcal P([a,1])}
 \left\{q\,\bar u(\varpi)+\lambda\Ent_{[a,1]}(\varpi)\right\}
 \notag\\
 &=-\lambda\log\left(\int_a^1e^{-uq/\lambda}\dd u\right).
 \label{eq:Hamiltonian-main-results}
\end{align}
Because $0<a<1$, the tilted density is positive on a nondegenerate interval; hence
$\mathscr H_\lambda$ is smooth, strictly increasing, and strictly concave, with
$\mathscr H_\lambda'(q)\in(a,1)$.  The exploratory HJB equation is
\begin{equation}\label{eq:exploratory-HJB-main-results}
 -\vartheta v_\lambda
 -\nabla E\cdot\nabla v_\lambda
 +E
 +\mathscr H_\lambda(\Delta v_\lambda)
 =0
 \qquad\text{on }\hX.
\end{equation}

\begin{theorem}\label{thm:HJB}
Suppose that Assumption~\ref{ass:geometry} holds and that $d\ge2$.  Then $v_\lambda$ is finite and
continuous on $\hX$, has at most quadratic growth, and is a viscosity solution of
\eqref{eq:exploratory-HJB-main-results}.  There exists $\beta\in(0,1)$ such that
\begin{equation}\label{eq:v-C2beta}
 v_\lambda\in C^{2,\beta}_{\mathrm{loc}}(\hX).
\end{equation}
Moreover,
\begin{equation}\label{eq:q-C1}
 q_\lambda:=\Delta v_\lambda\in C^1_{\mathrm{loc}}(\hX),
\end{equation}
so the HJB equation holds classically on $\hX$.

For every $X\in\hX$, the pointwise Hamiltonian problem
\begin{equation}\label{eq:pointwise-Hamiltonian-minimization}
 \inf_{\varpi\in\mathcal P([a,1])}
 \left\{
 \Delta v_\lambda(X)\,\bar u(\varpi)
 +\lambda\Ent_{[a,1]}(\varpi)
 \right\}
\end{equation}
has a unique minimizer $\varpi_\lambda^*(X)$.  It has density
\begin{equation}\label{eq:optimal-density}
 \pi_\lambda^*(u;X)
 :=
 \frac{\exp[-u\Delta v_\lambda(X)/\lambda]}
 {\displaystyle\int_a^1\exp[-s\Delta v_\lambda(X)/\lambda]\dd s},
 \qquad u\in[a,1],
\end{equation}
that is,
$\varpi_\lambda^*(X)(\dd u)=\pi_\lambda^*(u;X)\dd u$.  Its mean temperature is
\begin{equation}\label{eq:optimal-mean-temperature}
 \tau_\lambda(X)
 :=\bar u\bigl(\varpi_\lambda^*(X)\bigr)
 =\int_a^1u\pi_\lambda^*(u;X)\dd u
 =\mathscr H_\lambda'\bigl(\Delta v_\lambda(X)\bigr)
 \in(a,1).
\end{equation}
The diffusion coefficient
\begin{equation}\label{eq:h-feedback}
 h_\lambda(X):=\sqrt{2\tau_\lambda(X)}
\end{equation}
belongs to $C^1_{\mathrm{loc}}(\hX)$ and satisfies
\[
 \sqrt{2a}\le h_\lambda(X)\le\sqrt2,
 \qquad X\in\hX.
\]
For every $X\in\hX$, the feedback equation
\begin{equation}\label{eq:optimal-feedback-SDE}
 \dd X_t^*
 =-\nabla E(X_t^*)\dd t+h_\lambda(X_t^*)\dd W_t,
 \qquad X_0^*=X,
\end{equation}
has a unique global strong solution which almost surely never reaches $\Coll$.  Finally,
\begin{equation}\label{eq:optimal-relaxed-control-process}
 \pi_t^*(\dd u)
 :=\varpi_\lambda^*(X_t^*)(\dd u)
 =\pi_\lambda^*(u;X_t^*)\dd u
\end{equation}
has finite cost and is optimal:
\begin{equation}\label{eq:verification-value}
 v_\lambda(X)=J_\lambda(X;\pi^*)=\inf_\pi J_\lambda(X;\pi).
\end{equation}
\end{theorem}

The extra derivative in \eqref{eq:q-C1} is not inferred from
$C^{2,\beta}_{\mathrm{loc}}$ alone.  After the HJB equation has been shown to hold pointwise, it is
solved algebraically for $\Delta v_\lambda$ by applying the smooth inverse of
$\mathscr H_\lambda$; see \eqref{eq:laplacian-inversion-detailed}.

\medskip

\noindent\textbf{Euler exploration and convergence of the running record.} Let
\begin{equation}\label{eq:generic-temperature-main-results}
 \tau:\hX\longrightarrow[a,1]
\end{equation}
be an arbitrary fixed Borel function.   Let
$(\xi_n)_{n\ge1}$ be independent standard Gaussian vectors in $\Xc^N$ and define
\begin{equation}\label{eq:euler-chain-main-results}
 \begin{cases}
 X_0^\tau=X,\\[0.4ex]
 \displaystyle
 X_{n+1}^\tau
 =X_n^\tau-\eta\nabla E(X_n^\tau)
 +\sqrt{2\eta\tau(X_n^\tau)}\,\xi_{n+1},
 \qquad n\ge0.
 \end{cases}
\end{equation}
We denote its probability law and expectation by $\mathbb P_X^\tau$ and $\mathbb E_X^\tau$.
For $n\ge0$, let
\begin{equation}\label{eq:record-index-main-results}
 \kappa_n^\tau
 :=\min\left\{
 0\le k\le n:
 E(X_k^\tau)=\min_{0\le\ell\le n}E(X_\ell^\tau)
 \right\},
\end{equation}
and define
\begin{equation}\label{eq:record-definition-main-results}
 \Record_n^\tau:=X_{\kappa_n^\tau}^\tau.
\end{equation}
Thus
\[
 E(\Record_n^\tau)=\min_{0\le k\le n}E(X_k^\tau),
\]
with ties resolved by retaining the earliest minimizing iterate.  Set
\[
 V(X):=1+|X|^2
\]
and
\begin{equation}\label{eq:eta0}
 \eta_0:=\min\left\{1,\frac{\pmin}{4\pmax^2}\right\}.
\end{equation}

\begin{theorem}\label{thm:numerical-convergence}
Suppose that Assumption~\ref{ass:geometry} holds and that $d\ge1$.  Let
$\tau:\hX\to[a,1]$ be an arbitrary fixed Borel function, let $0<\eta\le\eta_0$, and let
$(X_n^\tau)_{n\ge0}$ be defined by \eqref{eq:euler-chain-main-results}.

Then $(X_n^\tau)_{n\ge0}$ is Lebesgue-irreducible and aperiodic on $\hX$ and admits a unique
invariant probability measure $\Pi_{\eta,\tau}$.  This measure has full support on $\hX$ and finite
second moment.  Moreover, there exist constants $C<\infty$ and $r\in(0,1)$, depending on $\eta$ and $\tau$, such that
\begin{equation}\label{eq:V-geometric}
 \sup_{\substack{
 f:\hX\to\mathbb R\ \mathrm{measurable}\\
 |f|\le V
 }}
 \left|
 \mathbb E_X^\tau[f(X_n^\tau)]
 -\int_{\hX}f\,\dd\Pi_{\eta,\tau}
 \right|
 \le C V(X)r^n,
 \qquad X\in\hX,\quad n\ge0.
\end{equation}
The record process satisfies
\begin{equation}\label{eq:record-energy-convergence}
 E(\Record_n^\tau)\longrightarrow E_*
 \qquad \mathbb P_X^\tau\text{-almost surely},
\end{equation}
and
\begin{equation}\label{eq:record-distance-convergence}
 \dist(\Record_n^\tau,\Mc_*)\longrightarrow0
 \qquad \mathbb P_X^\tau\text{-almost surely}.
\end{equation}
If $\Mc_*=\{X_*\}$, then
\[
 \Record_n^\tau\longrightarrow X_*
 \qquad \mathbb P_X^\tau\text{-almost surely}.
\]
By contrast, the raw iterates do not converge to a finite limit:
\begin{equation}\label{eq:raw-not-converge}
 \mathbb P_X^\tau\left[
 \lim_{n\to\infty}X_n^\tau\text{ exists in }\Xc^N
 \right]=0.
\end{equation}
\end{theorem}

The exact optimal mean temperature $\tau_\lambda$ is one admissible choice.  More generally, if
$q:\hX\to\mathbb R$ is a fixed Borel approximation of
$q_\lambda=\Delta v_\lambda$, define
\begin{equation}\label{eq:tau-q-main-results}
 \tau_q(X)
 :=
 \frac{\displaystyle\int_a^1u\,e^{-uq(X)/\lambda}\dd u}
 {\displaystyle\int_a^1e^{-uq(X)/\lambda}\dd u}
 \in(a,1).
\end{equation}
Then $q=q_\lambda$ implies $\tau_q=\tau_\lambda$.  Applying
Theorem~\ref{thm:numerical-convergence} with these two choices gives
\begin{equation}\label{eq:record-exact-feedback-convergence}
 E(\Record_n^{\tau_\lambda})\longrightarrow E_*
 \qquad \mathbb P_X^{\tau_\lambda}\text{-almost surely},
\end{equation}
and
\begin{equation}\label{eq:record-approximate-feedback-convergence}
 E(\Record_n^{\tau_q})\longrightarrow E_*
 \qquad \mathbb P_X^{\tau_q}\text{-almost surely}.
\end{equation}
The same conclusion holds for any fixed Borel approximation of $\tau_\lambda$ clipped to
$[a,1]$.

\paragraph{Comparison with deterministic Lloyd iterations.}
For $X\in\hX$, define the deterministic fixed-mass Lloyd map associated with
\eqref{eq:E-intro} by
\[
 \mathcal T_{\mathrm L}(X)
 :=\bigl(c_1(X),\ldots,c_N(X)\bigr),
 \qquad
 X^{k+1}=\mathcal T_{\mathrm L}(X^k).
\]
By \eqref{eq:DE-C2-theorem-main}, its fixed points are precisely the critical points of $E$ on
$\hX$.  The classical centroidal-Voronoi interpretation goes back to
\cite{DuFaberGunzburger1999}.  For the classical Lloyd algorithm,
\cite{DuEmelianenkoJu2006} establishes local and global convergence results, while
\cite{EmelianenkoJuRand2008} proves weak global convergence to the set of nondegenerate
fixed-point quantizers.  In a related power-diagram setting, \cite{BourneRoper2015} proves energy
decrease and a conditional sequential-convergence theorem for a generalized Lloyd algorithm.
More recently, \cite{PortalesCazellesPauwels2025} proves sequential convergence to a single
accumulation point for two variants of Lloyd's method under an analyticity assumption on the target
density.  These results establish convergence toward fixed or critical configurations under their
respective assumptions, but do not imply convergence to a global minimizer.

This distinction is already visible in the elementary case
\[
 d=N=2,
 \qquad
 p_1=p_2=\frac12,
 \qquad
 \Omega=[0,1]^2,
 \qquad
 \rho\equiv1.
\]
Proposition~\ref{prop:square-critical-points} shows that, up to exchange of the two labels, the
critical configurations are exactly
\begin{align*}
 X^{\mathrm{bt}}
 &:=\left(\left(\frac12,\frac14\right),
          \left(\frac12,\frac34\right)\right),
 &
 X^{\mathrm{lr}}
 &:=\left(\left(\frac14,\frac12\right),
          \left(\frac34,\frac12\right)\right),\\
 X^{\nearrow}
 &:=\left(\left(\frac13,\frac13\right),
          \left(\frac23,\frac23\right)\right),
 &
 X^{\searrow}
 &:=\left(\left(\frac23,\frac13\right),
          \left(\frac13,\frac23\right)\right).
\end{align*}
The first two configurations, together with their label exchanges, are the global minimizers and satisfy
\[
 E(X^{\mathrm{bt}})=E(X^{\mathrm{lr}})=\frac{5}{96},
\]
whereas the two diagonal configurations are saddle points and satisfy
\[
 E(X^{\nearrow})=E(X^{\searrow})=\frac1{18}>\frac{5}{96}.
\]
%A deterministic Lloyd iteration initialized at either saddle remains there.  By contrast, the positive-temperature Euler chain of Theorem~\ref{thm:numerical-convergence} is not trapped at a single critical configuration, and its running record converges almost surely to the global minimum.  This illustrates the global-search advantage of persistent Gaussian exploration.  It does not, by itself, assert a finite-time superiority of the HJB feedback over other positive temperature rules.
A deterministic Lloyd iteration initialized at either saddle remains there.
This suboptimal trapping is not an artifact of the two-point construction. For larger N, the energy landscape admits many stable local critical points, and the Lloyd algorithm generically converges to a suboptimal configuration when initialized in their basin of attraction. We illustrate this in Figure~\ref{fig:n8_illustration} with \(N=8\) equally weighted sites on the unit square: the Lloyd algorithm converges rapidly to a high-energy local minimizer and stagnates, while the fixed-temperature Euler–Langevin scheme escapes the local basin and its running record reaches a significantly lower energy.
By contrast, the positive-temperature Euler chain of Theorem~\ref{thm:numerical-convergence} is not trapped at a single critical configuration, and its running record converges almost surely to the global minimum. This illustrates the global-search advantage of persistent Gaussian exploration. It does not, by itself, assert a finite-time superiority of the HJB feedback over other positive temperature rules.
\begin{figure}[htbp]
  \centering
  \includegraphics[width=0.95\textwidth]{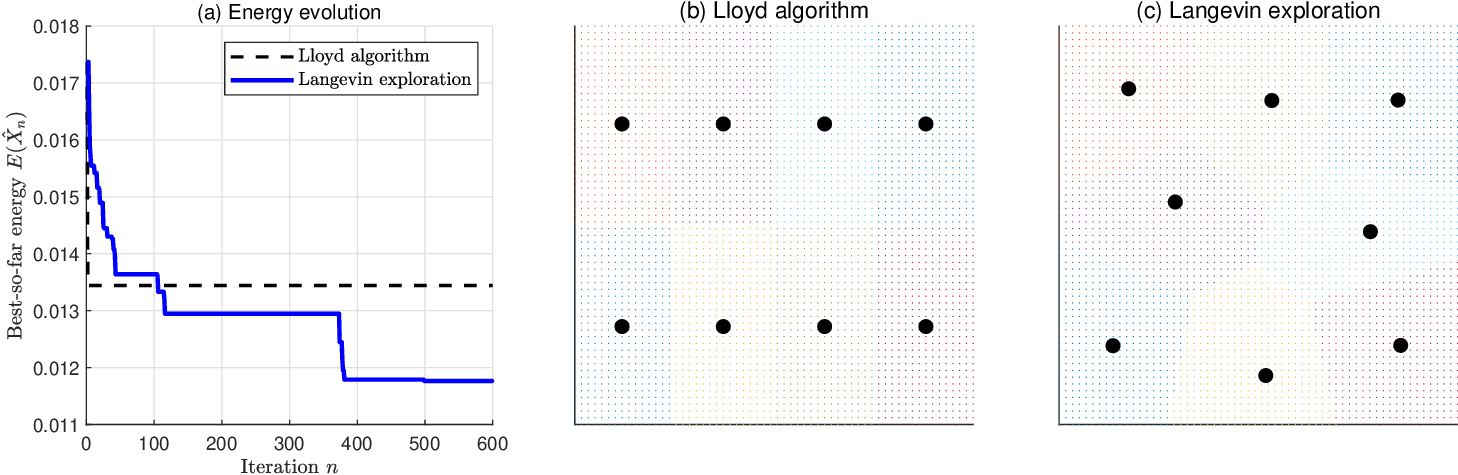}
  \caption{Numerical comparison for $N=8$ equally weighted sites
($p_i=1/8$) and the uniform target measure on
$\Omega=[0,1]^2$.  Both algorithms are initialized from the same
configuration obtained by adding a small perturbation to the regular
two-row lattice.  (a) Best-so-far energy
$\min_{0\le k\le n}E(X_k)$ as a function of the iteration number:
weighted Lloyd algorithm (black dashed line) and the fixed-temperature
Euler--Langevin scheme (blue solid line), with
$\eta=0.4$ and $\tau=5\times10^{-4}$.  The Lloyd iterates return to
the two-row stationary configuration and their energy stabilizes at a
non-global critical value.  By contrast, the Langevin dynamics escapes
the basin of this configuration, and its running-record energy decreases
in several steps until it reaches the numerically identified global
minimum.  (b) Limiting configuration produced by the Lloyd algorithm, namely the
regular two-row critical configuration.  (c) Best-record configuration produced by the Langevin scheme, providing
a numerical approximation of a global minimizer.  Its slight
irregularity reflects the fact that the record is an actual stochastic
iterate lying near the minimizing set, rather than an exactly stationary
configuration.
}
  \label{fig:n8_illustration}
\end{figure}

\begin{remark}\label{rem:main-results-separation}
Theorem~\ref{thm:HJB} and Theorem~\ref{thm:numerical-convergence} express different notions of
performance.  The former identifies the optimal relaxed feedback for the discounted criterion
$J_\lambda$.  The latter uses only dissipative confinement and uniformly positive Gaussian
exploration.  It neither proves that $\tau_\lambda$ minimizes expected hitting times or finite-time
record values, nor compares its convergence rate with constant temperatures or annealing schedules.
It also concerns homogeneous state-feedback chains and does not cover a temperature rule depending
on time or on the entire past trajectory without enlarging the state space.
\end{remark}

\subsection{Organization and notation}\label{subsec:organization-notation}

Section~\ref{sec:objective} proves the global estimates, the geometry of the minimizer set, and the
first- and second-order Laguerre formulas, followed by the explicit one-dimensional analysis and a
two-dimensional example exhibiting non-minimizing Lloyd fixed points.  Section~\ref{sec:control}
develops the relaxed control problem, controlled-diffusion estimates, canonical dynamic programming,
HJB regularity, and verification.  Section~\ref{sec:numerical} studies the Euler scheme,
finite-horizon consistency of the exact feedback, geometric ergodicity, and running-record convergence.
Appendix~\ref{app:parametric} verifies the joint semi-dual differentiation input.

For convenience, the most frequently used symbols are collected below.
\begin{center}
\renewcommand{\arraystretch}{1.15}
\begin{tabular}{@{}p{0.23\textwidth}p{0.69\textwidth}@{}}
$\Xc^N$, $\hX$, $\hOm$ & ambient configuration space, collision-free space, and collision-free configurations constrained to $\Omega^N$;\\
$\Coll$, $\Mc_*$ & collision set and set of global minimizers;\\
$U$, $\mathfrak U$ & zero-sum dual-weight space and relaxed temperature-action space $\mathcal P([a,1])$;\\
$\Phi^*$, $V_i$, $c_i$, $F_{ij}$ & balancing weights, balanced Laguerre cells, their barycenters, and their pairwise facets;\\
$\mathbb F^0$, $\mathbb F$ & raw canonical filtration and its usual augmentation;\\
$\tau_\lambda$, $\tau_q$, $\tau$ & exact HJB mean temperature, a fixed HJB approximation, and an arbitrary fixed Borel temperature rule;\\
$\Record_n^\tau$ & earliest best-so-far configuration up to time $n$.
\end{tabular}
\end{center}

\section{Regularity of the semi-discrete energy}\label{sec:objective}

This section proves Theorem~\ref{thm:C2}.  We first derive the global estimates which remain valid
at colliding configurations.  We then reduce the minimization problem to interior collision-free
configurations and perform the classical Laguerre-cell differentiation.  The final two subsections
give the explicit one-dimensional solution and a two-dimensional saddle-point example.

\subsection{Local Lipschitz continuity, semiconcavity, supergradients, and dissipativity}

\begin{proposition}\label{prop:lip}
The function $E:\Xc^N\to \R_+$ is continuous and locally Lipschitz.  More precisely, for
every $R>0$ there is $L_R<\infty$ such that
\begin{equation}\label{eq:local-lip}
 |E(X)-E(Z)|\le L_R|X-Z|,
 \qquad\text{for all } |X|\vee|Z|\le R.
\end{equation}
\end{proposition}

\begin{proof}
For $X,Z\in\Xc^N$, the coupling
$\sum_{i=1}^Np_i\delta_{(x_i,z_i)}$ gives
\begin{equation}\label{eq:atomic-coupling-bound}
 \Wtwo^2(\mu_X,\mu_Z)\le\sum_{i=1}^Np_i|x_i-z_i|^2\le|X-Z|^2.
\end{equation}
The triangle inequality therefore yields
\[
 \big|\Wtwo(\nu,\mu_X)-\Wtwo(\nu,\mu_Z)\big|\le|X-Z|.
\]
Moreover, when $|X|\le R$,
\[
 \Wtwo^2(\nu,\mu_X)
 \le 2\sum_{i=1}^Np_i|x_i|^2+2m_2(\nu)
 \le2R^2+2m_2(\nu).
\]
The identity $a^2-b^2=(a-b)(a+b)$ now proves \eqref{eq:local-lip}.  Continuity follows.
\end{proof}

We endow
$\Xc^N$ with the weighted quadratic form
\begin{equation}\label{eq:P-norm}
 |X|_P^2:=\sum_{i=1}^Np_i|x_i|^2,
 \qquad
 P:=\operatorname{diag}(p_1I_d,\ldots,p_NI_d).
\end{equation}

\begin{lemma}\label{lem:DC}
The function
\begin{equation}\label{eq:concave-part}
 X\longmapsto E(X)-\frac12|X|_P^2
\end{equation}
is concave on $\Xc^N$.  Equivalently,
$X\mapsto |X|_P^2-2E(X)$ is convex and, for $X,Z\in\Xc^N$ and $s\in[0,1]$,
\begin{equation}\label{eq:semiconcavity-ineq}
 E((1-s)X+sZ)
 \ge(1-s)E(X)+sE(Z)-\frac{s(1-s)}2|X-Z|_P^2.
\end{equation}
\end{lemma}

\begin{proof}
Expanding \eqref{eq:E-splitting} gives
\begin{equation}\label{eq:semiconcavity-representation}
 E(X)=\frac12|X|_P^2+\frac12m_2(\nu)
 +\inf_{\boldsymbol\nu\in\Ac_p(\nu)}
 \left(-\sum_{i=1}^Nx_i\cdot m_i(\boldsymbol\nu)\right).
\end{equation}
The last term is the infimum of affine functions of $X$, hence is concave.  Formula
\eqref{eq:semiconcavity-ineq} follows by expanding the weighted square.
\end{proof}

For later use, define the $P$-superdifferential of $E$ at $X$ by
\begin{equation}\label{eq:superdiff}
 D_P^+E(X):=\left\{g\in\Xc^N:
 E(Z)\le E(X)+g\cdot(Z-X)+\frac12|Z-X|_P^2
 \text{ for every }Z\in\Xc^N\right\}.
\end{equation}

\begin{proposition}\label{prop:supergradient-growth}
Let $X\in\Xc^N$ and let $\boldsymbol\nu=(\nu_i)$ be any optimal splitting in
\eqref{eq:E-splitting}.  Then
\begin{equation}\label{eq:supergradient-formula}
 g_i:=p_ix_i-m_i(\boldsymbol\nu),
 \qquad i=1,\ldots,N,
\end{equation}
defines an element $g\in D_P^+E(X)$.  Every vector $g$ obtained from an optimal splitting through
\eqref{eq:supergradient-formula} satisfies
\begin{align}
 X\cdot g
 &\ge \frac12|X|_P^2-\frac12m_2(\nu)
 \ge\frac{\pmin}{2}|X|^2-\frac12m_2(\nu),
 \label{eq:dissipativity-supergradient}\\
 |g|^2
 &\le2\pmax^2|X|^2+2\pmax m_2(\nu).
 \label{eq:growth-supergradient}
\end{align}
Moreover,
\begin{equation}\label{eq:E-growth}
 \frac14|X|_P^2-\frac12m_2(\nu)
 \le E(X)
 \le |X|_P^2+m_2(\nu).
\end{equation}
\end{proposition}

\begin{proof}
Use the optimal splitting for $X$ as a competitor at $Z$.  Expanding the squares yields
\[
 E(Z)\le E(X)+\sum_i\left(p_ix_i-m_i(\boldsymbol\nu)\right)\cdot(z_i-x_i)
 +\frac12|Z-X|_P^2,
\]
which proves \eqref{eq:supergradient-formula}.  By Cauchy--Schwarz,
\[
 |m_i(\boldsymbol\nu)|^2
 \le p_i\int|y|^2\nu_i(\dd y).
\]
Consequently,
\begin{align*}
 X\cdot g
 &=\sum_i p_i|x_i|^2-\sum_i x_i\cdot m_i(\boldsymbol\nu)\\
 &\ge\frac12\sum_i p_i|x_i|^2
 -\frac12\sum_i\frac{|m_i(\boldsymbol\nu)|^2}{p_i}
 \ge\frac12|X|_P^2-\frac12m_2(\nu),
\end{align*}
and \eqref{eq:growth-supergradient} follows similarly from
$|a-b|^2\le2|a|^2+2|b|^2$.  Finally,
$|x-y|^2\ge\frac12|x|^2-|y|^2$ gives the lower bound in \eqref{eq:E-growth}, while
$|x-y|^2\le2|x|^2+2|y|^2$ gives the upper bound.
\end{proof}

\begin{remark}\label{rem:bounded-gradient-Omega}
On the bounded set $\Omega^N$, every supergradient in Proposition~\ref{prop:supergradient-growth}
is uniformly bounded.  On the full space $\Xc^N$, the natural estimate is linear growth, not a
global bounded-gradient condition.  The lower bound
\eqref{eq:dissipativity-supergradient} is the key estimate for nonexplosion and ergodicity below.
\end{remark}

\subsection{Reduction of the minimization domain and proof of Theorem~\ref{thm:C2}}

The next proposition is the geometric reduction needed both for the deterministic minimization
problem and for the record-convergence argument in Section~\ref{sec:numerical}.
\begin{proposition}\label{prop:reduction}
Under Assumption~\ref{ass:geometry},
\begin{equation}\label{eq:reduction}
 \inf_{X\in\Xc^N}E(X)=\inf_{X\in\hOm}E(X).
\end{equation}
In fact, the global minimum is attained, and every global minimizer belongs to $\hX\cap(\Omega^\circ)^N\subset\hOm$.
Consequently, the minimizer set $\Mc_*$ defined in
\eqref{eq:minimizer-set} is a nonempty compact subset of $\hOm$.
\end{proposition}

\begin{proof}
Let $P_\Omega:\R^d\to\Omega$ be the Euclidean projection.  Convexity gives
$|P_\Omega(x)-y|\le|x-y|$ for every $y\in\Omega$.  Applying this inequality in
\eqref{eq:E-splitting} shows
\begin{equation}\label{eq:projection-lowers}
 E(P_\Omega x_1,\ldots,P_\Omega x_N)\le E(X).
\end{equation}
Thus it is enough to minimize over the compact set $\Omega^N$, and Proposition~\ref{prop:lip}
gives existence of a minimizer.

Let $X$ be a global minimizer and choose an optimal splitting
$\boldsymbol\nu=(\nu_i)\in\Ac_p(\nu)$ in \eqref{eq:E-splitting}.  Keeping this splitting fixed and
varying only $x_i$, the unique minimizer of
$x\mapsto\frac12\int|x-y|^2\nu_i(\dd y)$ is the splitting barycenter
\begin{equation}\label{eq:barycenter-splitting}
 \bar y_i(\boldsymbol\nu):=\frac{m_i(\boldsymbol\nu)}{p_i}
 =\frac1{p_i}\int y\,\nu_i(\dd y).
\end{equation}
Hence global optimality forces $x_i=\bar y_i(\boldsymbol\nu)$.  Since $\nu_i$ has positive mass
and is absolutely continuous with respect to Lebesgue measure, this barycenter lies in
$\Omega^\circ$.  Indeed, if $\bar y_i(\boldsymbol\nu)\in\partial\Omega$, a supporting hyperplane
there would force $\nu_i$ to be concentrated on that hyperplane, contrary to absolute continuity.
Thus $X\in(\Omega^\circ)^N$.

It remains to exclude collisions.  Suppose $x_i=x_j=:x$.  The preceding argument gives
\[
 \int y\,\nu_i(\dd y)=p_ix,
 \qquad
 \int y\,\nu_j(\dd y)=p_jx.
\]
Set $\eta:=\nu_i+\nu_j$.  This measure has mass $p_i+p_j$, barycenter $x$, and satisfies
$\eta\ll\Leb^d$ because $\eta\le\nu$.  Fix a unit vector $e$.  After a rotation of coordinates,
Fubini's theorem shows that the projection $e_\#\eta$ is absolutely continuous with respect to
one-dimensional Lebesgue measure.  Its distribution function is therefore continuous.  Since
$0<p_i<p_i+p_j$, there is $t\in\R$ such that
\[
 A:=\{y:e\cdot y<t\},
 \qquad
 \eta(A)=p_i,
 \qquad
 \eta(A^c)=p_j.
\]
Let $c_A$ and $c_{A^c}$ be the corresponding barycenters.  Their $e$-coordinates are strictly
ordered, hence $c_A\ne c_{A^c}$.  The variance decomposition gives
\begin{align*}
 \int|y-x|^2\eta(\dd y)
 &={}
 \int_A|y-c_A|^2\eta(\dd y)
 +\int_{A^c}|y-c_{A^c}|^2\eta(\dd y)\\
 &\quad
 +p_i|c_A-x|^2+p_j|c_{A^c}-x|^2.
\end{align*}
Replacing the two coincident sites by $c_A,c_{A^c}$ and the two pieces $\nu_i,\nu_j$ by
$\eta\restr A,\eta\restr A^c$ therefore strictly decreases the competitor cost in
\eqref{eq:E-splitting}, a contradiction.  Hence every minimizing configuration is collision free,
which proves \eqref{eq:reduction}.  Finally,
\[
 \Mc_*=E^{-1}(\{E_*\})\cap\Omega^N,
\]
so continuity of $E$ and compactness of $\Omega^N$ imply that $\Mc_*$ is compact.
\end{proof}

We now turn to differentiability on the collision-free set.  All cell notation introduced below is
used only for $X\in\hX$, where the balancing weights and Laguerre cells are well defined.

For $X\in\hX$ and $\Phi=(\phi_1,\ldots,\phi_N)\in\R^N$, define
\begin{equation}\label{eq:laguerre-main}
 V_i(X,\Phi):=\left\{y\in\Omega:
 \frac12|y-x_i|^2-\phi_i
 \le\frac12|y-x_j|^2-\phi_j\quad\text{for every }j\right\}.
\end{equation}
The semi-dual functional is
\begin{equation}\label{eq:semi-dual-main}
 \Kant(X,\Phi):=\sum_{i=1}^Np_i\phi_i+
 \int_\Omega\min_{1\le i\le N}
 \left\{\frac12|y-x_i|^2-\phi_i\right\}\rho(y)\dd y.
\end{equation}
Since $\Kant(X,\Phi+c\one)=\Kant(X,\Phi)$, we use the zero-mean normalization
\begin{equation}\label{eq:gauge-main}
 U:=\one^\perp=\left\{\Phi\in\R^N:\sum_{i=1}^N\phi_i=0\right\}.
\end{equation}
For every $X\in\hX$ there is a unique $\Phi^*(X)\in U$ such that
\begin{equation}\label{eq:balanced-main}
 \nu(V_i(X,\Phi^*(X)))=p_i,
 \qquad i=1,\ldots,N.
\end{equation}
Existence follows from standard semi-discrete duality; see, for example,
\cite[Chapter 5]{Villani2009}.  For uniqueness, the density bounds in
\eqref{eq:density-assumption-main} and the convexity of $\Omega$ imply the weighted
Poincar\'e--Wirtinger inequality required in
\cite[Theorems 1.1 and 1.4]{KitagawaMerigotThibert2019}.  Indeed, the usual $L^1$ Poincar\'e
inequality on a bounded convex body, together with the equivalence of $\nu$ and Lebesgue measure,
gives
\[
 \int_\Omega\left|f-\int_\Omega f\,\dd\nu\right|\dd\nu
 \le C_\Omega\int_\Omega|\nabla f|\dd\nu
\]
for every Lipschitz function $f$.
Hence the weighted Poincar\'e--Wirtinger condition of
\cite{KitagawaMerigotThibert2019} holds.  By
\cite[Theorem~1.4]{KitagawaMerigotThibert2019}, the semi-dual is
strictly concave on the subset of $U$ on which all Laguerre cells
have positive mass.  Since every balanced cell has mass $p_i>0$,
the normalized balancing vector is unique.

\begin{lemma}\label{lem:weight-continuity}
The normalized balancing map $X\mapsto\Phi^*(X)$ is continuous from $\hX$ to $U$.
\end{lemma}

\begin{proof}
Let $X^n\to X\in\hX$ and write $\Phi^n:=\Phi^*(X^n)$.  Since every balanced cell has mass
$p_i>0$, choose $y_i^n\in V_i(X^n,\Phi^n)$.  The cell inequalities at $y_i^n$ and $y_j^n$ give,
for every $i,j$,
\begin{align*}
 \frac12|y_i^n-x_i^n|^2-\frac12|y_i^n-x_j^n|^2
 &\le \phi_i^n-\phi_j^n\\
 &\le
 \frac12|y_j^n-x_i^n|^2-\frac12|y_j^n-x_j^n|^2.
\end{align*}
The sites $X^n$ remain bounded and $\Omega$ is compact, so all pairwise differences
$\phi_i^n-\phi_j^n$ are uniformly bounded.  The normalization $\sum_i\phi_i^n=0$ therefore makes
$(\Phi^n)$ bounded in $U$.

Consider a convergent subsequence, still denoted by $\Phi^n$, with limit $\Phi\in U$.  Outside the
finite union of the limiting competition hyperplanes, the minimizer in
\eqref{eq:laguerre-main} is unique and remains unchanged for all sufficiently large $n$.  Hence
\[
 \id_{V_i(X^n,\Phi^n)}(y)\longrightarrow\id_{V_i(X,\Phi)}(y)
 \quad\text{for Lebesgue-a.e. }y\in\Omega.
\]
Dominated convergence and \eqref{eq:balanced-main} yield
$\nu(V_i(X,\Phi))=p_i$ for every $i$.  By uniqueness of the normalized balancing weights,
$\Phi=\Phi^*(X)$.  Every convergent subsequence has the same limit, so the full sequence converges.
\end{proof}

We abbreviate
\begin{equation}\label{eq:optimal-cells-centroids}
 V_i(X):=V_i(X,\Phi^*(X)),
 \qquad
 c_i(X):=\frac1{p_i}\int_{V_i(X)}y\rho(y)\dd y.
\end{equation}

\begin{proposition}\label{prop:C1}
The function $E$ belongs to $C^1(\hX)$ and
\begin{equation}\label{eq:gradient-main}
 \nabla_{x_i}E(X)
 =\int_{V_i(X)}(x_i-y)\rho(y)\dd y
 =p_i\bigl(x_i-c_i(X)\bigr),
 \qquad i=1,\ldots,N.
\end{equation}
In particular, on $\hX$ the superdifferential in \eqref{eq:superdiff} is a singleton.
\end{proposition}

\begin{proof}
The boundaries of the Laguerre cells are contained in finitely many affine hyperplanes and are
$\nu$-null.  Hence the balanced cells induce the unique optimal splitting of $\nu$ up to null sets. The semi-dual identity
\begin{equation}\label{eq:E-dual-main}
 E(X)=\Kant(X,\Phi^*(X))
\end{equation}
and the optimality condition $D_\Phi\Kant(X,\Phi^*(X))=0$ give the envelope formula
\[
 DE(X)[\Xi]
 =\sum_{i=1}^N\int_{V_i(X)}(x_i-y)\cdot\xi_i\,\rho(y)\dd y.
\]
Continuity of this derivative under $X_n\to X\in\hX$ follows from
Lemma~\ref{lem:weight-continuity} and dominated convergence away from the limiting bisectors.
Equivalently, this is the first-order specialization of
\cite[Theorem 1 and Proposition 1.1]{deGournayKahnLebrat2019}.  This proves
\eqref{eq:gradient-main} and $C^1$ regularity.
\end{proof}

\begin{example}\label{ex:one-dimensional-collision}
Let $d=1$, $N=2$, $\Omega=[0,1]$, $\nu$ be Lebesgue measure, and
$p_1=p_2=1/2$.  A direct calculation gives
\begin{equation}\label{eq:collision-example}
 E(x_1,x_2)=\frac14(x_1^2+x_2^2)
 -\frac18\bigl(\min(x_1,x_2)+3\max(x_1,x_2)\bigr)+\frac16.
\end{equation}
Thus $E$ is smooth on the two ordering chambers but is not differentiable on the diagonal
$x_1=x_2$.  This shows why all classical Hessian statements must be formulated on $\hX$.
\end{example}

We first fix the notation used in the second-order formula.  We identify
$(\R^d)^N$ with $\R^{dN}$ and use the Euclidean scalar product
\[
 \Xi\cdot\Theta:=\sum_{i=1}^N\xi_i\cdot\theta_i,
 \qquad
 \Xi=(\xi_1,\ldots,\xi_N),\quad
 \Theta=(\theta_1,\ldots,\theta_N).
\]
We write $\HH^{d-1}$ for the standard $(d-1)$-dimensional Hausdorff measure,
normalized so that its restriction to every $(d-1)$-dimensional affine subspace agrees with the
usual $(d-1)$-dimensional Lebesgue measure.  If $F\subset\R^d$ is Borel, the restricted measure
$\HH^{d-1}\restr F$ is defined by
\begin{equation}\label{eq:restricted-Hausdorff-main}
 (\HH^{d-1}\restr F)(A):=\HH^{d-1}(A\cap F),
 \qquad A\subset\R^d\text{ Borel}.
\end{equation}
Thus, for every Borel function $g$ which is integrable on $F$,
\[
 \int_F g(y)\dd\HH^{d-1}(y)
 :=\int_{\R^d}g(y)\,(\HH^{d-1}\restr F)(\dd y)
 =\int_{\R^d}\id_F(y)g(y)\dd\HH^{d-1}(y).
\]
Accordingly, every integral over a facet $F_{ij}(X)$ below is a surface integral with respect to
$\HH^{d-1}$ restricted to that facet, and not a $d$-dimensional Lebesgue integral.  This is the
meaning of the sentence ``every integral over $F_{ij}(X)$ is taken with respect to
$\HH^{d-1}\restr F_{ij}(X)$.''

For $X\in\hX$ and $i\ne j$, define the balanced facet and the distance between the corresponding
sites by
\begin{equation}\label{eq:facets-main}
 F_{ij}(X):=\overline{V_i(X)}\cap\overline{V_j(X)},
 \qquad
 r_{ij}(X):=|x_i-x_j|>0.
\end{equation}
Notice that $F_{ij}(X)=F_{ji}(X)$ and $r_{ij}(X)=r_{ji}(X)$.  The set
$F_{ij}(X)$ may have zero $\HH^{d-1}$-measure; in that case every integral over this facet is
understood to be zero.  Set
\begin{equation}\label{eq:facet-weights-main}
 w_{ij}(X):=\int_{F_{ij}(X)}\frac{\rho(y)}{r_{ij}(X)}\dd\HH^{d-1}(y),
 \qquad i\ne j,
 \qquad
 w_{ii}(X):=0.
\end{equation}
Then $w_{ij}(X)=w_{ji}(X)\ge0$.  We introduce the linear map
\begin{equation}\label{eq:L-main}
 L_X:\R^N\longrightarrow\R^N,
 \qquad
 (L_X\Phi)_i:=\sum_{j\ne i}w_{ij}(X)(\Phi_i-\Phi_j),
 \quad i=1,\ldots,N.
\end{equation}
Equivalently, its matrix entries are
\begin{equation}\label{eq:L-matrix-main}
 (L_X)_{ii}=\sum_{j\ne i}w_{ij}(X),
 \qquad
 (L_X)_{ij}=-w_{ij}(X)\quad(i\ne j).
\end{equation}
The operator $L_X$ is the weighted graph Laplacian associated with the
balanced Laguerre tessellation.  It is also precisely the derivative,
with respect to the dual weights, of the vector of cell-mass
constraints; see \eqref{eq:DphiM-L} below.  For every
$\Phi\in\R^N$, symmetry of the weights gives
\begin{align}
 \sum_{i=1}^N(L_X\Phi)_i&=0,
 \label{eq:L-range-U-main}\\
 \Phi\cdot L_X\Phi
 &=\sum_{1\le i<j\le N}
 w_{ij}(X)(\Phi_i-\Phi_j)^2
 \ge0.
 \label{eq:L-quadratic-main}
\end{align}
Hence $L_X(\R^N)\subset U$ and, in particular, $L_X$ maps
$U$ into $U$.  Moreover, $L_X\one=0$.

We now prove that the constant vectors are the only elements of the
kernel.  Recall that
\begin{equation}\label{eq:kernel-definition-main}
 \ker L_X:=\{\Phi\in\R^N:L_X\Phi=0\},
\end{equation}
and set
\begin{equation}\label{eq:Rone-definition-main}
 \R\one:=\{c\one:c\in\R\},
 \qquad
 \one=(1,\ldots,1).
\end{equation}

For fixed $X\in\hX$, the gradient of the semi-dual with respect to the
weights is
\[
 D_\Phi\Kant(X,\Phi)
 =
 \bigl(p_i-\nu(V_i(X,\Phi))\bigr)_{i=1}^N.
\]
Proposition~\ref{prop:joint-C2-semidual}, together with
\eqref{eq:DphiM-app}, therefore gives
\begin{equation}\label{eq:weight-Hessian-Laplacian}
 D_{\Phi\Phi}^2\Kant(X,\Phi^*(X))
 =
 -L_X.
\end{equation}

We apply the strong-concavity estimate of
\cite[Theorem~5.1]{KitagawaMerigotThibert2019}.  To match the sign
convention of that reference, one sets $\psi=-\Phi$; the corresponding
Kantorovich functional is then exactly
$\Phi\mapsto\Kant(X,\Phi)$, and the change $\Phi\mapsto-\Phi$ leaves
the Hessian quadratic form unchanged.  For the quadratic cost, the
regularity, twist, and quasi-convexity assumptions of that theorem are
satisfied: the sites are distinct because $X\in\hX$, the competition
functions are affine, and the $c$-convexity condition reduces to the
ordinary convexity of $\Omega$.  The weighted
Poincar\'e--Wirtinger condition required there was established above.

At the balanced vector $\Phi^*(X)$, every cell has mass
\[
 \nu(V_i(X,\Phi^*(X)))=p_i\ge\pmin.
\]
Thus $\Phi^*(X)$ belongs to the positive-cell set of
\cite[Theorem~5.1]{KitagawaMerigotThibert2019} with, for instance,
the lower mass threshold $\varepsilon=\pmin/2$.  Consequently, there
exists a constant $\gamma_X>0$ such that
\begin{equation}\label{eq:L-coercive-on-U}
 \Psi\cdot L_X\Psi
 =
 -D_{\Phi\Phi}^2\Kant(X,\Phi^*(X))[\Psi,\Psi]
 \ge
 \gamma_X|\Psi|^2,
 \qquad
 \Psi\in U.
\end{equation}
The constant $\gamma_X$ may depend on the fixed configuration $X$;
only its positivity is needed here.

Let now $\Phi\in\ker L_X$, and write
\[
 \bar\phi:=\frac1N\sum_{i=1}^N\phi_i,
 \qquad
 \Psi:=\Phi-\bar\phi\,\one\in U.
\]
Since $L_X\one=0$, one has $L_X\Psi=L_X\Phi=0$.  Applying
\eqref{eq:L-coercive-on-U} gives $\Psi=0$, and hence
$\Phi=\bar\phi\,\one$.  The reverse inclusion follows from
$L_X\one=0$.  Therefore
\begin{equation}\label{eq:L-kernel-main}
 \ker L_X=\R\one.
\end{equation}

Finally, $L_X$ maps $U$ into $U$, and
\eqref{eq:L-coercive-on-U} shows that its restriction to $U$ is
injective.  Since $U$ is finite-dimensional, this restriction is
bijective.  Hence
\begin{equation}\label{eq:L-isomorphism-main}
 L_X|_U:U\longrightarrow U
\end{equation}
is an isomorphism.

For a direction $\Xi=(\xi_1,\ldots,\xi_N)\in(\R^d)^N$, define
$b_X[\Xi]=(b_{X,1}[\Xi],\ldots,b_{X,N}[\Xi])\in\R^N$ by
\begin{equation}\label{eq:bX-main}
 b_{X,i}[\Xi]:=\sum_{j\ne i}
 \int_{F_{ij}(X)}\frac{\rho(y)}{r_{ij}(X)}
 \left((y-x_i)\cdot\xi_i-(y-x_j)\cdot\xi_j\right)
 \dd\HH^{d-1}(y).
\end{equation}
The contribution of the oriented facet $(i,j)$ to $b_{X,i}[\Xi]$ is the opposite of its
contribution to $b_{X,j}[\Xi]$; consequently,
\begin{equation}\label{eq:bX-in-U-main}
 \sum_{i=1}^Nb_{X,i}[\Xi]=0,
 \qquad\text{and hence}\qquad b_X[\Xi]\in U.
\end{equation}
By \eqref{eq:L-isomorphism-main}, there is a unique vector
$\dot\Phi_X[\Xi]\in U$ satisfying
\begin{equation}\label{eq:linearized-weights-main}
 L_X\dot\Phi_X[\Xi]=-b_X[\Xi].
\end{equation}
Theorem~\ref{thm:C2} will show that
$\dot\Phi_X[\Xi]=D\Phi^*(X)[\Xi]$.  Finally, for $y\in F_{ij}(X)$, put
\begin{equation}\label{eq:Q-main}
 Q_{ij}^X[\Xi](y):=(y-x_i)\cdot\xi_i-(y-x_j)\cdot\xi_j
 +\dot\Phi_{X,i}[\Xi]-\dot\Phi_{X,j}[\Xi].
\end{equation}
Thus $Q_{ji}^X[\Xi]=-Q_{ij}^X[\Xi]$ on the common facet.

We also make explicit the differential notation.  If $E$ is Fr\'echet differentiable at
$X\in\hX$, its derivative is the linear functional
\begin{equation}\label{eq:DE-definition-main}
 DE(X):(\R^d)^N\longrightarrow\R,
 \qquad
 DE(X)[\Xi]=\sum_{i=1}^N\nabla_{x_i}E(X)\cdot\xi_i.
\end{equation}
Equivalently,
$DE(X)[\Xi]=\nabla E(X)\cdot\Xi$, where
\begin{equation}\label{eq:full-gradient-definition-main}
 \nabla E(X):=(\nabla_{x_1}E(X),\ldots,\nabla_{x_N}E(X))\in(\R^d)^N.
\end{equation}
If $DE$ is differentiable at $X$, the second derivative is the symmetric bilinear map
\begin{equation}\label{eq:D2E-definition-main}
 D^2E(X):(\R^d)^N\times(\R^d)^N\longrightarrow\R,
 \qquad
 D^2E(X)[\Xi,\Theta]:=D(DE)(X)[\Xi][\Theta].
\end{equation}
We write $D^2E(X)\preccurlyeq P$ when
\begin{equation}\label{eq:Hessian-order-definition-main}
 D^2E(X)[\Xi,\Xi]\le \Xi\cdot P\Xi
 =\sum_{i=1}^Np_i|\xi_i|^2
 \qquad\text{for every }\Xi\in(\R^d)^N.
\end{equation}

The derivative of the balancing weights is characterized by
\begin{equation}\label{eq:Dphi-C2-theorem-main}
 D\Phi^*(X)[\Xi]=\dot\Phi_X[\Xi],
\end{equation}
and the explicit facet representation announced in Theorem~\ref{thm:C2} is
\begin{equation}\label{eq:hessian-facet-main}
 D^2E(X)[\Xi,\Theta]
 =\sum_{i=1}^Np_i\,\xi_i\cdot\theta_i
 -\sum_{1\le i<j\le N}
 \int_{F_{ij}(X)}\frac{\rho(y)}{r_{ij}(X)}
 Q_{ij}^X[\Xi](y)Q_{ij}^X[\Theta](y)
 \dd\HH^{d-1}(y).
\end{equation}
In particular, the nonnegative bilinear form $\mathcal Q_X$ in
Theorem~\ref{thm:C2} is the second term on the right-hand side of
\eqref{eq:hessian-facet-main}.

\begin{remark}[Symmetry and gauge invariance]\label{rem:hessian-invariances}
Formula~\eqref{eq:hessian-facet-main} is manifestly symmetric in $(\Xi,\Theta)$.  It is also
independent of the additive dual gauge: replacing the derivative of the balancing weights by
$\dot\Phi_X[\Xi]+c\one$ leaves every difference
$\dot\Phi_{X,i}[\Xi]-\dot\Phi_{X,j}[\Xi]$, and hence every $Q_{ij}^X[\Xi]$, unchanged.  No
common-translation invariance should be expected, because the target measure $\nu$ and its support
remain fixed when all sites are translated.
\end{remark}

\begin{proof}[Proof of Theorem~\ref{thm:C2}]
We proceed in five steps.

\medskip
\noindent\textit{Step 1: the mass and moment maps and their first variations.}
For $(X,\Phi)\in\hX\times U$, define
\begin{equation}\label{eq:mass-moment-main}
 M_i(X,\Phi):=\int_{V_i(X,\Phi)}\rho(y)\dd y-p_i,
 \qquad
 B_i(X,\Phi):=\int_{V_i(X,\Phi)}y\rho(y)\dd y.
\end{equation}
Since the cells form a partition up to null sets and $\sum_i p_i=1$, one has
$\sum_iM_i(X,\Phi)=0$; thus $M(X,\Phi)\in U$.  Every balanced cell at
$(X,\Phi^*(X))$ has mass $p_i>0$.  Proposition~\ref{prop:joint-C2-semidual} shows that, in a
neighborhood of each such point, the semi-dual $\Kant$ is jointly $C^2$ and the maps $M$ and $B$
are jointly $C^1$.

For clarity, we recall the moving-cell formula used below.  Given directions
$\Xi=(\xi_i)\in(\R^d)^N$ and $\Psi=(\Psi_i)\in U$, set
\[
 h_{ij}(X,\Phi;y)
 :=\frac12|y-x_i|^2-\Phi_i-\frac12|y-x_j|^2+\Phi_j.
\]
The cell $V_i(X,\Phi)$ is the intersection of the half-spaces $h_{ij}\le0$.  Along a parameter
curve with velocity $(\Xi,\Psi)$, the normal velocity of the facet $h_{ij}=0$, measured in the
outward normal direction of $V_i$, is
\begin{equation}\label{eq:normal-velocity-main}
 \frac{(y-x_i)\cdot\xi_i-(y-x_j)\cdot\xi_j+\Psi_i-\Psi_j}
 {|x_i-x_j|}.
\end{equation}
The outer boundary $\partial\Omega$ is fixed, and intersections of two or more competition
hyperplanes have zero $\HH^{d-1}$-measure.  If $\Oc$ is any open neighborhood of $\Omega$, the moving-domain formula gives, for
$q\in C^1(\Oc;\R^k)$,
\begin{align}
 &D\left(\int_{V_i(X,\Phi)}q(y)\rho(y)\dd y\right)[\Xi,\Psi]\notag\\
 &\quad=\sum_{j\ne i}\int_{F_{ij}(X,\Phi)}q(y)\rho(y)
 \frac{(y-x_i)\cdot\xi_i-(y-x_j)\cdot\xi_j+\Psi_i-\Psi_j}
 {|x_i-x_j|}\dd\HH^{d-1}(y).
 \label{eq:moving-cell-main-proof}
\end{align}
The differentiability and continuity required to justify this formula are proved in
Proposition~\ref{prop:joint-C2-semidual}; formula
\eqref{eq:moving-cell-main-proof} is also the Euclidean specialization of
\cite[Theorem~1 and Proposition~1.1]{deGournayKahnLebrat2019}.

\medskip
\noindent\textit{Step 2: differentiability of the balancing weights.}
Taking $q=1$, $\Xi=0$, and $\Psi\in U$ in
\eqref{eq:moving-cell-main-proof} yields
\begin{align}
 D_\Phi M_i(X,\Phi^*(X))[\Psi]
 &=\sum_{j\ne i}\int_{F_{ij}(X)}\frac{\rho(y)}{r_{ij}(X)}
 (\Psi_i-\Psi_j)\dd\HH^{d-1}(y)\notag\\
 &=(L_X\Psi)_i.
 \label{eq:DphiM-L}
\end{align}
Thus
$D_\Phi M(X,\Phi^*(X))=L_X|_U:U\to U$, which is invertible by
\eqref{eq:L-isomorphism-main}.  The implicit-function theorem applied to
\[
 M(X,\Phi^*(X))=0
\]
therefore gives $\Phi^*\in C^1(\hX;U)$.  Differentiating this identity in the direction $\Xi$ gives
\begin{equation}\label{eq:implicit-mass-derivative-main}
 D_XM(X,\Phi^*(X))[\Xi]+L_XD\Phi^*(X)[\Xi]=0.
\end{equation}
Taking $q=1$ and $\Psi=0$ in \eqref{eq:moving-cell-main-proof} shows that the first term in
\eqref{eq:implicit-mass-derivative-main} is precisely $b_X[\Xi]$.  Uniqueness in
\eqref{eq:linearized-weights-main} consequently gives
\[
 D\Phi^*(X)[\Xi]=\dot\Phi_X[\Xi],
\]
which proves \eqref{eq:Dphi-C2-theorem-main}.

\medskip
\noindent\textit{Step 3: differentiability of the centroids and of the gradient.}
Taking $q(y)=y$ and
$\Psi=D\Phi^*(X)[\Xi]=\dot\Phi_X[\Xi]$ in
\eqref{eq:moving-cell-main-proof} gives
\begin{equation}\label{eq:Dcentroid-main}
 Dc_i(X)[\Xi]
 =\frac1{p_i}\sum_{j\ne i}
 \int_{F_{ij}(X)}y\rho(y)
 \frac{Q_{ij}^X[\Xi](y)}{r_{ij}(X)}
 \dd\HH^{d-1}(y).
\end{equation}
The map $c_i$ is therefore $C^1$ on $\hX$.  Proposition~\ref{prop:C1} gives
\[
 \nabla_{x_i}E(X)=p_i(x_i-c_i(X)),
\]
so the full gradient map $\nabla E:\hX\to(\R^d)^N$ is $C^1$.  Hence
$E\in C^2(\hX)$, and the componentwise gradient formula is exactly
\eqref{eq:DE-C2-theorem-main}.  Differentiating it gives
\begin{equation}\label{eq:Hessian-centroid-main}
 D^2E(X)[\Xi,\Theta]
 =\sum_{i=1}^Np_i\xi_i\cdot\theta_i
 -\sum_{i=1}^Np_iDc_i(X)[\Xi]\cdot\theta_i.
\end{equation}

\medskip
\noindent\textit{Step 4: symmetrization of the boundary terms.}
For brevity, write $Q_{ij}[\Xi]=Q_{ij}^X[\Xi]$.  Substituting
\eqref{eq:Dcentroid-main} into the last term of
\eqref{eq:Hessian-centroid-main} and pairing the two orientations of each facet gives
\begin{align}
 \sum_{i=1}^Np_iDc_i(X)[\Xi]\cdot\theta_i
 &=\sum_{1\le i<j\le N}\int_{F_{ij}(X)}\frac{\rho(y)}{r_{ij}(X)}
 Q_{ij}[\Xi](y)\,y\cdot(\theta_i-\theta_j)
 \dd\HH^{d-1}(y).
 \label{eq:hessian-paired-main}
\end{align}
Indeed, the $(j,i)$ contribution equals the negative of the $(i,j)$ contribution because
$Q_{ji}[\Xi]=-Q_{ij}[\Xi]$.

Differentiating each balanced mass constraint
$M_i(X,\Phi^*(X))=0$ in the direction $\Xi$ yields
\begin{equation}\label{eq:linearized-mass-zero-main}
 \sum_{j\ne i}\int_{F_{ij}(X)}
 \frac{\rho(y)}{r_{ij}(X)}Q_{ij}[\Xi](y)\dd\HH^{d-1}(y)=0,
 \qquad i=1,\ldots,N.
\end{equation}
Multiply the $i$th identity by
$x_i\cdot\theta_i-\dot\Phi_{X,i}[\Theta]$ and sum over $i$.  Pairing the orientations gives
\begin{align}
 0={}&\sum_{1\le i<j\le N}\int_{F_{ij}(X)}\frac{\rho(y)}{r_{ij}(X)}Q_{ij}[\Xi](y)
 \Bigl[x_i\cdot\theta_i-x_j\cdot\theta_j
       -\dot\Phi_{X,i}[\Theta]+\dot\Phi_{X,j}[\Theta]\Bigr]
 \dd\HH^{d-1}(y).
 \label{eq:mass-cancellation-hessian-main}
\end{align}
Subtracting \eqref{eq:mass-cancellation-hessian-main} from
\eqref{eq:hessian-paired-main} replaces the second factor by
\begin{align*}
 &y\cdot(\theta_i-\theta_j)
 -x_i\cdot\theta_i+x_j\cdot\theta_j
 +\dot\Phi_{X,i}[\Theta]-\dot\Phi_{X,j}[\Theta]\\
 &\qquad=(y-x_i)\cdot\theta_i-(y-x_j)\cdot\theta_j
 +\dot\Phi_{X,i}[\Theta]-\dot\Phi_{X,j}[\Theta]\\
 &\qquad=Q_{ij}^X[\Theta](y).
\end{align*}
Substitution into \eqref{eq:Hessian-centroid-main} proves
\eqref{eq:hessian-facet-main}.

\medskip
\noindent\textit{Step 5: the upper bound and compact-set estimates.}
Taking $\Theta=\Xi$ in \eqref{eq:hessian-facet-main}, every facet term in the second sum is
nonnegative.  Hence
\[
 D^2E(X)[\Xi,\Xi]
 \le\sum_{i=1}^Np_i|\xi_i|^2=\Xi\cdot P\Xi,
\]
which is precisely \eqref{eq:hessian-upper-P}.  Finally, the joint $C^1$ regularity of $M$ and $B$,
the continuity of $X\mapsto L_X|_U$, and continuity of matrix inversion on the set of invertible
linear maps show that $D\Phi^*$ and $D^2E$ are continuous on $\hX$.  A continuous map is bounded and
uniformly continuous on every compact subset $K\Subset\hX$, completing the proof.
\end{proof}

\begin{remark}\label{rem:known-C2}
Theorem~\ref{thm:C2} uses only the joint $C^2$ result explicitly proved in
\cite{deGournayKahnLebrat2019}.  A stronger joint $C^{2,\alpha}$ statement would require a separate
parameter-uniform H\"older analysis for rotating and translating Laguerre facets; it is not needed
for the feedback construction below.
\end{remark}

\subsection{The explicit one-dimensional problem}\label{sec:one-dimensional}

The following result is independent of the stochastic-control arguments.  It is included because,
in one dimension, monotone transport reduces the nonconvex problem to finitely many strictly convex
ordering chambers.

In this subsection only, let $d=1$, let $\Omega=[\ell,r]$, and let
$F(y):=\nu((-\infty,y])$.  Under Assumption~\ref{ass:geometry}, $F$ is continuous and strictly
increasing on $[\ell,r]$, and we write $F^{-1}$ for its quantile function.

For a permutation $\sigma$ of $\{1,\ldots,N\}$, consider the ordering chamber
\begin{equation}\label{eq:ordering-chamber}
 \Dc_\sigma:=\{X\in\R^N:x_{\sigma(1)}<\cdots<x_{\sigma(N)}\}.
\end{equation}
Set
\begin{equation}\label{eq:cumulative-masses-one-d}
 s_0^\sigma:=0,
 \qquad
 s_k^\sigma:=\sum_{\ell=1}^kp_{\sigma(\ell)},
 \qquad
 q_k^\sigma:=F^{-1}(s_k^\sigma),
 \quad k=0,\ldots,N,
\end{equation}
where $q_0^\sigma=\ell$ and $q_N^\sigma=r$.  Define the quantile intervals and their barycenters by
\begin{equation}\label{eq:quantile-cells-one-d}
 I_k^\sigma:=(q_{k-1}^\sigma,q_k^\sigma],
 \qquad
 \bar x_{\sigma(k)}^\sigma
 :=\frac1{p_{\sigma(k)}}\int_{I_k^\sigma}y\nu(\dd y).
\end{equation}

\begin{theorem}\label{thm:one-dimensional}
For every permutation $\sigma$ and every $X\in\Dc_\sigma$,
\begin{equation}\label{eq:E-one-dimensional}
 E(X)=\frac12\sum_{k=1}^N
 \int_{I_k^\sigma}|y-x_{\sigma(k)}|^2\nu(\dd y).
\end{equation}
The restriction of $E$ to $\Dc_\sigma$ extends to a strictly convex quadratic function of $X$, and
its unique minimizer is $\bar X^\sigma=(\bar x_1^\sigma,\ldots,\bar x_N^\sigma)$ defined in
\eqref{eq:quantile-cells-one-d}.  Consequently,
\begin{equation}\label{eq:global-one-d}
 E_*=\min_{\sigma\in\mathfrak S_N}E(\bar X^\sigma),
 \qquad
 \Mc_*=\left\{\bar X^\sigma:\sigma\in\mathfrak S_N,
 E(\bar X^\sigma)=E_*\right\}.
\end{equation}
If the minimizing ordering is unique, then the labeled global minimizer is unique.  For equal
weights the minimizer is unique only modulo permutations of labels, although the optimal unlabeled
atomic measure is unique.
\end{theorem}

\begin{proof}
For the quadratic cost in one dimension, every optimal transport plan is monotone.  Once the sites
are ordered according to $\sigma$, the atom at $x_{\sigma(k)}$ must therefore receive the quantile
interval of mass $p_{\sigma(k)}$, namely $I_k^\sigma$.  This proves
\eqref{eq:E-one-dimensional}.  Expanding around the barycenter gives
\begin{align}\label{eq:variance-one-d}
 \int_{I_k^\sigma}|y-x_{\sigma(k)}|^2\nu(\dd y)
 &={}
 \int_{I_k^\sigma}|y-\bar x_{\sigma(k)}^\sigma|^2\nu(\dd y)
 +p_{\sigma(k)}|x_{\sigma(k)}-\bar x_{\sigma(k)}^\sigma|^2.
\end{align}
Because $\rho>0$ and $\nu(I_k^\sigma)=p_{\sigma(k)}>0$, the barycenter of each interval lies
strictly between its endpoints:
\[
 q_{k-1}^\sigma<\bar x_{\sigma(k)}^\sigma<q_k^\sigma.
\]
Hence $\bar X^\sigma\in\Dc_\sigma$.  The Hessian in the chamber is
$\operatorname{diag}(p_1,\ldots,p_N)$ and the unique chamber minimizer is
$\bar X^\sigma$.  Since there are finitely many orderings and every collision can be strictly
improved by Proposition~\ref{prop:reduction}, taking the minimum over $\sigma$ proves the first
identity in \eqref{eq:global-one-d}.  Every global minimizer is collision free and therefore lies
in one ordering chamber, where it must equal the corresponding $\bar X^\sigma$; the second identity
follows.
\end{proof}

\begin{remark}[Consistency with the facet formula]\label{rem:one-dimensional-Hessian}
Inside a fixed ordering chamber, the prescribed masses fix the boundaries of the balanced cells at
the quantiles $q_k^\sigma$, independently of the site locations.  Their first-order normal velocities
therefore vanish.  Equivalently, $Q_{ij}^X[\Xi]=0$ on every nonempty facet, and
\eqref{eq:hessian-facet-main} reduces to
\[
 D^2E(X)=P,
\]
in agreement with the quadratic decomposition \eqref{eq:variance-one-d}.
\end{remark}

\subsection{A two-dimensional saddle-point example}

The next explicit example illustrates why convergence of a deterministic Lloyd iteration to a
critical point is not a global-optimization result.

\begin{proposition}[Critical points on the unit square]
\label{prop:square-critical-points}
Let $d=N=2$, let $p_1=p_2=1/2$, and let $\nu$ be the uniform probability measure on
$\Omega=[0,1]^2$.  Up to exchange of the two labels, the critical points of $E$ on $\hX$ are
exactly
\begin{align}
 X^{\mathrm{bt}}
 &:=\left(\left(\frac12,\frac14\right),
          \left(\frac12,\frac34\right)\right),
 &
 X^{\mathrm{lr}}
 &:=\left(\left(\frac14,\frac12\right),
          \left(\frac34,\frac12\right)\right),
 \label{eq:square-axis-critical-points}\\
 X^{\nearrow}
 &:=\left(\left(\frac13,\frac13\right),
          \left(\frac23,\frac23\right)\right),
 &
 X^{\searrow}
 &:=\left(\left(\frac23,\frac13\right),
          \left(\frac13,\frac23\right)\right).
 \label{eq:square-diagonal-critical-points}
\end{align}
Moreover,
\begin{equation}\label{eq:square-critical-energies}
 E(X^{\mathrm{bt}})=E(X^{\mathrm{lr}})=\frac{5}{96},
 \qquad
 E(X^{\nearrow})=E(X^{\searrow})=\frac1{18}.
\end{equation}
Consequently, $X^{\mathrm{bt}}$ and $X^{\mathrm{lr}}$, together with their label exchanges, are
exactly the global minimizers.  The two diagonal configurations and their label exchanges are
saddle points.
\end{proposition}

\begin{proof}
Let $c=(1/2,1/2)$ be the center of the square.  At a critical point,
Proposition~\ref{prop:C1} gives $x_i=c_i(X)$ for $i=1,2$.  The two balanced Laguerre cells are
separated by an affine line and each has area $1/2$.  For a fixed normal direction, the line cutting
the centrally symmetric square into two sets of equal area is unique and passes through $c$.

Up to exchange of the labels and the dihedral symmetries of the square, we may therefore assume
that the normal is proportional to $(1,t)$ for some $t\in[0,1]$.  In coordinates centered at $c$,
consider the positive half-cell
\[
 H_t
 :=
 \left\{(u,v)\in\left[-\frac12,\frac12\right]^2:
 u+tv\ge0\right\}.
\]
It has area $1/2$, and its centroid relative to $c$ is
\begin{align}
 m(t)
 &:=2\int_{-1/2}^{1/2}\int_{-tv}^{1/2}(u,v)\,\dd u\,\dd v
 \notag\\
 &=\left(\frac{3-t^2}{12},\frac{t}{6}\right).
 \label{eq:half-square-centroid}
\end{align}
The two cell centroids are $c-m(t)$ and $c+m(t)$.  At a critical configuration, the normal of the
Laguerre interface, which is parallel to $x_2-x_1=2m(t)$, must also be parallel to $(1,t)$.  Hence
\[
 0
 =\det\!\left((1,t),m(t)\right)
 =\frac{t}{6}-\frac{t(3-t^2)}{12}
 =\frac{t(t^2-1)}{12}.
\]
Thus $t=0$ or $t=1$.  Applying the symmetries of the square gives precisely the four configurations
in \eqref{eq:square-axis-critical-points}--\eqref{eq:square-diagonal-critical-points}.

For an axis-aligned partition, symmetry reduces the energy to the second moment of one half-square:
\begin{align*}
 E(X^{\mathrm{lr}})
 &=\int_0^{1/2}\int_0^1
 \left[\left(u-\frac14\right)^2+
       \left(v-\frac12\right)^2\right]\dd v\,\dd u
 =\frac{5}{96}.
\end{align*}
The same value holds for $X^{\mathrm{bt}}$.  For a diagonal partition,
\begin{align*}
 E(X^{\searrow})
 &=\int_0^1\int_0^u
 \left[\left(u-\frac23\right)^2+
       \left(v-\frac13\right)^2\right]\dd v\,\dd u
 =\frac1{18},
\end{align*}
and the same value holds for $X^{\nearrow}$.  Since every global minimizer is collision free by
Proposition~\ref{prop:reduction} and hence is a critical point by Proposition~\ref{prop:C1}, the
energy comparison proves the assertion concerning the global minimizers.

It remains to classify the diagonal critical points.  Consider
$X=X^{\nearrow}$ and the directions
\[
 \Xi:=\bigl((1,-1),(-1,1)\bigr),
 \qquad
 \Theta:=\bigl((1,0),(1,0)\bigr).
\]
The common facet is parametrized by $y(s)=(s,1-s)$, $0\le s\le1$, and
$r_{12}(X)=\sqrt2/3$.  For the direction $\Xi$, the linearized mass term has zero facet average,
so $\dot\Phi_X[\Xi]=0$, and
\[
 Q_{12}^X[\Xi](y(s))=4\left(s-\frac12\right).
\]
Using \eqref{eq:hessian-facet-main} and
$r_{12}(X)^{-1}\dd\HH^1=3\,\dd s$ on the facet, we obtain
\[
 D^2E(X)[\Xi,\Xi]
 =2-3\int_0^1 16\left(s-\frac12\right)^2\dd s
 =-2<0.
\]
For the common-translation direction $\Theta$, the linearized weight correction makes
$Q_{12}^X[\Theta]$ vanish identically, and therefore
\[
 D^2E(X)[\Theta,\Theta]=1>0.
\]
Thus $X^{\nearrow}$ is a saddle point.  The same conclusion holds for $X^{\searrow}$ by symmetry.
\end{proof}

\section{Exploratory temperature control}\label{sec:control}
This section proves Theorem~\ref{thm:HJB}.  Throughout, Assumption~\ref{ass:geometry} holds and
\begin{equation}\label{eq:d-ge-2}
 d\ge2,
 \qquad
 m:=dN.
\end{equation}
The state space is the connected open set $\hX\subset\R^m$.  We write
\begin{equation}\label{eq:separation-control}
 \mathfrak d(X):=\min_{1\le i<j\le N}|x_i-x_j|,
 \qquad X\in\hX,
\end{equation}
so that the continuous extension of $\mathfrak d$ to $\Xc^N$ vanishes precisely on the collision
set $\Coll$.
The entropy-regularized relaxed-control formulation follows the exploratory-control framework of
\cite{WangZariphopoulouZhou2020,GaoXuZhou2022}; see also
\cite{TangZhangZhou2022} for exploratory HJB equations.  The points requiring additional work here
are the punctured state space $\hX$, the merely local Lipschitz regularity of the drift, collision
avoidance, and the regularity bootstrap needed to construct a locally Lipschitz feedback.
\subsection{Relaxed temperature controls and the canonical formulation}\label{subsec:relaxed-controls}
Fix a discount factor $\vartheta>0$, an entropy parameter $\lambda>0$, and the nondegenerate compact
temperature interval
\begin{equation}\label{eq:temperature-range}
 \Uc:=[a,1],
 \qquad 0<a<1.
\end{equation}
Let
\begin{equation}\label{eq:relaxed-action-space}
 \mathfrak U:=\mathcal P([a,1])
\end{equation}
be endowed with the topology of weak convergence.  This is a compact Polish space.  For
$\varpi\in\mathfrak U$, define
\begin{align}
 \bar u(\varpi)&:=\int_a^1u\,\varpi(\dd u)\in[a,1],
 \label{eq:mean-action}\\
 \Ent_{[a,1]}(\varpi)&:=
 \begin{cases}
  \displaystyle\int_a^1 f(u)\log f(u)\dd u,
       &\varpi(\dd u)=f(u)\dd u,\\[1mm]
  +\infty,&\varpi\not\ll\dd u.
 \end{cases}
 \label{eq:entropy-action}
\end{align}
Thus $\Ent_{[a,1]}$ is entropy relative to Lebesgue measure on $[a,1]$, with the value $+\infty$
assigned to singular laws.  It is lower semicontinuous, hence Borel measurable, on $\mathfrak U$.
Jensen's inequality gives
\begin{equation}\label{eq:entropy-lower}
 \Ent_{[a,1]}(\varpi)\ge-\log(1-a)>0.
\end{equation}
We work on the canonical Wiener space
\begin{equation}\label{eq:canonical-Wiener-space}
 \Omega_W:=\{\omega\in C([0,\infty);\R^m):\omega(0)=0\},
 \qquad
 W_t(\omega):=\omega(t),
 \qquad
 \Fc_t^0:=\sigma(W_s:0\le s\le t),
\end{equation}
endowed with the topology of locally uniform convergence and Wiener measure $\P$.  Write
$\mathbb F^0=(\Fc_t^0)_{t\ge0}$ for the raw canonical filtration and $\mathbb F=(\Fc_t)_{t\ge0}$
for its usual augmentation.  An admissible relaxed control is an actual, rather than merely
completed-equivalence-class, $\mathfrak U$-valued process
\begin{equation}\label{eq:admissible-control-class}
 \pi:[0,\infty)\times\Omega_W\longrightarrow\mathfrak U
\end{equation}
which is progressively measurable for $\mathbb F^0$.  The stochastic integrals are taken with
respect to the augmented filtration $\mathbb F$.  Proposition~\ref{prop:controlled-dynamics} below
shows that the corresponding state process can nevertheless be chosen continuous and raw adapted;
this is the version used in all stopping-time and concatenation arguments.
For an admissible control $\pi$, put
\begin{equation}\label{eq:mean-temperature}
 \bar u_t^\pi:=\bar u(\pi_t)\in[a,1],
 \qquad
 \sigma_t^\pi:=\sqrt{2\bar u_t^\pi}\in[\sqrt{2a},\sqrt2].
\end{equation}
Starting from $X\in\hX$, the controlled state equation is
\begin{equation}\label{eq:controlled-SDE}
 \dd X_t^\pi=-\nabla E(X_t^\pi)\dd t+\sigma_t^\pi\dd W_t,
 \qquad X_0^\pi=X.
\end{equation}
Here and below, a scalar diffusion coefficient multiplies the identity matrix of $\R^m$.
Define the extended running cost
\begin{equation}\label{eq:running-cost-notation}
 \ell_\lambda(X,\varpi):=E(X)+\lambda\Ent_{[a,1]}(\varpi)\in[0,+\infty]
\end{equation}
and the discounted performance criterion
\begin{equation}\label{eq:control-cost}
 J_\lambda(X;\pi)
 :=\E_X\left[\int_0^\infty e^{-\vartheta t}
 \ell_\lambda(X_t^\pi,\pi_t)\dd t\right]\in[0,+\infty].
\end{equation}
The value function is
\begin{equation}\label{eq:value-function}
 v_\lambda(X):=\inf_\pi J_\lambda(X;\pi),
\end{equation}
where the infimum is over all admissible raw-progressive relaxed controls.  The constant control
given by the uniform probability law on $[a,1]$ has finite entropy and, by
Proposition~\ref{prop:controlled-dynamics}, finite cost.  In particular, $v_\lambda$ is finite from
above.
For $q\in\R$, set
\begin{equation}\label{eq:partition-H}
 Z_\lambda(q):=\int_a^1e^{-uq/\lambda}\dd u,
 \qquad
 \pi_q^\lambda(u):=Z_\lambda(q)^{-1}e^{-uq/\lambda},
 \qquad
 \varpi_q^\lambda(\dd u):=\pi_q^\lambda(u)\dd u,
\end{equation}
and define the scalar Hamiltonian
\begin{equation}\label{eq:Hamiltonian-control}
 \mathscr H_\lambda(q):=-\lambda\log Z_\lambda(q).
\end{equation}
Differentiation under the integral sign gives
\begin{align}
 \mathscr H_\lambda'(q)
 &=\int_a^1u\,\varpi_q^\lambda(\dd u)\in(a,1),
 \label{eq:H-prime}\\
 \mathscr H_\lambda''(q)
 &=-\frac1\lambda\operatorname{Var}_{\varpi_q^\lambda}(u)<0.
 \label{eq:H-second}
\end{align}
The strict inequality follows because $0<a<1$ and $\varpi_q^\lambda$ has a strictly positive
density on the nondegenerate interval $[a,1]$.  Hence $\mathscr H_\lambda$ is smooth, strictly
increasing, and strictly concave.  Moreover, for $q_1\le q_2$,
\begin{equation}\label{eq:H-uniform-ellipticity}
 a(q_2-q_1)
 \le \mathscr H_\lambda(q_2)-\mathscr H_\lambda(q_1)
 \le q_2-q_1.
\end{equation}
In particular, $\mathscr H_\lambda(q)\to\pm\infty$ as $q\to\pm\infty$.
Consequently,
\begin{equation}\label{eq:H-diffeomorphism}
 \mathscr H_\lambda:\R\longrightarrow\R
\end{equation}
is a $C^\infty$ increasing diffeomorphism and
$\mathscr H_\lambda^{-1}$ is globally $a^{-1}$-Lipschitz.
For every $\varpi\in\mathfrak U$, the exact Gibbs identity is
\begin{equation}\label{eq:gibbs-relative-entropy}
 q\bar u(\varpi)+\lambda\Ent_{[a,1]}(\varpi)
 =\mathscr H_\lambda(q)
 +\lambda\Ent\!\left(\varpi\,\middle|\,\varpi_q^\lambda\right),
\end{equation}
with both sides equal to $+\infty$ when $\varpi$ is singular with respect to Lebesgue measure.
Therefore
\begin{equation}\label{eq:gibbs-variational}
 \inf_{\varpi\in\mathfrak U}
 \left\{q\bar u(\varpi)+\lambda\Ent_{[a,1]}(\varpi)\right\}
 =\mathscr H_\lambda(q),
\end{equation}
and the unique minimizer is $\varpi_q^\lambda$.
\paragraph{Interpretation of the parameters.}
The discount factor fixes the time scale: $e^{-\vartheta t}$ has half-life
$(\log2)/\vartheta$.  The parameter $\lambda$ controls the concentration of the Gibbs law, since
the ratio of its unnormalized endpoint densities is
\begin{equation}\label{eq:lambda-endpoint-calibration}
 \frac{e^{-aq/\lambda}}{e^{-q/\lambda}}
 =\exp\left(\frac{(1-a)q}{\lambda}\right).
\end{equation}
For $q=0$, $\varpi_0^\lambda$ is exactly the uniform probability law on $[a,1]$ for every
$\lambda>0$.  For fixed $q\ne0$,
\begin{equation}\label{eq:lambda-regimes-control}
 \varpi_q^\lambda\Longrightarrow
 \begin{cases}
  \delta_a,&q>0,\\
  \delta_1,&q<0,
 \end{cases}
 \qquad\text{as }\lambda\downarrow0,
\end{equation}
whereas $\varpi_q^\lambda$ converges to the uniform law as $\lambda\to\infty$.
These observations are only tuning heuristics; all mathematical results below hold for every
$\vartheta>0$ and $\lambda>0$.  The lower bound $a$ has a different role: it is the exploration
floor and the uniform ellipticity constant used throughout the analysis.
\subsection{Well-posedness, localization, and continuity}\label{subsec:controlled-wellposedness}
We first construct the controlled equation on the collision-free state space and establish estimates
uniform over all admissible controls.
\begin{proposition}\label{prop:controlled-dynamics}
For every relaxed control $\pi$ and every initial condition $X\in\hX$, there exists a unique
continuous process $X^\pi=(X_t^\pi)_{t\ge0}$, adapted to the raw canonical filtration
$(\Fc_t^0)_{t\ge0}$, such that, almost surely,
\begin{equation}\label{eq:controlled-integral-form}
 X_t^\pi
 =X-\int_0^t\nabla E(X_s^\pi)\dd s
 +\int_0^t\sigma_s^\pi\dd W_s,
 \qquad t\ge0.
\end{equation}
In other words, \eqref{eq:controlled-SDE} has a pathwise unique global strong solution.  This
solution never reaches the collision set:
\begin{equation}\label{eq:no-collision-control}
 \P_X\left[X_t^\pi\in\hX\text{ for every }t\ge0\right]=1.
\end{equation}
It also satisfies
\begin{equation}\label{eq:moment-controlled}
 \E_X|X_t^\pi|^2
 \le e^{-\pmin t}|X|^2+
 \frac{m_2(\nu)+2m}{\pmin},
 \qquad t\ge0,
\end{equation}
uniformly over all relaxed controls.
\end{proposition}
\begin{proof}
Put $b:=-\nabla E$.  We divide the proof into four steps.
\medskip
\noindent\textit{Step 1: construction of the maximal strong solution.}
We first record why the solution may be chosen raw adapted.  Since
$\sigma^\pi$ is bounded and raw progressively measurable, it can be approximated in
$L^2([0,T]\times\Omega_W,\dd t\otimes\P)$ by raw predictable step processes on every finite horizon $T$.
The corresponding stochastic integrals are continuous and raw adapted, and the
Burkholder--Davis--Gundy inequality gives convergence, along a subsequence, uniformly on
compact time intervals almost surely.  We henceforth use the resulting continuous raw-adapted
version of
\[
 M_t^\pi:=\int_0^t\sigma_s^\pi\dd W_s.
\]
For a globally Lipschitz drift, Picard iteration in the integral equation driven by $M^\pi$
then produces a continuous raw-adapted strong solution.
For $n\ge1$, let
\begin{equation}\label{eq:exhaustion-controlled-dynamics}
 D_n:=\left\{Y\in\hX:|Y|<n,\ \mathfrak d(Y)>n^{-1}\right\}.
\end{equation}
The sets $D_n$ are increasing, $\overline D_n\Subset\hX$, and
$\bigcup_nD_n=\hX$.  By Theorem~\ref{thm:C2}, $b$ is locally Lipschitz on $\hX$; hence its
restriction to $\overline D_n$ is Lipschitz.  Extend it componentwise to a globally Lipschitz map
$b_n:\R^m\to\R^m$ which agrees with $b$ on $\overline D_n$.  Since $\sigma^\pi$ is progressively
measurable and bounded, the equation
\begin{equation}\label{eq:extended-SDE-controlled}
 \dd Y_t^n=b_n(Y_t^n)\dd t+\sigma_t^\pi\dd W_t,
 \qquad Y_0^n=X,
\end{equation}
has a pathwise unique global strong solution.  Choose $n_0$ with $X\in D_{n_0}$ and, for
$n\ge n_0$, define
\[
 \tau_n:=\inf\{t\ge0:Y_t^n\notin D_n\},
 \qquad \inf\varnothing:=\infty.
\]
If $k>n$, the coefficients in the equations for $Y^n$ and $Y^k$ agree as long as both paths remain
in $D_n$.  Pathwise uniqueness therefore gives
\[
 Y_t^k=Y_t^n\quad\text{for }0\le t\le\tau_n,
 \qquad
 \tau_n\le\tau_k.
\]
We may consequently define
\begin{equation}\label{eq:maximal-controlled-solution}
 \tau_\infty:=\lim_{n\to\infty}\tau_n,
 \qquad
 X_t^\pi:=Y_t^n\quad\text{whenever }t<\tau_n.
\end{equation}
This is the pathwise unique maximal strong solution in $\hX$.
The possible endpoints of the maximal solution can now be defined without circularity.  For
$R>0$ and $\delta>0$, put
\begin{align}
 \zeta_R&:=\inf\{t<\tau_\infty:|X_t^\pi|\ge R\},
 &\zeta&:=\lim_{R\to\infty}\zeta_R,
 \label{eq:explosion-times-control}\\
 \chi_\delta&:=\inf\{t<\tau_\infty:\mathfrak d(X_t^\pi)\le\delta\},
 &\chi&:=\lim_{\delta\downarrow0}\chi_\delta.
 \label{eq:collision-times-control}
\end{align}
The time $\zeta$ is the explosion time and $\chi$ is the collision-approach time.  The exhaustion
construction implies
\begin{equation}\label{eq:lifetime-decomposition-control}
 \tau_\infty=\zeta\wedge\chi.
\end{equation}
Indeed, if $\tau_\infty<\infty$ while the path remained in a compact subset of $\hX$, that compact
set would be contained in some $D_n$ and the solution could be continued beyond $\tau_\infty$,
contradicting maximality.
\medskip
\noindent\textit{Step 2: exclusion of finite-time explosion.}
Fix $T,R>0$.  For $n$ sufficiently large that $X\in D_n$, apply It\^o's formula up to
$T\wedge\zeta_R\wedge\chi_{1/n}$.  Proposition~\ref{prop:supergradient-growth} and
Proposition~\ref{prop:C1} give
\begin{equation}\label{eq:drift-Lyapunov-control}
 2X\cdot b(X)+m(\sigma_t^\pi)^2
 =-2X\cdot\nabla E(X)+2m\bar u_t^\pi
 \le-\pmin|X|^2+m_2(\nu)+2m.
\end{equation}
Dropping the negative term, taking expectations, and using the stopped martingale property yield
\begin{equation}\label{eq:stopped-second-moment-control}
 \E_X\left[|X_{T\wedge\zeta_R\wedge\chi_{1/n}}^\pi|^2\right]
 \le |X|^2+\bigl(m_2(\nu)+2m\bigr)T.
\end{equation}
On the event
$\{\zeta_R\le T,\ \zeta_R<\chi_{1/n}\}$, the stopped norm equals $R$.  Therefore
\[
 \P_X(\zeta_R\le T,\ \zeta_R<\chi_{1/n})
 \le\frac{|X|^2+(m_2(\nu)+2m)T}{R^2}.
\]
Letting $n\to\infty$ gives the same bound with $\chi$ in place of $\chi_{1/n}$.
If $\zeta\le T$, then $\zeta<\infty$.  By the definition of $\zeta$ as the increasing
limit of the level-hitting times, every finite level $R>|X|$ is reached at some time
$\zeta_R<\tau_\infty$.  Since
$\tau_\infty=\zeta\wedge\chi$, this implies
$\zeta_R\le\zeta\le T$ and $\zeta_R<\chi$.  Consequently,
\[
 \{\zeta\le T\}\subset
 \{\zeta_R\le T,\ \zeta_R<\chi\}
 \qquad\text{for every }R>|X|,
\]
and hence
\[
 \P_X(\zeta\le T)
 \le\frac{|X|^2+(m_2(\nu)+2m)T}{R^2}
 \qquad\text{for every sufficiently large }R.
\]
Letting $R\to\infty$ proves $\P_X(\zeta\le T)=0$.  Since $T$ is arbitrary,
\begin{equation}\label{eq:no-explosion-control-detailed}
 \P_X(\zeta=\infty)=1.
\end{equation}
\medskip
\noindent\textit{Step 3: exclusion of collisions.}
We recall the probabilistic terminology used here.  A singleton is said to be \emph{polar} for a
process if the process, when started outside that singleton, hits it with probability zero.  For a
$d$-dimensional Brownian motion $B$ and $d\ge2$,
\begin{equation}\label{eq:Brownian-point-polar-control}
 \P_z(\exists t\ge0:B_t=0)=0,
 \qquad z\ne0;
\end{equation}
that is, every point is polar for Brownian motion.
Fix $R,T>0$.  On the event $\{\chi\le T\wedge\zeta_R\}$, the drift is bounded up to
$\chi$ and the diffusion coefficient is bounded; hence both integrals in
\eqref{eq:controlled-integral-form} have continuous limits as $t\uparrow\chi$.  We use this
limit to define the stopped path at $\chi$.  Moreover, if $\delta_n\downarrow0$, then
$\chi_{\delta_n}\uparrow\chi$ and
$\mathfrak d(X_{\chi_{\delta_n}}^\pi)\le\delta_n$; continuity therefore gives
$\mathfrak d(X_\chi^\pi)=0$.  Thus a collision is actually attained at time $\chi$ on this
event.  Set
\[
 \sigma:=T\wedge\zeta_R\wedge\chi.
\]
By \eqref{eq:growth-supergradient}, $b$ is bounded on
$\{Y\in\hX:|Y|\le R\}$; let $M_R$ be such a bound.  Define the progressively measurable Girsanov
integrand
\begin{equation}\label{eq:Girsanov-integrand-control}
 \gamma_t:=(\sigma_t^\pi)^{-1}b(X_t^\pi)\id_{\{t<\sigma\}}.
\end{equation}
This process is sometimes called the Girsanov kernel.  Here ``bounded'' means that its Euclidean
norm is bounded by a deterministic constant, uniformly in time and in the sample point.  Indeed,
$(\sigma_t^\pi)^{-1}$ denotes multiplication by the scalar reciprocal and
$\sigma_t^\pi\ge\sqrt{2a}$, so
\[
 |\gamma_t|\le M_R/\sqrt{2a},
 \qquad0\le t\le T,
 \quad\P\text{-almost surely}.
\]
Thus Novikov's condition holds, and
\begin{equation}\label{eq:Girsanov-density-control}
 \mathcal E_T
 :=\exp\left(-\int_0^T\gamma_s\cdot\dd W_s
 -\frac12\int_0^T|\gamma_s|^2\dd s\right)
\end{equation}
is a strictly positive martingale of mean one.  Define an equivalent probability measure on
$\Fc_T$ by $\dd\Q_{R,T}=\mathcal E_T\dd\P$.  Girsanov's theorem states that
\begin{equation}\label{eq:Girsanov-Brownian-control}
 W_t^{R,T}:=W_t+\int_0^t\gamma_s\dd s,
 \qquad 0\le t\le T,
\end{equation}
is an $m$-dimensional Brownian motion under $\Q_{R,T}$.  Substitution into the stopped equation
shows that
\begin{equation}\label{eq:zero-drift-stopped-control}
 X_{t\wedge\sigma}^\pi
 =X+\int_0^{t\wedge\sigma}\sigma_s^\pi\dd W_s^{R,T},
 \qquad 0\le t\le T.
\end{equation}
Fix $i<j$.  The difference of the two $d$-dimensional site coordinates satisfies, under
$\Q_{R,T}$,
\begin{equation}\label{eq:difference-local-martingale}
 X_{t\wedge\sigma}^{\pi,i}-X_{t\wedge\sigma}^{\pi,j}
 =x_i-x_j+\int_0^{t\wedge\sigma}2\sqrt{\bar u_s^\pi}\dd B_s^{ij},
\end{equation}
where $B^{ij}$ is a $d$-dimensional Brownian motion.  Its coordinate covariations are
\begin{equation}\label{eq:qv-difference}
 \left\langle X^{\pi,i,k}-X^{\pi,j,k},
 X^{\pi,i,\ell}-X^{\pi,j,\ell}\right\rangle_{t\wedge\sigma}
 =4\delta_{k\ell}\int_0^{t\wedge\sigma}\bar u_s^\pi\dd s.
\end{equation}
The multidimensional Dambis--Dubins--Schwarz theorem therefore gives a $d$-dimensional Brownian
motion $\widetilde B^{ij}$ such that
\begin{equation}\label{eq:DDS-difference-control}
 X_{t\wedge\sigma}^{\pi,i}-X_{t\wedge\sigma}^{\pi,j}
 =x_i-x_j+\widetilde B^{ij}_{A_t},
 \qquad
 A_t:=4\int_0^{t\wedge\sigma}\bar u_s^\pi\dd s\le4T.
\end{equation}
On $\{\chi\le T\wedge\zeta_R\}$, the preceding continuity argument shows that,
for at least one pair $i<j$, the continuous path in
\eqref{eq:DDS-difference-control} hits the point $0$.  This has zero probability by
\eqref{eq:Brownian-point-polar-control}; see also
\cite[Chapters~V and XI]{RevuzYor1999}.  The union over the finitely many pairs is still null.
Equivalence of $\Q_{R,T}$ and $\P$ consequently gives
\[
 \P_X(\chi\le T\wedge\zeta_R)=0.
\]
By \eqref{eq:no-explosion-control-detailed}, on the event $\{\chi\le T\}$ the path up to time
$\chi$ is contained in some ball.  Hence
\[
 \{\chi\le T\}\subset
 \bigcup_{R\in\mathbb N}\{\chi\le T\wedge\zeta_R\}
 \qquad\P_X\text{-almost surely}.
\]
Each event on the right is null, so $\P_X(\chi\le T)=0$, and then
$\P_X(\chi=\infty)=1$.  Together with \eqref{eq:lifetime-decomposition-control}, this proves
global existence and \eqref{eq:no-collision-control}.
\medskip
\noindent\textit{Step 4: the uniform second-moment estimate.}
Now that the solution is global, apply It\^o's formula to
$e^{\pmin t}|X_t^\pi|^2$ up to $t\wedge\zeta_R$.  Using
\eqref{eq:drift-Lyapunov-control} gives
\[
 \E_X\left[e^{\pmin(t\wedge\zeta_R)}
 |X_{t\wedge\zeta_R}^\pi|^2\right]
 \le |X|^2+\frac{m_2(\nu)+2m}{\pmin}(e^{\pmin t}-1).
\]
Let $R\to\infty$.  Since $\zeta_R\to\infty$ almost surely, Fatou's lemma yields
\[
 e^{\pmin t}\E_X|X_t^\pi|^2
 \le |X|^2+\frac{m_2(\nu)+2m}{\pmin}(e^{\pmin t}-1),
\]
which implies \eqref{eq:moment-controlled}.  Every constant used above is independent of the
control, completing the proof.
\end{proof}
The next lemma is not an additional qualitative well-posedness result.  Its role is to provide
\emph{uniform} compact localization over all relaxed controls.  This uniformity is essential for
the continuity of the cost and value functions below, and later for the locally uniform Euler
approximation in Proposition~\ref{prop:euler-finite-horizon}.
\begin{lemma}\label{lem:uniform-localization-control}
Let $K\Subset\hX$ and $T>0$.  Uniformly over all relaxed controls and all initial states $X\in K$,
\begin{equation}\label{eq:uniform-fourth-moment-control}
 \sup_{\pi}\sup_{X\in K}
 \E_X\left[\sup_{0\le t\le T}|X_t^\pi|^4\right]<\infty.
\end{equation}
Moreover,
\begin{align}
 \lim_{R\to\infty}\sup_{\pi}\sup_{X\in K}
 \P_X\left(\sup_{0\le t\le T}|X_t^\pi|\ge R\right)&=0,
 \label{eq:uniform-no-escape-control}\\
 \lim_{\delta\downarrow0}\sup_{\pi}\sup_{X\in K}
 \P_X\left(\min_{i<j}\inf_{0\le t\le T}
 |X_t^{\pi,i}-X_t^{\pi,j}|\le\delta\right)&=0.
 \label{eq:uniform-no-near-collision-control}
\end{align}
\end{lemma}
\begin{proof}
We first prove the fourth-moment estimate with the stopping made explicit.  By
Propositions~\ref{prop:supergradient-growth} and \ref{prop:C1},
\begin{equation}\label{eq:linear-growth-drift-control}
 |b(Y)|^2\le C(1+|Y|^2),
 \qquad Y\in\hX.
\end{equation}
For $n\ge1$, let
$\zeta_n:=\inf\{t\ge0:|X_t^\pi|\ge n\}$ and put
$Y_t^n:=X_{t\wedge\zeta_n}^\pi$.  The stopped process satisfies
\begin{equation}\label{eq:stopped-equation-fourth-moment}
 Y_t^n
 =X+\int_0^t\id_{\{s<\zeta_n\}}b(Y_s^n)\dd s
 +\int_0^t\id_{\{s<\zeta_n\}}\sigma_s^\pi\dd W_s.
\end{equation}
The elementary inequality $|z_1+z_2+z_3|^4\le27\sum_k|z_k|^4$, H\"older's inequality in time,
\eqref{eq:linear-growth-drift-control}, and the Burkholder--Davis--Gundy inequality give, for
$t\le T$,
\begin{align}
 \E_X\left[\sup_{0\le r\le t}|Y_r^n|^4\right]
 &\le C_T\left(1+|X|^4+
 \int_0^t\E_X\left[\sup_{0\le r\le s}|Y_r^n|^4\right]\dd s\right).
 \label{eq:stopped-fourth-moment-control}
\end{align}
For completeness, the drift term in \eqref{eq:stopped-equation-fourth-moment} is bounded by
\begin{align*}
 \E\sup_{r\le t}\left|\int_0^{r\wedge\zeta_n}b(X_s^\pi)\dd s\right|^4
 &\le t^3\int_0^t\E|b(Y_s^n)|^4\dd s\\
 &\le C_T\left(1+\int_0^t
 \E\sup_{0\le r\le s}|Y_r^n|^4\dd s\right),
\end{align*}
while the stochastic integral contributes at most
\[
 C\E\left(\int_0^{t\wedge\zeta_n}(\sigma_s^\pi)^2\dd s\right)^2
 \le4CT^2.
\]
Gronwall's lemma makes the right-hand side of
\eqref{eq:stopped-fourth-moment-control} independent of $n$, of the control, and, when $X\in K$, of
the initial state.  Since Proposition~\ref{prop:controlled-dynamics} gives
$\zeta_n\uparrow\infty$ almost surely, $Y^n\to X^\pi$ uniformly on $[0,T]$ almost surely.  Fatou's
lemma applied to the stopped processes now yields \eqref{eq:uniform-fourth-moment-control}.
Markov's inequality immediately gives \eqref{eq:uniform-no-escape-control}.
We next prove the uniform near-collision estimate.  Fix $R>0$ and write
\[
 \zeta_R:=\inf\{t\ge0:|X_t^\pi|\ge R\}.
\]
Let
\[
 M_R:=\sup\{|b(Y)|:Y\in\hX,\ |Y|\le R\}<\infty,
 \qquad
 K_R:=M_R/\sqrt{2a}.
\]
Define
\[
 \gamma_t:=(\sigma_t^\pi)^{-1}b(X_t^\pi)
 \id_{\{t<T\wedge\zeta_R\}}.
\]
Then $|\gamma_t|\le K_R$.  With $\mathcal E_T$ defined as in
\eqref{eq:Girsanov-density-control}, let $\dd\Q_{R,T}=\mathcal E_T\dd\P$.  To quantify the
comparison between the two measures, fix $p>1$ and let $q=p/(p-1)$.  Under $\Q_{R,T}$,
\[
 \frac{\dd\P}{\dd\Q_{R,T}}
 =\exp\left(\int_0^T\gamma_s\cdot\dd W_s^{R,T}
 -\frac12\int_0^T|\gamma_s|^2\dd s\right).
\]
Writing the $p$th power as the stochastic exponential with integrand $p\gamma$ multiplied by its
finite-variation correction gives
\begin{align}
 \E_{\Q_{R,T}}\left[
 \left(\frac{\dd\P}{\dd\Q_{R,T}}\right)^p\right]
 &\le\exp\left(\frac{p(p-1)}2K_R^2T\right).
 \label{eq:RN-p-moment-collision}
\end{align}
Therefore H\"older's inequality gives, for every $A\in\Fc_T$,
\begin{equation}\label{eq:girsanov-holder-collision}
 \P_X(A)
 =\E_{\Q_{R,T}}\left[
 \frac{\dd\P}{\dd\Q_{R,T}}\id_A\right]
 \le \exp\left(\frac{p-1}{2}K_R^2T\right)
 \Q_{R,T}(A)^{1/q}.
\end{equation}
This estimate is uniform over $X\in K$ and over the relaxed control.
For a fixed pair $i<j$, the stopped difference under $\Q_{R,T}$ has the representation
\begin{equation}\label{eq:pair-time-change-localization}
 X_{t\wedge\zeta_R}^{\pi,i}-X_{t\wedge\zeta_R}^{\pi,j}
 =x_i-x_j+\widetilde B^{ij}_{A_t},
 \qquad
 A_t=4\int_0^{t\wedge\zeta_R}\bar u_s^\pi\dd s\le4T,
\end{equation}
for a $d$-dimensional Brownian motion $\widetilde B^{ij}$.  Since $A_t\in[0,4T]$, entry of the
time-changed difference into $\overline B_\delta(0)$ before $T\wedge\zeta_R$ is contained in
the event that an ordinary Brownian motion enters $\overline B_\delta(0)$ before time $4T$.
Since $K\Subset\hX$,
\begin{equation}\label{eq:rK-control}
 r_K:=\min_{X\in K}\mathfrak d(X)>0.
\end{equation}
We claim that, for a standard $d$-dimensional Brownian motion $B$,
\begin{equation}\label{eq:brownian-small-ball-localization}
 \lim_{\delta\downarrow0}
 \sup_{|z|\ge r_K}
 \P\left(\inf_{0\le s\le4T}|z+B_s|\le\delta\right)=0.
\end{equation}
We give the details of this uniform Brownian estimate.  Fix $\varepsilon>0$.  First choose
$M>r_K+1$ so large that
\[
 \P\left(\sup_{0\le s\le4T}|B_s|\ge M-1\right)<\varepsilon.
\]
For $|z|>M$ and $0<\delta\le1$, hitting $\overline B_\delta(0)$ from $z$ before time $4T$ requires
$\sup_{s\le4T}|B_s|\ge|z|-\delta>M-1$, so the corresponding probability is less than
$\varepsilon$.  It remains to consider the compact annulus
$A:=\{z:r_K\le|z|\le M\}$.  For each fixed $z\in A$, the events
\[
 H_\delta(z):=\left\{\inf_{0\le s\le4T}|z+B_s|\le\delta\right\}
\]
decrease, as $\delta\downarrow0$, to the event that the Brownian path hits the polar point $-z$.
Thus $\P(H_\delta(z))\downarrow0$.  Choose $\delta_z>0$ such that
$\P(H_{2\delta_z}(z))<\varepsilon$.  If $|z'-z|<\delta_z$, then, under the coupling using the same
Brownian path,
\[
 H_{\delta_z}(z')\subset H_{2\delta_z}(z).
\]
A finite collection of the balls $B(z,\delta_z)$ covers $A$.  Taking the minimum of the finitely
many radii shows that
$\sup_{z\in A}\P(H_\delta(z))<\varepsilon$ for all sufficiently small $\delta$.  Together with the
far-field estimate, this proves \eqref{eq:brownian-small-ball-localization}.
Combining \eqref{eq:pair-time-change-localization} and
\eqref{eq:brownian-small-ball-localization}, then applying
\eqref{eq:girsanov-holder-collision}, gives
\[
 \lim_{\delta\downarrow0}\sup_\pi\sup_{X\in K}
 \P_X\left(\inf_{0\le t\le T\wedge\zeta_R}
 |X_t^{\pi,i}-X_t^{\pi,j}|\le\delta\right)=0.
\]
There are only finitely many pairs.  Hence
\begin{align*}
 &\P_X\left(\min_{i<j}\inf_{0\le t\le T}
 |X_t^{\pi,i}-X_t^{\pi,j}|\le\delta\right)\\
 &\quad\le
 \P_X(\zeta_R\le T)
 +\sum_{i<j}\P_X\left(\inf_{0\le t\le T\wedge\zeta_R}
 |X_t^{\pi,i}-X_t^{\pi,j}|\le\delta\right).
\end{align*}
First let $\delta\downarrow0$ and then let $R\to\infty$, using
\eqref{eq:uniform-no-escape-control}.  This proves
\eqref{eq:uniform-no-near-collision-control}.
\end{proof}
The next lemma is the analytic bridge from the controlled dynamics to dynamic programming.  Its
growth estimate makes the value finite, while its equicontinuity statement supplies the countable
near-optimal continuation selector used in Proposition~\ref{prop:DPP}.
\begin{lemma}\label{lem:value-continuity}
The value function is finite and satisfies
\begin{equation}\label{eq:value-growth}
 -\frac{\lambda\log(1-a)}{\vartheta}
 \le v_\lambda(X)\le C_\lambda(1+|X|^2),
 \qquad X\in\hX.
\end{equation}
Moreover, for every compact set $K\Subset\hX$, the family of maps
\[
 X\longmapsto J_\lambda(X;\pi),
\]
indexed by controls having finite cost at one point of $K$, is equicontinuous on $K$.  Every such
control has finite cost at all points of $K$.  Consequently, $v_\lambda$ is continuous on $\hX$.
\end{lemma}
\begin{proof}
The lower bound follows from $E\ge0$ and \eqref{eq:entropy-lower}.  For the upper bound, use the
constant action given by the uniform density $(1-a)^{-1}\id_{[a,1]}$.  Its entropy is
$-\log(1-a)$, and \eqref{eq:E-growth} together with \eqref{eq:moment-controlled} gives
\[
 \E_XE(X_t^{\pi^0})\le C\bigl(1+e^{-\pmin t}|X|^2\bigr).
\]
Integration against $e^{-\vartheta t}$ proves the upper bound in \eqref{eq:value-growth}.
For an arbitrary admissible control, decompose
\begin{align*}
 J_\lambda(X;\pi)
 &=J_E(X;\pi)+J_{\mathrm{ent}}(\pi),\\
 J_E(X;\pi)
 &:=\E_X\int_0^\infty e^{-\vartheta t}E(X_t^\pi)\dd t,\\
 J_{\mathrm{ent}}(\pi)
 &:=\lambda\E\int_0^\infty e^{-\vartheta t}
 \Ent_{[a,1]}(\pi_t)\dd t.
\end{align*}
The moment bound \eqref{eq:moment-controlled} and the quadratic growth of $E$ imply that
$J_E(X;\pi)<\infty$ for every $X\in\hX$ and every admissible control.  The entropy term is
independent of the initial state.  Hence, if $J_\lambda(Y;\pi)<\infty$ at one point $Y$, then
$J_{\mathrm{ent}}(\pi)<\infty$ and $J_\lambda(X;\pi)<\infty$ for every $X\in\hX$.  For such a
control,
\begin{equation}\label{eq:entropy-cancels-value-continuity}
 J_\lambda(X;\pi)-J_\lambda(Z;\pi)
 =J_E(X;\pi)-J_E(Z;\pi).
\end{equation}
This observation avoids subtracting extended real-valued costs.
Fix $K\Subset\hX$, $X,Z\in K$, and a common open-loop relaxed control $\pi$ of finite cost at one
point of $K$.  Construct the two state processes with the same Brownian motion.  The diffusion
coefficient $\sigma_t^\pi$ depends on the control but not on the state, so the stochastic terms
cancel exactly in the difference.  For $R>0$ and $\delta>0$, let
\begin{equation}\label{eq:safe-domain-value-continuity}
 D_{R,\delta}:=\{Y\in\hX:|Y|<R,\ \mathfrak d(Y)>\delta\}
\end{equation}
and let $\sigma_{R,\delta}$ be the first time at which either coupled trajectory leaves this set.
By a union bound and Lemma~\ref{lem:uniform-localization-control}, $R$ may be chosen large and then
$\delta$ small so that, uniformly over the control and over $X,Z\in K$,
\begin{equation}\label{eq:bad-event-value-continuity}
 \P(\sigma_{R,\delta}\le T)\le\varepsilon.
\end{equation}
On $\overline D_{R,\delta}$, the drift $b=-\nabla E$ is Lipschitz, say with constant
$L_{R,\delta}$.  Hence, for $0\le t\le T\wedge\sigma_{R,\delta}$,
\[
 |X_t^\pi-Z_t^\pi|
 \le |X-Z|+L_{R,\delta}\int_0^t|X_s^\pi-Z_s^\pi|\dd s,
\]
and Gronwall's lemma gives
\begin{equation}\label{eq:coupling-value-continuity}
 |X_t^\pi-Z_t^\pi|
 \le e^{L_{R,\delta}t}|X-Z|,
 \qquad 0\le t\le T\wedge\sigma_{R,\delta}.
\end{equation}
On the good event $\{\sigma_{R,\delta}>T\}$, local Lipschitz continuity of $E$ and
\eqref{eq:coupling-value-continuity} give
\begin{equation}\label{eq:good-cost-value-continuity}
 \E\left[\id_{\{\sigma_{R,\delta}>T\}}
 \int_0^T e^{-\vartheta t}|E(X_t^\pi)-E(Z_t^\pi)|\dd t\right]
 \le C_{R,\delta,T}|X-Z|.
\end{equation}
On the bad event, the quadratic growth bound for $E$ gives
\[
 |E(X_t^\pi)-E(Z_t^\pi)|
 \le C\bigl(1+|X_t^\pi|^2+|Z_t^\pi|^2\bigr).
\]
Cauchy--Schwarz, \eqref{eq:bad-event-value-continuity}, and the uniform fourth-moment estimate yield
\begin{equation}\label{eq:bad-cost-value-continuity}
 \E\left[\id_{\{\sigma_{R,\delta}\le T\}}
 \int_0^T e^{-\vartheta t}|E(X_t^\pi)-E(Z_t^\pi)|\dd t\right]
 \le C_{K,T}\varepsilon^{1/2}.
\end{equation}
Finally, \eqref{eq:E-growth} and \eqref{eq:moment-controlled} imply
\begin{equation}\label{eq:tail-value-continuity}
 \sup_\pi\sup_{Y\in K}
 \E_Y\int_T^\infty e^{-\vartheta t}E(X_t^\pi)\dd t
 \le C_Ke^{-\vartheta T}.
\end{equation}
Combining \eqref{eq:entropy-cancels-value-continuity} and
\eqref{eq:good-cost-value-continuity}--\eqref{eq:tail-value-continuity} gives
\begin{equation}\label{eq:cost-continuity-pre-DPP}
 |J_\lambda(X;\pi)-J_\lambda(Z;\pi)|
 \le C_{R,\delta,T}|X-Z|+C_{K,T}\varepsilon^{1/2}+2C_Ke^{-\vartheta T}.
\end{equation}
First choose $T$, then $R$ and $\delta$, and finally $|X-Z|$.  This proves equicontinuity uniformly
over controls of finite cost at one point of $K$.  Applying the estimate to an
$\varepsilon$-optimal control at $X$ and then to an $\varepsilon$-optimal control at $Z$ proves the
continuity of $v_\lambda$.
\end{proof}
\subsection{Canonical shifts, concatenation, and dynamic programming}\label{subsec:DPP}
We now make explicit the canonical operations used in the dynamic programming argument.  For
$t\ge0$ and $\omega,\omega'\in\Omega_W$, define
\begin{equation}\label{eq:path-concatenation-control}
 (\omega\otimes_t\omega')(s)
 :=
 \begin{cases}
  \omega(s),&0\le s\le t,\\
  \omega(t)+\omega'(s-t),&s>t,
 \end{cases}
\end{equation}
and define the increment shift
\begin{equation}\label{eq:Wiener-shift-control}
 (\mathsf S_t\omega)(s):=\omega(t+s)-\omega(t),
 \qquad s\ge0.
\end{equation}
If $\theta$ is a bounded $\mathbb F^0$-stopping time, write
\[
 \omega\otimes_\theta\omega'
 :=\omega\otimes_{\theta(\omega)}\omega',
 \qquad
 \mathsf S_\theta\omega:=\mathsf S_{\theta(\omega)}\omega.
\]
Since $\omega\otimes_\theta\omega'$ agrees with $\omega$ up to time $\theta(\omega)$,
Galmarino's test gives
\begin{equation}\label{eq:stopping-time-stable-concatenation}
 \theta(\omega\otimes_\theta\omega')=\theta(\omega).
\end{equation}
For an admissible control $\pi$ and a past path $\omega$, define the shifted control
\begin{equation}\label{eq:shifted-control-definition}
 \pi_r^{\theta,\omega}(\omega')
 :=\pi_{\theta(\omega)+r}(\omega\otimes_\theta\omega'),
 \qquad r\ge0,
 \quad \omega'\in\Omega_W.
\end{equation}
\begin{lemma}\label{lem:canonical-shift-control}
Let $\theta$ be a bounded $\mathbb F^0$-stopping time.
\begin{enumerate}[label=\textup{(\roman*)},leftmargin=2.3em]
\item For every $\omega\in\Omega_W$, the shifted process
$\pi^{\theta,\omega}$ in \eqref{eq:shifted-control-definition} is an admissible raw-progressive
control.
\item The probability kernel
\begin{equation}\label{eq:canonical-rcpd-kernel}
 K_\theta(\omega,A)
 :=\int_{\Omega_W}\id_A(\omega\otimes_\theta\omega')\P(\dd\omega'),
 \qquad A\in\Fc_\infty^0,
\end{equation}
is a regular conditional probability distribution of Wiener measure given $\Fc_\theta^0$.
Equivalently, for every nonnegative or bounded Borel functional $F$ on $\Omega_W$,
\begin{equation}\label{eq:canonical-conditional-formula}
 \E[F\mid\Fc_\theta^0](\omega)
 =\int_{\Omega_W}F(\omega\otimes_\theta\omega')\P(\dd\omega')
 \qquad\text{for }\P\text{-a.e. }\omega.
\end{equation}
\item Let $X^{X,\pi}$ denote the raw-adapted version constructed in
Proposition~\ref{prop:controlled-dynamics}.  For $\P$-almost every past path $\omega$, one has,
for Wiener-almost every future path $\omega'$,
\begin{equation}\label{eq:strong-flow-control}
 X_{\theta(\omega)+r}^{X,\pi}(\omega\otimes_\theta\omega')
 =X_r^{X_\theta^{X,\pi}(\omega),\,\pi^{\theta,\omega}}(\omega'),
 \qquad r\ge0.
\end{equation}
Consequently,
\begin{align}
 &\E_X\left[
  \left.
  \int_\theta^\infty e^{-\vartheta s}
  \ell_\lambda(X_s^\pi,\pi_s)\dd s
  \right|\Fc_\theta^0\right](\omega)
\notag\\
 &\qquad=e^{-\vartheta\theta(\omega)}
 J_\lambda\left(X_\theta^{X,\pi}(\omega);\pi^{\theta,\omega}\right)
 \qquad\text{for }\P\text{-a.e. }\omega.
 \label{eq:conditional-tail-cost-control}
\end{align}
\item Let $\pi^0,\pi^1,\pi^2,\ldots$ be admissible controls and let
$(B_k)_{k\ge1}$ be a countable $\Fc_\theta^0$-measurable partition of $\Omega_W$.  Define
\begin{equation}\label{eq:concatenated-control-definition}
 \bar\pi_s(\omega)
 :=
 \begin{cases}
  \pi_s^0(\omega),&s<\theta(\omega),\\[1mm]
  \displaystyle
  \sum_{k\ge1}\id_{B_k}(\omega)
  \pi_{s-\theta(\omega)}^k(\mathsf S_\theta\omega),
  &s\ge\theta(\omega).
 \end{cases}
\end{equation}
Then $\bar\pi$ is admissible.  Its state agrees with the state controlled by $\pi^0$ up to
$\theta$, and, on $B_k$, its post-$\theta$ state is the strong solution restarted from
$X_\theta^{X,\pi^0}$ under the continuation control $\pi^k$.  In particular,
\begin{align}
 &\E_X\left[
  \left.
  \int_\theta^\infty e^{-\vartheta s}
  \ell_\lambda(X_s^{\bar\pi},\bar\pi_s)\dd s
  \right|\Fc_\theta^0\right]
\notag\\
 &\qquad=e^{-\vartheta\theta}
 \sum_{k\ge1}\id_{B_k}
 J_\lambda(X_\theta^{X,\pi^0};\pi^k).
 \label{eq:conditional-pasted-cost-control}
\end{align}
\end{enumerate}
\end{lemma}
\begin{proof}
Fix first a past path $\omega$ and put $t=\theta(\omega)$.  If $r\le R$, then the restriction of
$\omega\otimes_t\omega'$ to $[0,t+r]$ depends on $\omega'$ only through its restriction to
$[0,r]$.  Since $\pi$ is raw progressive, the map
\[
 (r,\omega')\longmapsto
 \pi_{t+r}(\omega\otimes_t\omega')
\]
is measurable on $[0,R]\times\Omega_W$ with respect to
$\mathcal B([0,R])\otimes\Fc_R^0$.  This proves (i).
The map $(\omega,\omega')\mapsto\omega\otimes_\theta\omega'$ is Borel measurable, and the
right-hand side of \eqref{eq:canonical-conditional-formula} is $\Fc_\theta^0$-measurable.  For a
deterministic time, \eqref{eq:canonical-conditional-formula} follows first for bounded cylinder
functionals from the independent-increment property of Brownian motion and then for all bounded
Borel functionals by a monotone-class argument.  For bounded continuous cylinder functionals, approximating a bounded stopping time from above
by finitely valued stopping times and using continuity of the canonical paths gives the same formula
at $\theta$.  A second monotone-class argument then extends it to every bounded Borel functional.
Thus the kernel \eqref{eq:canonical-rcpd-kernel} is a regular conditional probability
distribution, proving (ii).
We next prove the flow identity.  The integral equation for $X^{X,\pi}$ holds outside one
$\P$-null set simultaneously at every rational time, hence at every time by continuity.  Applying
the regular conditional distribution in (ii), for almost every past $\omega$ the same integral
equation holds under $K_\theta(\omega,\cdot)$.  Subtracting its value at $\theta(\omega)$ from its
value at $\theta(\omega)+r$ gives
\begin{align*}
 X_{\theta(\omega)+r}^{X,\pi}(\omega\otimes_\theta\omega')
 ={}&X_\theta^{X,\pi}(\omega)
 -\int_0^r\nabla E\!
 \left(X_{\theta(\omega)+s}^{X,\pi}(\omega\otimes_\theta\omega')\right)\dd s\\
 &+\int_0^r
 \sigma_{\theta(\omega)+s}^\pi(\omega\otimes_\theta\omega')\dd W_s(\omega').
\end{align*}
The shift identity for the stochastic integral is immediate for raw predictable step integrands and
extends to bounded raw-progressive integrands by the It\^o isometry and the
Burkholder--Davis--Gundy inequality.  By \eqref{eq:shifted-control-definition}, the last integrand is
$\sigma_s^{\pi^{\theta,\omega}}(\omega')$.  The right-hand side is therefore the restarted state
equation, and pathwise uniqueness from Proposition~\ref{prop:controlled-dynamics} proves
\eqref{eq:strong-flow-control}.  Formula \eqref{eq:conditional-tail-cost-control} follows from
\eqref{eq:canonical-conditional-formula}, \eqref{eq:strong-flow-control}, the change of variables
$s=\theta+r$, and Tonelli's theorem.
For (iv), first consider an elementary raw-progressive continuation process.  On
$\{s\ge\theta\}$, its value at time $s-\theta$ along $\mathsf S_\theta\omega$ depends only on
$\theta$, the path up to time $\theta$, and the increments between $\theta$ and $s$; it is therefore
$\mathcal B([0,t])\otimes\Fc_t^0$-measurable on $[0,t]\times\Omega_W$.  A monotone-class
argument extends this observation to arbitrary raw-progressive continuation controls.  Since each
$B_k\in\Fc_\theta^0$ and exactly one indicator in
\eqref{eq:concatenated-control-definition} is nonzero, $\bar\pi$ is raw progressive.  The prefix of
$\bar\pi$ agrees with $\pi^0$, so pathwise uniqueness gives equality of the two states up to
$\theta$.  Applying the flow argument after $\theta$ on each $B_k$ proves the restarted-state
identity and then \eqref{eq:conditional-pasted-cost-control}.
\end{proof}
\begin{proposition}\label{prop:DPP}
Let $\tau$ be any bounded stopping time for the raw canonical filtration $\mathbb F^0$.  Then, for
every $X\in\hX$,
\begin{equation}\label{eq:DPP-HJB-detailed}
 v_\lambda(X)=\inf_\pi\E_X\left[
 \int_0^\tau e^{-\vartheta s}\ell_\lambda(X_s^\pi,\pi_s)\dd s
 +e^{-\vartheta\tau}v_\lambda(X_\tau^\pi)\right].
\end{equation}
The same principle holds for control-dependent localized exit times.  Namely, if
$D\Subset\hX$ is open, $T>0$, and
\begin{equation}\label{eq:localized-stopping-time-DPP}
 \tau_{D,T}^\pi
 :=T\wedge\inf\{s\ge0:X_s^\pi\notin D\},
\end{equation}
then, for every $X\in D$,
\begin{equation}\label{eq:DPP-stopping-control}
 v_\lambda(X)=\inf_\pi\E_X\left[
 \int_0^{\tau_{D,T}^\pi}e^{-\vartheta s}
 \ell_\lambda(X_s^\pi,\pi_s)\dd s
 +e^{-\vartheta\tau_{D,T}^\pi}
 v_\lambda(X_{\tau_{D,T}^\pi}^\pi)\right].
\end{equation}
In both formulas, controls for which the displayed expectation is infinite may be omitted from the
infimum.
\end{proposition}
\begin{proof}
We prove both identities simultaneously.  For a fixed admissible control $\pi$, let $\theta^\pi$
denote either the prescribed stopping time $\tau$ or the exit time $\tau_{D,T}^\pi$.  Since
$X^\pi$ is continuous and raw adapted, $\tau_{D,T}^\pi$ is a bounded raw stopping time.  Indeed,
for $t<T$,
\[
 \{\tau_{D,T}^\pi\le t\}
 =\left\{
 \inf_{q\in\mathbb Q\cap[0,t]}
 \operatorname{dist}(X_q^\pi,D^c)=0
 \right\}\in\Fc_t^0,
\]
while $\{\tau_{D,T}^\pi\le t\}=\Omega_W$ for $t\ge T$.
Let first $\pi$ have finite cost.  Splitting its nonnegative running cost at $\theta^\pi$ and using
\eqref{eq:conditional-tail-cost-control} gives
\begin{align}
 J_\lambda(X;\pi)
 =\E_X\Bigg[&\int_0^{\theta^\pi}e^{-\vartheta s}
 \ell_\lambda(X_s^\pi,\pi_s)\dd s\notag\\
 &+e^{-\vartheta\theta^\pi}
 J_\lambda\left(X_{\theta^\pi}^\pi;
 \pi^{\theta^\pi,\omega}\right)\Bigg].
 \label{eq:cost-split-DPP}
\end{align}
For almost every past path, $\pi^{\theta^\pi,\omega}$ is admissible.  Hence the last cost is at
least $v_\lambda(X_{\theta^\pi}^\pi)$.  Taking the infimum over $\pi$ proves the inequality
``$\ge$'' in \eqref{eq:DPP-HJB-detailed} and \eqref{eq:DPP-stopping-control}.
For the reverse inequality, fix $\varepsilon>0$.  For each $Y\in\hX$, choose an admissible control
$\pi^Y$ such that
\begin{equation}\label{eq:epsilon-optimal-continuations-DPP}
 J_\lambda(Y;\pi^Y)\le v_\lambda(Y)+\varepsilon.
\end{equation}
Choose $r_Y>0$ with $\overline B_{2r_Y}(Y)\Subset\hX$.  Lemma~\ref{lem:value-continuity}, applied
on this compact set, shows that both $v_\lambda$ and
$J_\lambda(\,\cdot\,;\pi^Y)$ are continuous there.  After decreasing $r_Y$ if necessary,
\begin{equation}\label{eq:local-selector-DPP}
 J_\lambda(Z;\pi^Y)\le v_\lambda(Z)+3\varepsilon,
 \qquad Z\in O_Y:=B_{r_Y}(Y).
\end{equation}
The open set $\hX\subset\R^m$ is second countable, so $(O_Y)_{Y\in\hX}$ has a countable subcover
$(O_k)_{k\ge1}$ with corresponding controls $(\pi^k)_{k\ge1}$.  Define the Borel partition
\begin{equation}\label{eq:Borel-partition-DPP}
 A_1:=O_1,
 \qquad
 A_k:=O_k\setminus\bigcup_{\ell<k}O_\ell,
 \quad k\ge2.
\end{equation}
Then \eqref{eq:local-selector-DPP} holds with $\pi^k$ for every $Z\in A_k$.
Fix now a prefix control $\pi^0$ for which the expression on the right-hand side of the relevant
DPP formula is finite, and put $\theta:=\theta^{\pi^0}$.  The events
\[
 B_k:=\{X_\theta^{\pi^0}\in A_k\}
\]
form an $\Fc_\theta^0$-measurable partition.  Paste $\pi^0$ with the continuation controls
$\pi^k$ according to \eqref{eq:concatenated-control-definition}, and call the resulting admissible
control $\bar\pi$.  In the exit-time case, the pasted control agrees with $\pi^0$ until $\theta$;
therefore
\[
 \tau_{D,T}^{\bar\pi}=\tau_{D,T}^{\pi^0}=\theta.
\]
Using \eqref{eq:conditional-pasted-cost-control}, \eqref{eq:local-selector-DPP}, and the tower
property, we obtain
\[
 J_\lambda(X;\bar\pi)
 \le\E_X\left[
 \int_0^\theta e^{-\vartheta s}
 \ell_\lambda(X_s^{\pi^0},\pi_s^0)\dd s
 +e^{-\vartheta\theta}
 \bigl(v_\lambda(X_\theta^{\pi^0})+3\varepsilon\bigr)
 \right].
\]
Since $e^{-\vartheta\theta}\le1$, taking the infimum over $\pi^0$ and then letting
$\varepsilon\downarrow0$ proves the reverse inequality in both DPP formulas.
\end{proof}
\subsection{The exploratory Hamilton--Jacobi--Bellman equation}\label{subsec:HJB}
The variational identity \eqref{eq:gibbs-variational} shows that the dynamic programming equation
associated with \eqref{eq:control-cost} is
\begin{equation}\label{eq:exploratory-HJB}
 -\vartheta v_\lambda(X)
 -\nabla E(X)\cdot\nabla v_\lambda(X)
 +E(X)
 +\mathscr H_\lambda\bigl(\Delta v_\lambda(X)\bigr)=0,
 \qquad X\in\hX.
\end{equation}
We first isolate the two regularity steps needed after the viscosity characterization.
\begin{lemma}\label{lem:HJB-interior-regularity}
Let $D'\Subset D\Subset\hX$, and let $u$ be a bounded viscosity solution on $D$ of
\[
 -\vartheta u-\nabla E\cdot\nabla u+E+\mathscr H_\lambda(\Delta u)=0.
\]
There exists
$\bar\beta=\bar\beta(m,a)\in(0,1)$ such that, for every
$0<\beta<\bar\beta$,
\[
 u\in C^{2,\beta}(D').
\]
The corresponding interior norm is bounded in terms of
\[
 \operatorname{dist}(D',\partial D),\quad m,\quad a,\quad
 \vartheta,\quad\lambda,\quad
 \|E\|_{C^2(D)},\quad\sup_D|u|.
\]
\end{lemma}
\begin{proof}
Put $w:=-u$.  After adding the constant $\mathscr H_\lambda(0)$, the equation for $w$ can be
written as
\begin{equation}\label{eq:convex-HJB-for-w}
 \widehat\Fc_\lambda(D^2w,\nabla w,w,X)
 =E(X)+\mathscr H_\lambda(0),
\end{equation}
where
\begin{equation}\label{eq:normalized-HJB-operator-main}
 \widehat\Fc_\lambda(A,p,r,X)
 :=-\mathscr H_\lambda(-\tr A)+\mathscr H_\lambda(0)
 -\nabla E(X)\cdot p-\vartheta r.
\end{equation}
Then $\widehat\Fc_\lambda(0,0,0,X)=0$.  If $B$ is nonnegative and symmetric,
\begin{equation}\label{eq:proper-uniform-ellipticity-HJB}
 a\tr B
 \le
 \widehat\Fc_\lambda(A+B,p,r,X)-\widehat\Fc_\lambda(A,p,r,X)
 \le\tr B
\end{equation}
by \eqref{eq:H-uniform-ellipticity}.  Moreover,
$A\mapsto\widehat\Fc_\lambda(A,p,r,X)$ is convex because
$s\mapsto-\mathscr H_\lambda(-s)$ is convex.  On $D$,
\begin{align}
 &|\widehat\Fc_\lambda(A,p,r,X)-
 \widehat\Fc_\lambda(A,p',r',X)|\notag\\
 &\qquad\le \|\nabla E\|_{L^\infty(D)}|p-p'|
 +\vartheta|r-r'|.
 \label{eq:lower-order-HJB-structure}
\end{align}
For $X_0\in D$, freezing the spatial coefficient gives
\begin{align}
 &|\widehat\Fc_\lambda(A,p,r,X)-
 \widehat\Fc_\lambda(A,p,r,X_0)|\notag\\
 &\qquad\le \|D^2E\|_{L^\infty(D)}|X-X_0|\,|p|.
 \label{eq:coefficient-oscillation-HJB-main}
\end{align}
Thus the special structure condition and the coefficient-oscillation hypothesis of
\cite[Theorem~5.1]{LianWangZhang2026} hold, with ellipticity constants $(a,1)$ and a Lipschitz
oscillation modulus.  The right-hand side of \eqref{eq:convex-HJB-for-w} is locally Lipschitz and
therefore locally $C^{0,\beta}$ for every $\beta<1$.  Since the coefficients are continuous, the
standard and $L^m$ viscosity notions agree; see
\cite[Remark~2.9]{LianWangZhang2026}.  Translating and rescaling balls compactly contained in $D$,
applying \cite[Theorem~5.1]{LianWangZhang2026}, and using the equivalence between pointwise and
local H\"older regularity stated in \cite[Remark~2.3]{LianWangZhang2026} give the asserted estimate
after a finite covering of $D'$.  The argument is entirely interior and uses no boundary value on
$\Coll$.
\end{proof}
\begin{lemma}\label{lem:laplacian-C1}
Let $u\in C_{\mathrm{loc}}^{2,\beta}(\hX)$ be a viscosity solution of
\[
 -\vartheta u-\nabla E\cdot\nabla u+E+\mathscr H_\lambda(\Delta u)=0.
\]
Then the equation holds pointwise and
\begin{equation}\label{eq:laplacian-C1-lemma}
 \Delta u\in C_{\mathrm{loc}}^1(\hX).
\end{equation}
\end{lemma}
\begin{proof}
Since $u\in C_{\mathrm{loc}}^2(\hX)$ and the HJB operator is continuous, the viscosity equation is
a classical pointwise identity.  Since $\mathscr H_\lambda$ is a smooth diffeomorphism,
\begin{equation}\label{eq:laplacian-inversion-detailed}
 \Delta u
 =\mathscr H_\lambda^{-1}
 \left(\vartheta u+\nabla E\cdot\nabla u-E\right).
\end{equation}
Now $E\in C^2(\hX)$ by Theorem~\ref{thm:C2}, while
$u\in C_{\mathrm{loc}}^{2,\beta}(\hX)$.  Therefore the expression in parentheses belongs to
$C_{\mathrm{loc}}^1(\hX)$.  Composition with the smooth map
$\mathscr H_\lambda^{-1}$ proves \eqref{eq:laplacian-C1-lemma}.  Notice that this additional
derivative is obtained from the HJB identity; it is not a consequence of
$C^{2,\beta}$ regularity alone.
\end{proof}
\begin{proof}[Proof of Theorem~\ref{thm:HJB}]
We indicate explicitly how the preceding auxiliary results enter the proof.

\medskip
\noindent\textit{Step 1: finiteness, growth, and continuity.}
These conclusions are exactly Lemma~\ref{lem:value-continuity}.  The proof of that lemma uses the
uniform localization result of Lemma~\ref{lem:uniform-localization-control}, but not the dynamic
programming principle.

\medskip
\noindent\textit{Step 2: viscosity solution property.}
Define the proper operator
\begin{equation}\label{eq:proper-HJB-operator}
 \Fc_\lambda(X,r,p,A)
 :=\vartheta r+\nabla E(X)\cdot p-E(X)
 -\mathscr H_\lambda(\tr A).
\end{equation}
Equation \eqref{eq:exploratory-HJB} is equivalent to
$\Fc_\lambda(X,v,\nabla v,D^2v)=0$.
Fix $X_0\in\hX$ and choose $r>0$ such that
$\overline B_{2r}(X_0)\Subset\hX$.  For a control $\pi$ and $h>0$, set
\begin{equation}\label{eq:viscosity-exit-time-control}
 \tau_h^\pi
 :=h\wedge\inf\{t\ge0:X_t^\pi\notin B_r(X_0)\}.
\end{equation}
On $B_r(X_0)$ the drift is bounded, uniformly over controls, and the diffusion coefficient is
bounded by $\sqrt2$.  The stopped integral equation and the Burkholder--Davis--Gundy inequality
therefore give
\begin{equation}\label{eq:short-time-state-estimate}
 \sup_\pi\E_{X_0}\left[
 \sup_{0\le s\le\tau_h^\pi}|X_s^\pi-X_0|^2\right]
 \le C_r(h+h^2).
\end{equation}
Indeed, the drift contribution is bounded by $C_rh^2$, while the quadratic expectation of the
stopped stochastic integral is bounded by $2mh$.  In particular,
\begin{equation}\label{eq:short-time-exit-probability-control}
 \sup_\pi\P_{X_0}(\tau_h^\pi<h)
 \le r^{-2}C_r(h+h^2)\longrightarrow0.
\end{equation}
Consequently, $\tau_h^\pi/h\to1$ in probability uniformly over controls.  If $G$ is continuous on
$\overline B_r(X_0)$, uniform continuity of $G$, \eqref{eq:short-time-state-estimate}, and
\eqref{eq:short-time-exit-probability-control} yield
\begin{equation}\label{eq:short-time-average-control}
 \frac1h\E_{X_0}\int_0^{\tau_h^\pi}e^{-\vartheta s}G(X_s^\pi)\dd s
 \longrightarrow G(X_0),
\end{equation}
uniformly over controls whenever $G$ itself does not depend on the chosen control.
Suppose first that $v_\lambda-\varphi$ has a local maximum zero at $X_0$, with
$\varphi\in C^2(B_{2r}(X_0))$.  After decreasing $r$ if necessary,
$v_\lambda\le\varphi$ on $\overline B_r(X_0)$.  Fix a constant action
$\varpi\in\mathfrak U$ with finite entropy and apply the stopping-time DPP
\eqref{eq:DPP-stopping-control} with $D=B_r(X_0)$ and horizon $h$.  Since the value is the infimum,
using this particular constant control gives
\[
 0\le\E_{X_0}\left[
 \int_0^{\tau_h^\varpi}e^{-\vartheta s}
 \ell_\lambda(X_s^\varpi,\varpi)\dd s
 +e^{-\vartheta\tau_h^\varpi}\varphi(X_{\tau_h^\varpi}^\varpi)
 -\varphi(X_0)\right].
\]
It\^o's formula for $e^{-\vartheta t}\varphi(X_t^\varpi)$ up to
$\tau_h^\varpi$ turns this into
\begin{align*}
 0\le\E_{X_0}\int_0^{\tau_h^\varpi}e^{-\vartheta s}
 \Bigl[&E(X_s^\varpi)+\lambda\Ent_{[a,1]}(\varpi)
 -\vartheta\varphi(X_s^\varpi)\\
 &-\nabla E(X_s^\varpi)\cdot\nabla\varphi(X_s^\varpi)
 +\bar u(\varpi)\Delta\varphi(X_s^\varpi)\Bigr]\dd s.
\end{align*}
Divide by $h$ and let $h\downarrow0$ using
\eqref{eq:short-time-average-control}.  Since $\varpi$ is arbitrary,
\eqref{eq:gibbs-variational} gives
\[
 E(X_0)-\vartheta\varphi(X_0)
 -\nabla E(X_0)\cdot\nabla\varphi(X_0)
 +\mathscr H_\lambda(\Delta\varphi(X_0))\ge0,
\]
which is equivalent to
\[
 \Fc_\lambda(X_0,\varphi(X_0),\nabla\varphi(X_0),D^2\varphi(X_0))\le0.
\]
Thus $v_\lambda$ is a viscosity subsolution.
Suppose now that $v_\lambda-\varphi$ has a local minimum zero at $X_0$.  Then
$v_\lambda\ge\varphi$ on $\overline B_r(X_0)$.  In the stopping-time DPP choose a control
$\pi^h$ whose displayed value is within $\varepsilon_h:=h^2$ of the infimum.  The terminal
comparison with $\varphi$ and It\^o's formula give
\begin{align*}
 h^2\ge\E_{X_0}\int_0^{\tau_h^{\pi^h}}e^{-\vartheta s}
 \Bigl[&E(X_s^{\pi^h})+\lambda\Ent_{[a,1]}(\pi_s^h)
 -\vartheta\varphi(X_s^{\pi^h})\\
 &-\nabla E(X_s^{\pi^h})\cdot\nabla\varphi(X_s^{\pi^h})
 +\bar u(\pi_s^h)\Delta\varphi(X_s^{\pi^h})\Bigr]\dd s.
\end{align*}
Pointwise in $(s,\omega)$, the Gibbs inequality yields
\[
 \bar u(\pi_s^h)\Delta\varphi(X_s^{\pi^h})
 +\lambda\Ent_{[a,1]}(\pi_s^h)
 \ge\mathscr H_\lambda(\Delta\varphi(X_s^{\pi^h})).
\]
The remaining integrand is now a continuous function of the state only.  Divide by $h$ and use
\eqref{eq:short-time-average-control}; since $h^2/h\to0$, we obtain
\[
 \Fc_\lambda(X_0,\varphi(X_0),\nabla\varphi(X_0),D^2\varphi(X_0))\ge0.
\]
This proves that $v_\lambda$ is a viscosity solution.  The role of
Proposition~\ref{prop:DPP} in the main theorem is precisely this step.

\medskip
\noindent\textit{Step 3: interior regularity and differentiability of the Laplacian.}
Lemma~\ref{lem:HJB-interior-regularity}, applied on arbitrary domains
$D'\Subset D\Subset\hX$, yields a number $\beta\in(0,1)$ such that
\[
 v_\lambda\in C_{\mathrm{loc}}^{2,\beta}(\hX).
\]
This proves \eqref{eq:v-C2beta}.  Lemma~\ref{lem:laplacian-C1} then gives
\[
 q_\lambda:=\Delta v_\lambda\in C_{\mathrm{loc}}^1(\hX),
\]
which is \eqref{eq:q-C1}.  In particular, the HJB equation holds classically on $\hX$.

\medskip
\noindent\textit{Step 4: construction and admissibility of the feedback.}
For $X\in\hX$, the unique minimizer of the pointwise Hamiltonian is
\[
 \varpi_\lambda^*(X)(\dd u)
 :=\varpi_{q_\lambda(X)}^\lambda(\dd u)
 =\pi_\lambda^*(u;X)\dd u,
\]
where $\pi_\lambda^*$ is given by \eqref{eq:optimal-density}.  Equations
\eqref{eq:H-prime} and \eqref{eq:q-C1} imply
\[
 \tau_\lambda,h_\lambda\in C^1_{\mathrm{loc}}(\hX),
 \qquad
 \sqrt{2a}\le h_\lambda\le\sqrt2.
\]
Thus both coefficients of the feedback equation \eqref{eq:optimal-feedback-SDE} are locally
Lipschitz on $\hX$.  The Picard and exhaustion construction in Step~1 of
Proposition~\ref{prop:controlled-dynamics}, now applied to the two state-dependent coefficients,
produces a pathwise unique maximal strong solution which can be chosen continuous and adapted to
the raw canonical filtration.  The Lyapunov calculation of Step~2 there is unchanged because
$h_\lambda^2=2\tau_\lambda\le2$.  In the localized collision argument, Girsanov removes the drift
and each pairwise difference has quadratic clock
\[
 4\int_0^t\tau_\lambda(X_s^*)\dd s,
\]
which lies between $4at$ and $4t$.  The same polarity argument therefore excludes collisions, and
the feedback solution is global.  The same Lyapunov calculation also yields the estimate
\eqref{eq:moment-controlled} with $X^\pi$ replaced by $X^*$.
The map $X\mapsto\varpi_\lambda^*(X)$ is continuous from $\hX$ to $\mathfrak U$.  Since $X^*$ is
continuous and raw adapted, it is raw progressive; hence
\begin{equation}\label{eq:feedback-control-raw-progressive}
 \pi_t^*(\dd u)
 :=\varpi_\lambda^*(X_t^*)(\dd u)
 =\pi_\lambda^*(u;X_t^*)\dd u
\end{equation}
is raw progressively measurable and is admissible in the sense of
\eqref{eq:admissible-control-class}.

\medskip
\noindent\textit{Step 5: verification and finite feedback cost.}
Let $\pi$ be any finite-cost control and set
$q_t:=\Delta v_\lambda(X_t^\pi)$.  The Gibbs inequality gives
\begin{equation}\label{eq:gibbs-verification-detailed}
 \bar u(\pi_t)q_t+\lambda\Ent_{[a,1]}(\pi_t)
 \ge\mathscr H_\lambda(q_t),
\end{equation}
with equality if and only if
$\pi_t(\dd u)=\pi_\lambda^*(u;X_t^\pi)\dd u$.
Choose the exhaustion
$D_k=\{Y\in\hX:|Y|<k,\ \mathfrak d(Y)>k^{-1}\}$ and let
$\sigma_k$ be the first exit time of $X^\pi$ from $D_k$.  Since
$v_\lambda$, $\nabla v_\lambda$, and $D^2v_\lambda$ are bounded on
$\overline D_k$, It\^o's formula for
$e^{-\vartheta t}v_\lambda(X_t^\pi)$ up to $T\wedge\sigma_k$ has a martingale with zero expectation.
Using the classical HJB equation and \eqref{eq:gibbs-verification-detailed} yields
\begin{align}
 &\E_X\left[e^{-\vartheta(T\wedge\sigma_k)}
 v_\lambda(X_{T\wedge\sigma_k}^\pi)\right]-v_\lambda(X)\notag\\
 &\qquad\ge-\E_X\int_0^{T\wedge\sigma_k}e^{-\vartheta t}
 \ell_\lambda(X_t^\pi,\pi_t)\dd t.
 \label{eq:Ito-verification-detailed}
\end{align}
By Lemma~\ref{lem:uniform-localization-control},
$\sup_{s\le T}|X_s^\pi|^4$ is integrable.  Together with the quadratic growth
\eqref{eq:value-growth}, this makes
$v_\lambda(X_{T\wedge\sigma_k}^\pi)$ uniformly integrable in $k$.  Since
$\sigma_k\to\infty$ almost surely and the running cost is nonnegative, we may let $k\to\infty$ in
\eqref{eq:Ito-verification-detailed}.  Furthermore, \eqref{eq:value-growth} and
\eqref{eq:moment-controlled} give
\begin{equation}\label{eq:terminal-vanishes-verification}
 \E_X\left[e^{-\vartheta T}|v_\lambda(X_T^\pi)|\right]
 \le C e^{-\vartheta T}
 \left(1+e^{-\pmin T}|X|^2+\frac{m_2(\nu)+2m}{\pmin}\right)
 \longrightarrow0.
\end{equation}
Letting $T\to\infty$ proves
$v_\lambda(X)\le J_\lambda(X;\pi)$.  For an infinite-cost control this inequality is automatic.
For the feedback process $X^*$, equality holds in
\eqref{eq:gibbs-verification-detailed}.  The stopped identity becomes
\[
 \E_X\int_0^{T\wedge\sigma_k}e^{-\vartheta t}
 \ell_\lambda(X_t^*,\pi_t^*)\dd t
 =v_\lambda(X)-
 \E_X\left[e^{-\vartheta(T\wedge\sigma_k)}
 v_\lambda(X_{T\wedge\sigma_k}^*)\right]
 \le v_\lambda(X),
\]
because \eqref{eq:value-growth} implies
$v_\lambda\ge-\lambda\log(1-a)/\vartheta>0$.  Monotone convergence first in $k$ and then in $T$
shows that the feedback cost is finite.  Finally, letting $k,T\to\infty$ in the equality and using
\eqref{eq:terminal-vanishes-verification} gives
$J_\lambda(X;\pi^*)=v_\lambda(X)$.
\end{proof}
\begin{remark}\label{rem:regularity-feedback}
The conclusion $\Delta v_\lambda\in C_{\mathrm{loc}}^1$ is not inferred from
$v_\lambda\in C_{\mathrm{loc}}^{2,\beta}$.  The latter gives only
$\Delta v_\lambda\in C_{\mathrm{loc}}^{0,\beta}$.  The additional derivative is supplied by the
algebraic inversion of the trace-form HJB equation in Lemma~\ref{lem:laplacian-C1}.  Thus no full
$C^3$ estimate for the value function, and no unproved H\"older modulus for the moving-facet
second derivatives of $E$, is needed.
\end{remark}
\begin{remark}\label{rem:HJB-selection}
The theorem identifies the solution selected by the discounted stochastic control problem.  We do
not claim uniqueness among arbitrary classical solutions on the punctured state space $\hX$ without
a growth condition and without data on the polar collision set.  A numerical HJB solver should
therefore approximate the value-function solution, for instance through discounted finite-domain
problems with a consistent stochastic representation.
\end{remark}
\begin{remark}\label{rem:feedback-lipschitz}
Writing $\beta_\lambda(q):=\mathscr H_\lambda'(q)$, one has
\[
 \beta_\lambda'(q)
 =-\lambda^{-1}\operatorname{Var}_{\varpi_q^\lambda}(u),
 \qquad
 |\beta_\lambda'(q)|\le\frac{(1-a)^2}{4\lambda}.
\]
Since $\beta_\lambda\ge a$,
\begin{equation}\label{eq:h-feedback-Lipschitz-q}
 \left|\frac{\dd}{\dd q}\sqrt{2\beta_\lambda(q)}\right|
 \le\frac{(1-a)^2}{4\lambda\sqrt{2a}}.
\end{equation}
Thus an approximation error in $\Delta v_\lambda$ produces an explicitly controlled pointwise error
in the feedback temperature.
\end{remark}

\section{Langevin discretization and global record convergence}\label{sec:numerical}

This section has two logically distinct purposes.  First, when $d\ge2$, we record a
finite-horizon consistency result for the Euler discretization of the exact HJB feedback from
Theorem~\ref{thm:HJB}.  Second, and independently of that consistency result, we study the
long-time behavior of a homogeneous Euler chain driven by an arbitrary fixed Borel temperature
rule with values in $[a,1]$.  The latter argument uses only the global dissipativity of
$-\nabla E$ and the nondegeneracy of the Gaussian increments; in particular, it does not use
optimality of the HJB feedback and is valid in every dimension
\begin{equation}\label{eq:d-ge-1-discrete}
 d\ge1,
 \qquad
 m:=dN.
\end{equation}
The distinction between the two results is important: the HJB feedback is optimal for the
discounted criterion of Section~\ref{sec:control}, whereas running-record convergence holds for
every fixed Borel temperature rule bounded away from zero.

The feedback $\tau_\lambda$ is determined by an elliptic equation in state dimension $m=dN$.
Classical grid discretizations are therefore realistic only for small $m$; larger problems require
a parametric approximation of $\Delta v_\lambda$.  Recent numerical work uses physics-informed
neural networks tailored to the exploratory HJB structure
\cite{WangLiWangZhang2026}.  We do not analyze a particular HJB solver here.  Instead, the
asymptotic result below is formulated so that every fixed Borel approximation of the temperature,
once clipped to $[a,1]$, retains the record-convergence guarantee.  This robustness concerns the
temperature only and does not cover errors in the semi-discrete Wasserstein gradient.

\subsection{The Euler exploration algorithm}

Assume first that $d\ge2$, so that the exact feedback of Theorem~\ref{thm:HJB} is available.  Let
$\eta>0$ and let $\xi_1,\xi_2,\ldots$ be independent standard Gaussian vectors in $\R^m$.  The
Euler chain associated with the exact exploratory feedback is
\begin{equation}\label{eq:euler-exact}
 X_{n+1}
 =X_n-\eta\nabla E(X_n)
 +\sqrt{2\eta\tau_\lambda(X_n)}\,\xi_{n+1},
 \qquad X_0\in\hX.
\end{equation}
Here and below, the gradient is the full Euclidean gradient introduced in
\eqref{eq:full-gradient-definition-main}.  Its $i$th $\R^d$-valued component is
\begin{equation}\label{eq:gradient-component-algorithm}
 \nabla_{x_i}E(X_n)
 =p_i\bigl(x_{n,i}-c_i(X_n)\bigr).
\end{equation}
Consequently, the deterministic part of the update can be written componentwise as
\begin{equation}\label{eq:underrelaxed-Lloyd-step}
 x_{n+1,i}^{\mathrm{det}}
 =x_{n,i}-\eta p_i\bigl(x_{n,i}-c_i(X_n)\bigr)
 =(1-\eta p_i)x_{n,i}+\eta p_i c_i(X_n).
\end{equation}
Thus, whenever $0<\eta\le1$, the drift step is an under-relaxed weighted Lloyd step, followed by
Gaussian exploration.

Let $q:\hX\to\R$ be a fixed Borel approximation of $q_\lambda=\Delta v_\lambda$.  The induced
mean-temperature rule is
\begin{equation}\label{eq:approx-temperature-algorithm}
 \tau_q(X)
 :=\mathscr H_\lambda'(q(X))
 =\frac{\int_a^1u\exp[-uq(X)/\lambda]\dd u}
 {\int_a^1\exp[-uq(X)/\lambda]\dd u}
 \in(a,1),
 \qquad X\in\hX.
\end{equation}
For $q=q_\lambda$, this gives $\tau_q=\tau_\lambda$.

\begin{algorithm}[State-dependent Laguerre--Langevin exploration]\label{alg:exploration}
Fix $X_0\in\hX$, a step size $\eta>0$, and a fixed Borel function
$q:\hX\to\R$.  For $n=0,1,2,\ldots$:
\begin{enumerate}[label=\textup{\arabic*.},leftmargin=2.4em]
\item solve the normalized dual problem and obtain $\Phi^*(X_n)\in U$;
\item construct the balanced Laguerre cells and compute
\[
 c_i(X_n)=\frac1{p_i}\int_{V_i(X_n)}y\rho(y)\dd y,
 \qquad i=1,\ldots,N;
\]
\item set
\[
 g_{n,i}:=p_i\bigl(x_{n,i}-c_i(X_n)\bigr),
 \qquad
 g_n:=(g_{n,1},\ldots,g_{n,N})=\nabla E(X_n);
\]
\item compute the approximate mean temperature $\tau_q(X_n)$ from
\eqref{eq:approx-temperature-algorithm};
\item draw $\xi_{n+1}\sim\Nc(0,I_m)$ independently of the past and update
\begin{equation}\label{eq:algorithm-update}
 X_{n+1}=X_n-\eta g_n+\sqrt{2\eta\tau_q(X_n)}\,\xi_{n+1};
\end{equation}
\item retain the earliest best-so-far record
\begin{equation}\label{eq:record-definition}
 \Record_n:=X_{\kappa_n},
 \qquad
 \kappa_n:=\min\left\{0\le k\le n:
 E(X_k)=\min_{0\le\ell\le n}E(X_\ell)\right\}.
\end{equation}
\end{enumerate}
\end{algorithm}

The long-time result is more general than this HJB-based implementation.  From now on, let
$d\ge1$ and let
\begin{equation}\label{eq:generic-temperature}
 \tau:\hX\longrightarrow[a,1]
\end{equation}
be an arbitrary fixed Borel function.  Consider the homogeneous Markov chain
\begin{equation}\label{eq:euler-generic}
 X_{n+1}
 =X_n-\eta\nabla E(X_n)
 +\sqrt{2\eta\tau(X_n)}\,\xi_{n+1},
 \qquad X_0=X\in\hX.
\end{equation}
The word ``fixed'' means that the same state-dependent rule is used at every iteration; a rule
which changes with $n$ or depends on the entire past would produce a nonhomogeneous chain and is
not covered by Theorem~\ref{thm:numerical-convergence} without enlarging the state space.  When
$d\ge2$, the exact rule $\tau=\tau_\lambda$ and every fixed Borel approximation $\tau=\tau_q$ are
included.

The update in \eqref{eq:euler-generic} is well defined almost surely at every step in every
dimension $d\ge1$.  Indeed, conditionally on $X_n\in\hX$, the random vector $X_{n+1}$ has a
nondegenerate Gaussian distribution on the full ambient space $\R^m$.  The collision set $\Coll$ is
a finite union of affine subspaces of codimension $d\ge1$ and hence has $m$-dimensional Lebesgue
measure zero.  Therefore
\[
 \P(X_{n+1}\in\Coll\mid X_n)=0,
\]
and induction gives $X_n\in\hX$ for every $n$ almost surely whenever $X_0\in\hX$.  Notice that this
argument does not use polarity or path-connectedness: in dimension one the Gaussian chain may jump
across a collision hyperplane without landing on it.

\begin{remark}\label{rem:exact-gradient}
The convergence theorem below uses the exact drift $-\nabla E$.  In practice, both the balancing
weights and the cell moments are computed numerically.  A perturbed-gradient theorem would require
quantitative assumptions ensuring that the numerical drift preserves the global dissipativity and
the recurrence estimates used below.  Such a conclusion does not follow merely by clipping the
temperature and lies outside the present analysis.
\end{remark}

\subsection{Finite-horizon consistency of the exact-feedback scheme}

Throughout this subsection only, assume $d\ge2$.  Let $X^*$ denote the feedback diffusion
constructed in Theorem~\ref{thm:HJB}.  On the same Brownian space, put $t_n=n\eta$ and
\[
 \kappa_\eta(t):=t_n,
 \qquad t\in[t_n,t_{n+1}),
\]
and define the continuous Euler interpolation by
\begin{equation}\label{eq:continuous-Euler-interpolation}
 \overline X_t^\eta
 :=X_0-\int_0^t\nabla E\bigl(\overline X_{\kappa_\eta(s)}^\eta\bigr)\dd s
 +\int_0^t\sqrt{2\tau_\lambda
 \bigl(\overline X_{\kappa_\eta(s)}^\eta\bigr)}\dd W_s.
\end{equation}
At the grid times $t_n$, this interpolation has the same law as the chain
\eqref{eq:euler-exact}.

\begin{proposition}\label{prop:euler-finite-horizon}
Let $T>0$.  Then
\begin{equation}\label{eq:euler-convergence-prob}
 \sup_{0\le t\le T}|\overline X_t^\eta-X_t^*|
 \longrightarrow0
 \qquad\text{in probability as }\eta\downarrow0.
\end{equation}
More precisely, for every compact set $K_0\Subset\hX$ and every $\varepsilon_0>0$,
\begin{equation}\label{eq:euler-convergence-locally-uniform}
 \lim_{\eta\downarrow0}\sup_{X_0\in K_0}
 \P_{X_0}\left(
 \sup_{0\le t\le T}|\overline X_t^\eta-X_t^*|>\varepsilon_0
 \right)=0.
\end{equation}
\end{proposition}

\begin{proof}
The argument localizes the standard Euler estimate for globally Lipschitz coefficients.

\smallskip
\noindent\textit{Step 1: three nested collision-free domains.}
Fix $K_0\Subset\hX$ and $\varepsilon>0$.  By
Lemma~\ref{lem:uniform-localization-control}, applied to the optimal feedback control, we may first
choose $R>1$ and then $\delta>0$ so small that $K_0$ is contained in
\[
 D_0:=\left\{X\in\hX:|X|<R,\ \mathfrak d(X)>4\delta\right\}
\]
and
\begin{equation}\label{eq:exact-exit-small-Euler}
 \sup_{X_0\in K_0}\P_{X_0}(\tau_0\le T)<\varepsilon,
 \qquad
 \tau_0:=\inf\{t\ge0:X_t^*\notin D_0\}.
\end{equation}
Introduce the larger domains
\begin{align*}
 D_1&:=\left\{X\in\hX:|X|<2R,\ \mathfrak d(X)>2\delta\right\},\\
 D_2&:=\left\{X\in\hX:|X|<3R,\ \mathfrak d(X)>\delta\right\}.
\end{align*}
Then
\[
 \overline D_0\subset D_1,
 \qquad
 \overline D_1\subset D_2,
 \qquad
 \overline D_2\Subset\hX.
\]
The radial and separation buffers will ensure that a sufficiently accurate Euler path remains in the
region where the original and truncated coefficients coincide.

\smallskip
\noindent\textit{Step 2: globally Lipschitz truncated coefficients.}
By Theorem~\ref{thm:HJB},
\[
 b(X):=-\nabla E(X),
 \qquad
 \sigma(X):=\sqrt{2\tau_\lambda(X)}
\]
are locally Lipschitz on $\hX$.  Choose $\chi\in C_c^\infty(D_2)$ with
$0\le\chi\le1$ and $\chi=1$ on $\overline D_1$.  The products $\chi b$ and $\chi\sigma$ vanish in
a neighborhood of $\partial D_2$ and hence extend by zero to globally Lipschitz functions on
$\R^m$.  Denote these extensions by $b^{R,\delta}$ and $\sigma^{R,\delta}$.

Let $X^{R,\delta}$ solve
\begin{equation}\label{eq:truncated-feedback-SDE-Euler}
 \dd X_t^{R,\delta}
 =b^{R,\delta}(X_t^{R,\delta})\dd t
 +\sigma^{R,\delta}(X_t^{R,\delta})\dd W_t,
 \qquad X_0^{R,\delta}=X_0,
\end{equation}
and let $\overline X^{\eta,R,\delta}$ be its continuous Euler interpolation.  The coefficients are
globally Lipschitz, with constants independent of $X_0\in K_0$.  The standard strong Euler estimate
therefore gives
\begin{equation}\label{eq:global-truncated-Euler-estimate}
 \sup_{X_0\in K_0}
 \E_{X_0}\left[
 \sup_{0\le t\le T}
 |\overline X_t^{\eta,R,\delta}-X_t^{R,\delta}|^2
 \right]
 \le C_{R,\delta,T,K_0}\eta;
\end{equation}
see \cite[Chapter~10]{KloedenPlaten1992}.

On the event $\{\tau_0>T\}$, the exact feedback path stays in $D_0\subset D_1$.  Since the
truncated and original coefficients agree on $D_1$, pathwise uniqueness up to the first exit from
$D_1$ gives
\begin{equation}\label{eq:truncated-exact-coincidence}
 X_t^{R,\delta}=X_t^*,
 \qquad0\le t\le T,
 \quad\text{on }\{\tau_0>T\}.
\end{equation}
Indeed, an exit of $X^{R,\delta}$ from $D_1$ before time $T$ would, by equality up to that exit,
force $X^*$ to reach $\partial D_1$, which is impossible while $X^*$ remains in $D_0$.

\smallskip
\noindent\textit{Step 3: the buffers keep the Euler grid points in $D_1$.}
Set
\begin{equation}\label{eq:euler-buffer}
 \varepsilon_{R,\delta}:=
 \min\left\{\frac R2,\frac{\delta}{\sqrt2}\right\}
\end{equation}
and consider the event
\begin{equation}\label{eq:euler-good-event}
 G_\eta:=\{\tau_0>T\}\cap
 \left\{
 \sup_{0\le t\le T}
 |\overline X_t^{\eta,R,\delta}-X_t^{R,\delta}|
 <\varepsilon_{R,\delta}
 \right\}.
\end{equation}
On $G_\eta$, relation \eqref{eq:truncated-exact-coincidence} shows that every grid value satisfies
\[
 |\overline X_{t_n}^{\eta,R,\delta}|
 <R+\frac R2<2R.
\]
Moreover, for every pair $i<j$,
\begin{align*}
 |\overline X_{t_n}^{\eta,R,\delta,i}
   -\overline X_{t_n}^{\eta,R,\delta,j}|
 &\ge |X_{t_n}^{*,i}-X_{t_n}^{*,j}|
   -|\overline X_{t_n}^{\eta,R,\delta,i}-X_{t_n}^{*,i}|
   -|\overline X_{t_n}^{\eta,R,\delta,j}-X_{t_n}^{*,j}|\\
 &\ge4\delta-\sqrt2
 |\overline X_{t_n}^{\eta,R,\delta}-X_{t_n}^*|\\
 &>4\delta-\sqrt2\,\varepsilon_{R,\delta}
 \ge3\delta>2\delta.
\end{align*}
Thus every truncated Euler grid value belongs to $D_1$.

We now prove by induction over the grid intervals that, on $G_\eta$, the original and truncated Euler
interpolations coincide.  They have the same initial value.  Suppose they coincide at $t_n$.  The
common grid value belongs to $D_1$, where $b^{R,\delta}=b$ and
$\sigma^{R,\delta}=\sigma$.  Both interpolations therefore use the same frozen coefficients and the
same Brownian increment on $[t_n,t_{n+1}\wedge T]$, and hence coincide throughout this interval.
Thus
\begin{equation}\label{eq:Euler-original-truncated-coincidence}
 \overline X_t^\eta=\overline X_t^{\eta,R,\delta},
 \qquad0\le t\le T,
 \quad\text{on }G_\eta.
\end{equation}

\smallskip
\noindent\textit{Step 4: removal of the localization.}
By Markov's inequality and \eqref{eq:global-truncated-Euler-estimate},
\begin{align}
 \sup_{X_0\in K_0}\P_{X_0}(G_\eta^c)
 &\le\sup_{X_0\in K_0}\P_{X_0}(\tau_0\le T)
 +\frac{C_{R,\delta,T,K_0}\eta}{\varepsilon_{R,\delta}^2}\notag\\
 &\le\varepsilon+
 \frac{C_{R,\delta,T,K_0}\eta}{\varepsilon_{R,\delta}^2}.
 \label{eq:Euler-good-event-probability}
\end{align}
On $G_\eta$, equations \eqref{eq:truncated-exact-coincidence} and
\eqref{eq:Euler-original-truncated-coincidence} give
\[
 \sup_{0\le t\le T}|\overline X_t^\eta-X_t^*|
 =\sup_{0\le t\le T}
 |\overline X_t^{\eta,R,\delta}-X_t^{R,\delta}|.
\]
Consequently, for every $\varepsilon_0>0$,
\begin{align*}
 &\sup_{X_0\in K_0}\P_{X_0}\left(
 \sup_{0\le t\le T}|\overline X_t^\eta-X_t^*|>\varepsilon_0
 \right)\\
 &\quad\le
 \varepsilon+
 \frac{C_{R,\delta,T,K_0}\eta}{\varepsilon_{R,\delta}^2}
 +\frac{C_{R,\delta,T,K_0}\eta}{\varepsilon_0^2}.
\end{align*}
First let $\eta\downarrow0$ and then $\varepsilon\downarrow0$.  This proves
\eqref{eq:euler-convergence-locally-uniform}, and hence
\eqref{eq:euler-convergence-prob}.
\end{proof}

\subsection{Geometric ergodicity and convergence of the record}

Recall the step-size threshold $\eta_0$ from \eqref{eq:eta0}.  For a measurable function
$V:\hX\to[1,\infty)$ and a finite signed measure $\zeta$ on $\hX$, define
\begin{equation}\label{eq:V-norm-definition}
 \norm{\zeta}_V:=\sup\left\{\left|\int_{\hX} f\,\dd\zeta\right|:
 f:\hX\to\R\text{ measurable and }|f|\le V\right\}.
\end{equation}
For $V\equiv1$, this is the usual total variation norm under the convention used here.

The deterministic mechanism turning recurrence into optimization is isolated in the next lemma.

\begin{lemma}\label{lem:running-record-principle}
Let $(Y_n)_{n\ge0}$ be any sequence in $\hX$, and let
\[
 \widehat Y_n:=Y_{\iota_n},
 \qquad
 \iota_n:=\min\left\{0\le k\le n:
 E(Y_k)=\min_{0\le\ell\le n}E(Y_\ell)\right\}.
\]
Suppose that, for every $j\ge1$, the open sublevel set
\[
 A_j:=\{X\in\hX:E(X)<E_*+j^{-1}\}
\]
is visited infinitely often by $(Y_n)$.  Then
\[
 E(\widehat Y_n)\longrightarrow E_*,
 \qquad
 \dist(\widehat Y_n,\Mc_*)\longrightarrow0.
\]
If $\Mc_*=\{X_*\}$, then $\widehat Y_n\to X_*$.
\end{lemma}

\begin{proof}
For every $j$, after the first visit to $A_j$ the running minimum remains below $E_*+j^{-1}$.
Since it is always bounded below by $E_*$, this proves
$E(\widehat Y_n)\to E_*$.  If the distance to $\Mc_*$ did not converge to zero, there would be
$\delta_0>0$ and a subsequence $n_k$ such that
\[
 \dist(\widehat Y_{n_k},\Mc_*)\ge\delta_0,
 \qquad
 E(\widehat Y_{n_k})\longrightarrow E_*.
\]
The coercive lower bound \eqref{eq:E-growth} makes $(\widehat Y_{n_k})$ bounded.  Passing to a
further subsequence, $\widehat Y_{n_k}\to\bar X$ in $\Xc^N$.  Continuity of $E$ gives
$E(\bar X)=E_*$, so $\bar X\in\Mc_*$ by Proposition~\ref{prop:reduction}, contradicting the
distance bound.  The singleton case is immediate.
\end{proof}

\begin{proof}[Proof of Theorem~\ref{thm:numerical-convergence}]
The proof below uses only $d\ge1$.  Let
\begin{equation}\label{eq:discrete-filtration}
 \Fc_n:=\sigma(X_0,\xi_1,\ldots,\xi_n),
 \qquad n\ge0,
\end{equation}
and let $P_{\eta,\tau}$ denote the transition kernel of \eqref{eq:euler-generic}.  Let $\psi$ be
$m$-dimensional Lebesgue measure restricted to $\hX$.

\smallskip
\noindent\textit{Step 1: the global quadratic drift inequality.}
Conditionally on $X_n=X$, the next state is Gaussian with mean
\[
 \mu_\eta(X):=X-\eta\nabla E(X)
\]
and covariance $2\eta\tau(X)I_m$.  Therefore, for $V(X)=1+|X|^2$,
\begin{equation}\label{eq:PV-computation}
 P_{\eta,\tau}V(X)
 =1+|X-\eta\nabla E(X)|^2+2\eta m\tau(X).
\end{equation}
On $\hX$, Proposition~\ref{prop:C1} identifies $\nabla E$ with the unique supergradient.  Hence
\eqref{eq:dissipativity-supergradient} and \eqref{eq:growth-supergradient} give
\begin{align}
 -2\eta X\cdot\nabla E(X)
 &\le-\pmin\eta|X|^2+\eta m_2(\nu),
 \label{eq:drift-term-chain-detailed}\\
 \eta^2|\nabla E(X)|^2
 &\le2\pmax^2\eta^2|X|^2
 +2\pmax m_2(\nu)\eta^2.
 \label{eq:square-gradient-chain-detailed}
\end{align}
The definition
\[
 \eta_0=\min\left\{1,\frac{\pmin}{4\pmax^2}\right\}
\]
ensures, for $0<\eta\le\eta_0$, that
\begin{equation}\label{eq:step-size-two-consequences}
 -\pmin\eta+2\pmax^2\eta^2
 \le-\frac{\pmin\eta}{2},
 \qquad
 \eta^2\le\eta.
\end{equation}
Using also $\tau\le1$ in \eqref{eq:PV-computation}, we obtain
\begin{equation}\label{eq:global-drift-chain}
 P_{\eta,\tau}V(X)
 \le\left(1-\frac{\pmin\eta}{2}\right)V(X)+B_0\eta,
 \qquad X\in\hX,
\end{equation}
where, for example,
\[
 B_0:=\frac{\pmin}{2}+m_2(\nu)+2\pmax m_2(\nu)+2m.
\]
The constant $B_0$ is independent of $X$ and of the Borel rule $\tau$.

Choose $R>0$ so large that
$B_0\le(\pmin/4)V(X)$ whenever $|X|>R$.  With
\begin{equation}\label{eq:gamma-chain-detailed}
 \gamma:=1-\frac{\pmin\eta}{4}\in(0,1),
 \qquad
 C_R:=\overline B_R(0)\cap\hX,
\end{equation}
we obtain the Foster--Lyapunov condition
\begin{equation}\label{eq:drift-small-set}
 P_{\eta,\tau}V(X)
 \le\gamma V(X)+b_R\id_{C_R}(X),
 \qquad
 b_R:=\sup_{X\in C_R}P_{\eta,\tau}V(X)<\infty.
\end{equation}
The finiteness of $b_R$ follows directly from \eqref{eq:global-drift-chain}; continuity of $\tau$
is not required.

\smallskip
\noindent\textit{Step 2: transition density, irreducibility, and a small set.}
The one-step law has the Lebesgue density on the full ambient space $\R^m$
\begin{equation}\label{eq:transition-density-chain}
 p_{\eta,\tau}(X,Y)
 =\frac1{(4\pi\eta\tau(X))^{m/2}}
 \exp\left(-\frac{|Y-\mu_\eta(X)|^2}{4\eta\tau(X)}\right).
\end{equation}
It is Borel measurable and strictly positive for every $X\in\hX$ and $Y\in\R^m$.  Since
$\Coll$ is Lebesgue-null, the kernel assigns probability one to $\hX$.  If
$A\subset\hX$ is Borel and $\psi(A)>0$, then
\[
 P_{\eta,\tau}(X,A)
 =\int_Ap_{\eta,\tau}(X,Y)\dd Y>0,
 \qquad X\in\hX.
\]
Thus the chain is $\psi$-irreducible.  This remains true when $d=1$, although $\hX$ then has
several connected components, because a Gaussian step can jump between them.

For $X\in C_R$, the global gradient-growth estimate gives a finite constant $M_R$ such that
$|\mu_\eta(X)|\le M_R$.  For $Y\in B_R(0)$,
$|Y-\mu_\eta(X)|\le R+M_R$.  Since $a\le\tau(X)\le1$,
\begin{equation}\label{eq:explicit-density-lower-chain}
 p_{\eta,\tau}(X,Y)
 \ge (4\pi\eta)^{-m/2}
 \exp\left(-\frac{(R+M_R)^2}{4\eta a}\right)
 =:c_R>0
\end{equation}
for every $X\in C_R$ and Lebesgue-almost every $Y\in C_R$.  Define
\[
 \nu_R(A):=\frac{\Leb^m(A\cap C_R)}{\Leb^m(C_R)}.
\]
Because $\Coll$ is Lebesgue-null,
$\Leb^m(C_R)=\Leb^m(B_R(0))>0$.  Hence
\begin{equation}\label{eq:minorization-chain}
 P_{\eta,\tau}(X,A)\ge\varepsilon_R\nu_R(A),
 \qquad X\in C_R,
 \qquad
 \varepsilon_R:=c_R\Leb^m(C_R)>0.
\end{equation}
Necessarily $\varepsilon_R\le1$, since the left-hand side is a probability kernel.  Thus $C_R$ is a
one-step small set.  Since $\nu_R(C_R)=1$, the minorization is strongly aperiodic, and the
$\psi$-irreducible chain is aperiodic.

\smallskip
\noindent\textit{Step 3: positive Harris recurrence, invariant moments, and geometric ergodicity.}
The drift condition \eqref{eq:drift-small-set} and the one-step small-set property imply positive
Harris recurrence and the existence of a unique invariant probability measure
$\Pi_{\eta,\tau}$; see \cite[Theorem~15.0.1]{MeynTweedie2009}.  Aperiodicity then gives convergence
in total variation from every initial state; see
\cite[Theorem~13.3.3]{MeynTweedie2009}.

We verify the second moment before invoking the weighted geometric conclusion.  Iterating
\eqref{eq:global-drift-chain} gives
\begin{equation}\label{eq:iterated-moment-chain}
 \sup_{n\ge0}P_{\eta,\tau}^nV(X)
 \le V(X)+\frac{2B_0}{\pmin}<\infty,
 \qquad X\in\hX.
\end{equation}
For $M>0$, the function $V\wedge M$ is bounded.  Total variation convergence therefore yields
\[
 \Pi_{\eta,\tau}(V\wedge M)
 =\lim_{n\to\infty}P_{\eta,\tau}^n(V\wedge M)(X)
 \le\sup_{n\ge0}P_{\eta,\tau}^nV(X).
\]
Letting $M\uparrow\infty$ and using monotone convergence gives
\begin{equation}\label{eq:invariant-second-moment-chain}
 \Pi_{\eta,\tau}(V)<\infty.
\end{equation}
The geometric drift theorem for an aperiodic $\psi$-irreducible chain now yields constants
$C<\infty$ and $r\in(0,1)$ such that \eqref{eq:V-geometric} holds; see
\cite[Theorem~16.1.2]{MeynTweedie2009}.

Let $O\subset\hX$ be nonempty and open.  Then $\Leb^m(O)>0$, and strict positivity of the Gaussian
density gives $P_{\eta,\tau}(X,O)>0$ for every $X\in\hX$.  Invariance implies
\[
 \Pi_{\eta,\tau}(O)
 =\int_{\hX}P_{\eta,\tau}(X,O)\Pi_{\eta,\tau}(\dd X)>0.
\]
Thus $\Pi_{\eta,\tau}$ has full support in $\hX$.

\smallskip
\noindent\textit{Step 4: infinitely many visits to every low-energy set.}
For $j\ge1$, let
\begin{equation}\label{eq:low-energy-set}
 A_j:=\left\{X\in\hX:E(X)<E_*+j^{-1}\right\}.
\end{equation}
Proposition~\ref{prop:reduction} provides a global minimizer in $\hX$, and continuity of $E$ makes
$A_j$ a nonempty open set.  By full support,
$\Pi_{\eta,\tau}(A_j)>0$.  The ergodic theorem for positive Harris chains gives, for every initial
state $X\in\hX$,
\begin{equation}\label{eq:ergodic-indicator-low-energy}
 \frac1n\sum_{k=0}^{n-1}\id_{A_j}(X_k)
 \longrightarrow\Pi_{\eta,\tau}(A_j)>0
 \qquad\P_X^\tau\text{-almost surely};
\end{equation}
see \cite[Theorem~17.0.1]{MeynTweedie2009}.  Hence every $A_j$ is visited infinitely often almost
surely.  Applying Lemma~\ref{lem:running-record-principle} pathwise proves
\eqref{eq:record-energy-convergence} and \eqref{eq:record-distance-convergence}, together with the
singleton-minimizer conclusion.

\smallskip
\noindent\textit{Step 5: the raw iterates cannot converge.}
Fix an integer $K\ge1$.  The global gradient-growth estimate gives
\[
 M_K:=\sup\{|\nabla E(X)|:X\in\hX,\ |X|\le K\}<\infty.
\]
On the event $\{|X_n|\le K\}$, the reverse triangle inequality and $\tau(X_n)\ge a$ imply
\begin{align}
 &\P_X^\tau\left(|X_{n+1}-X_n|>1\mid\Fc_n\right)\notag\\
 &\quad\ge
 \P\left(
 |\xi_{n+1}|>\frac{1+\eta M_K}{\sqrt{2\eta a}}
 \right)
 =:c_K>0.
 \label{eq:uniform-large-increment-chain}
\end{align}
The constant $c_K$ is deterministic and independent of $n$.  For integers $N,L\ge1$, repeated
conditioning gives
\begin{align}
 &\P_X^\tau\left(
 |X_n|\le K\text{ and }|X_{n+1}-X_n|\le1
 \text{ for every }N\le n<N+L
 \right)\notag\\
 &\qquad\le(1-c_K)^L.
 \label{eq:block-small-increments-chain}
\end{align}
Letting $L\to\infty$ shows that, with probability one, the chain cannot remain forever in $B_K$
while all subsequent increments are at most one.  If $X_n$ converged to a finite limit, then for
some integers $K,N$ it would remain in $B_K$ and satisfy
$|X_{n+1}-X_n|\le1$ for every $n\ge N$.  Taking the countable union over $K,N$ proves
\eqref{eq:raw-not-converge}.
\end{proof}

\begin{remark}\label{rem:role-feedback}
Record convergence in Theorem~\ref{thm:numerical-convergence} uses only two robust properties:
confinement by the exact drift $-\nabla E$ and uniformly positive Gaussian exploration.  It does not
say that all temperature rules have the same finite-time performance.  The HJB feedback of
Theorem~\ref{thm:HJB} is optimal for the discounted criterion \eqref{eq:control-cost}; the discrete
theorem supplies the separate asymptotic guarantee that replacing it by any fixed clipped Borel
approximation does not destroy convergence of the running record.  No claim is made that
$\tau_\lambda$ minimizes hitting times, optimizes finite-time records, or gives the fastest record
convergence.
\end{remark}

\begin{remark}\label{rem:positive-floor-invariant-law}
The positive floor $a>0$ is essential for irreducibility and the full-support conclusion.  At fixed
$a$, the theorem does not assert that the invariant measures $\Pi_{\eta,\tau}$ concentrate near
$\Mc_*$ as $\eta\downarrow0$.  Persistent exploration is also why the raw iterates do not converge;
the optimization statement concerns the running record rather than the stationary state itself.
\end{remark}

\appendix

\section{Joint \texorpdfstring{$C^2$}{C2} regularity of the semi-dual}
\label{app:parametric}

This appendix records the precise joint site-and-weight differentiation result used in
Theorem~\ref{thm:C2} and verifies explicitly the hypotheses of the semi-discrete regularity theorem
invoked from \cite{deGournayKahnLebrat2019}.  Define the positive-cell parameter set
\begin{equation}\label{eq:positive-cell-set}
 \Oc_+:=\left\{(X,\Phi)\in\hX\times U:
 \Leb^d(V_i(X,\Phi))>0\text{ for every }i\right\}.
\end{equation}
Because $\rho$ is bounded above and below by positive constants on $\Omega$, positivity of
Lebesgue volume is equivalent to positivity of $\nu$-mass.  For $(X,\Phi)\in\Oc_+$ and $i\ne j$,
put
\begin{equation}\label{eq:general-facets-app}
 F_{ij}(X,\Phi):=\overline{V_i(X,\Phi)}\cap\overline{V_j(X,\Phi)}.
\end{equation}
All surface integrals below are taken with respect to the normalized Hausdorff measure
$\HH^{d-1}$ fixed in \eqref{eq:restricted-Hausdorff-main}.  Lower-dimensional multiple
intersections are therefore invisible to these integrals.

\begin{proposition}\label{prop:joint-C2-semidual}
The set $\Oc_+$ is relatively open in $\hX\times U$, and the semi-dual $\Kant$ in
\eqref{eq:semi-dual-main} belongs to $C^2(\Oc_+)$.  Let $\Oc$ be an open neighborhood of
$\Omega$ and let $q\in C^1(\Oc;\R^k)$.  For
\begin{equation}\label{eq:Jq-app}
 J_i^q(X,\Phi):=\int_{V_i(X,\Phi)}q(y)\rho(y)\dd y,
\end{equation}
one has $J_i^q\in C^1(\Oc_+;\R^k)$ and, for
$\Xi=(\xi_i)\in(\R^d)^N$ and $\Psi=(\psi_i)\in U$,
\begin{align}
 DJ_i^q(X,\Phi)[\Xi,\Psi]
 ={}&\sum_{j\ne i}\int_{F_{ij}(X,\Phi)}
 q(y)\rho(y)
 \frac{(y-x_i)\cdot\xi_i-(y-x_j)\cdot\xi_j+
 \psi_i-\psi_j}{|x_i-x_j|}
 \dd\HH^{d-1}(y).
 \label{eq:moving-cell-app}
\end{align}
In particular, for the mass map $M$ in \eqref{eq:mass-moment-main},
\begin{align}
 D_\Phi M(X,\Phi)[\Psi]_i
 &=\sum_{j\ne i}\int_{F_{ij}(X,\Phi)}
 \frac{\rho(y)}{|x_i-x_j|}(\psi_i-\psi_j)\dd\HH^{d-1}(y),
 \label{eq:DphiM-app}\\
 D_XM(X,\Phi)[\Xi]_i
 &=\sum_{j\ne i}\int_{F_{ij}(X,\Phi)}
 \frac{\rho(y)}{|x_i-x_j|}
 \left((y-x_i)\cdot\xi_i-(y-x_j)\cdot\xi_j\right)
 \dd\HH^{d-1}(y).
 \label{eq:DXM-app}
\end{align}
\end{proposition}

\begin{proof}
We divide the verification into three steps.

\smallskip
\noindent\textit{Step 1: openness of the positive-cell set and local nondegeneracy.}
For fixed $i$, the map
\[
 (X,\Phi)\longmapsto \Leb^d(V_i(X,\Phi))
\]
is continuous on $\hX\times U$.  Indeed, if $(X^n,\Phi^n)\to(X,\Phi)$, then, outside the finite
union of the limiting competition hyperplanes, the minimizer in the definition of the Laguerre
cells is unique and remains unchanged for all sufficiently large $n$.  Hence
\[
 \id_{V_i(X^n,\Phi^n)}(y)
 \longrightarrow
 \id_{V_i(X,\Phi)}(y)
 \qquad\text{for Lebesgue-a.e. }y\in\Omega,
\]
and dominated convergence gives convergence of the cell volumes.  This proves that $\Oc_+$ is
relatively open.

Fix $(X^0,\Phi^0)\in\Oc_+$.  Since $X^0\in\hX$ and every cell has positive volume, there are
constants $r_0>0$ and $m_0>0$ and a neighborhood $\Nc$ of $(X^0,\Phi^0)$ in $\hX\times U$ such
that, for every $(X,\Phi)\in\Nc$,
\begin{equation}\label{eq:local-positive-separated-app}
 \min_{i\ne j}|x_i-x_j|\ge r_0,
 \qquad
 \min_{1\le i\le N}\Leb^d(V_i(X,\Phi))\ge m_0.
\end{equation}
After shrinking $\Nc$, we may and do assume that $\Nc\subset\Oc_+$.  Thus distinctness of the
sites and positivity of all cells hold on a full parameter neighborhood, rather than only at the
base point.

\smallskip
\noindent\textit{Step 2: verification of the joint $C^2$ theorem.}
Set
\[
 c_i(X;y):=\frac12|y-x_i|^2,
 \qquad
 h_{ij}(X,\Phi;y):=c_i(X;y)-\phi_i-c_j(X;y)+\phi_j.
\]
We match our setting with the Euclidean specialization
\cite[Proposition~1.1]{deGournayKahnLebrat2019}.  The interior
$\Omega^\circ$ is a bounded convex Lipschitz domain, and replacing the compact body $\Omega$ by
$\Omega^\circ$ does not alter any integral because $\partial\Omega$ has Lebesgue measure zero.  By
Assumption~\ref{ass:geometry},
\[
 \rho\in C^0(\Omega)\cap W^{1,1}(\Omega^\circ)\cap L^\infty(\Omega^\circ).
\]
Moreover, throughout the neighborhood $\Nc$, the sites are pairwise distinct and all Laguerre cells
have positive volume by \eqref{eq:local-positive-separated-app}.  These are exactly the additional
hypotheses imposed in \cite[Proposition~1.1]{deGournayKahnLebrat2019} for the quadratic Euclidean
cost.  That proposition verifies, in particular, the pairwise-interface, multiple-intersection, and
fixed-boundary conditions required by the general second-order theorem.  The basic transversality is
also explicit here:
\begin{equation}\label{eq:transversality-app}
 |\nabla_yh_{ij}(X,\Phi;y)|
 =|x_i-x_j|
 \ge r_0,
 \qquad (X,\Phi)\in\Nc.
\end{equation}
Thus the hypotheses of \cite[Theorem~1]{deGournayKahnLebrat2019}, including its continuity
condition for the second derivatives, hold at every point of $\Nc$.  The theorem gives joint twice
differentiability, in $(X,\Phi)$, of the integral part of $\Kant$ and continuity of its Hessian.
Adding the affine term $\sum_i p_i\phi_i$ proves
\[
 \Kant\in C^2(\Nc).
\]
Since the base point was arbitrary, $\Kant\in C^2(\Oc_+)$.  In particular, differentiating the
definition of $\Kant$ with respect to the weights gives the exact identity
\[
 M(X,\Phi)=-D_\Phi\Kant(X,\Phi),
\]
so the mass map is jointly $C^1$ on $\Oc_+$.  This is the continuous differentiability, rather than
merely pointwise existence of derivatives, required in the implicit-function argument of
Section~\ref{sec:objective}.

\smallskip
\noindent\textit{Step 3: differentiation of moving cell integrals.}
Because $q\in C^1(\Oc;\R^k)$ and $\rho\in W^{1,1}(\Omega^\circ)\cap L^\infty(\Omega^\circ)$,
each component of $q\rho$ belongs to
$W^{1,1}(\Omega^\circ)\cap L^\infty(\Omega^\circ)$.  We may therefore apply the moving-domain
formula of \cite[Lemmas~1.1--1.2]{deGournayKahnLebrat2019} componentwise.

The cell is
\[
 V_i(X,\Phi)=\Omega\cap\bigcap_{j\ne i}\{h_{ij}(X,\Phi;\cdot)\le0\}.
\]
Along a parameter curve with velocity $(\Xi,\Psi)$, differentiation of
$h_{ij}(X,\Phi;y)=0$ shows that the normal velocity of the $i$--$j$ interface, measured in the
outward normal direction of $V_i$, is
\begin{equation}\label{eq:normal-velocity-app}
 \frac{(y-x_i)\cdot\xi_i-(y-x_j)\cdot\xi_j+
 \psi_i-\psi_j}{|x_i-x_j|}.
\end{equation}
The outer boundary $\partial\Omega$ is fixed and therefore has zero parameter velocity.  The
multiple-intersection strata have zero $\HH^{d-1}$-measure by the verification in Step~2, so they
produce no additional surface term.  Integrating \eqref{eq:normal-velocity-app} against $q\rho$ on
each internal facet gives \eqref{eq:moving-cell-app}.  The same theorem gives continuity of this
derivative with respect to $(X,\Phi)$, hence $J_i^q\in C^1(\Oc_+;\R^k)$.

Taking $q=1$ yields \eqref{eq:DphiM-app}--\eqref{eq:DXM-app}.  Taking $q(y)=y$ gives the moment
derivative used in \eqref{eq:Dcentroid-main}.
\end{proof}

\begin{remark}\label{rem:compact-C2}
If $K\Subset\hX$, continuity of $\Phi^*$ implies that the graph
\[
 \{(X,\Phi^*(X)):X\in K\}
\]
is compact.  Every balanced cell has $\nu$-mass $p_i>0$, and hence positive Lebesgue volume, so this
graph is a compact subset of $\Oc_+$.  The continuity conclusions of
Proposition~\ref{prop:joint-C2-semidual}, together with the continuity and invertibility of
$L_X|_U$, are therefore uniform on $K$.  This is the compactness input used in
Theorem~\ref{thm:C2}; no H\"older exponent for the joint second derivatives is asserted.
\end{remark}


\begin{thebibliography}{99}

\bibitem{BourneRoper2015}
D.~P. Bourne and S.~M. Roper,
\newblock Centroidal power diagrams, Lloyd's algorithm, and applications to optimal location problems,
\newblock \emph{SIAM J. Numer. Anal.} \textbf{53} (2015), no.~6, 2545--2569.

\bibitem{ButtazzoSantambrogio2009}
G.~Buttazzo and F.~Santambrogio,
\newblock A mass transportation model for the optimal planning of an urban region,
\newblock \emph{SIAM Rev.} \textbf{51} (2009), no.~3, 593--610.

\bibitem{ChiangHwangSheu1987}
T.-S.~Chiang, C.-R.~Hwang, and S.-J.~Sheu,
\newblock Diffusion for global optimization in $\mathbb R^n$,
\newblock \emph{SIAM J. Control Optim.} \textbf{25} (1987), no.~3, 737--753.

\bibitem{deGournayKahnLebrat2019}
F.~de~Gournay, J.~Kahn, and L.~Lebrat,
\newblock Differentiation and regularity of semi-discrete optimal transport with respect to the parameters of the discrete measure,
\newblock \emph{Numer. Math.} \textbf{141} (2019), no.~2, 429--453.

\bibitem{DuEmelianenkoJu2006}
Q.~Du, M.~Emelianenko, and L.~Ju,
\newblock Convergence of the Lloyd algorithm for computing centroidal Voronoi tessellations,
\newblock \emph{SIAM J. Numer. Anal.} \textbf{44} (2006), no.~1, 102--119.

\bibitem{DuFaberGunzburger1999}
Q.~Du, V.~Faber, and M.~Gunzburger,
\newblock Centroidal Voronoi tessellations: applications and algorithms,
\newblock \emph{SIAM Rev.} \textbf{41} (1999), no.~4, 637--676.

\bibitem{EmelianenkoJuRand2008}
M.~Emelianenko, L.~Ju, and A.~Rand,
\newblock Nondegeneracy and weak global convergence of the Lloyd algorithm in $\mathbb R^d$,
\newblock \emph{SIAM J. Numer. Anal.} \textbf{46} (2008), no.~3, 1423--1441.


\bibitem{GaoXuZhou2022}
X.~Gao, Z.~Q. Xu, and X.~Y. Zhou,
\newblock State-dependent temperature control for Langevin diffusions,
\newblock \emph{SIAM J. Control Optim.} \textbf{60} (2022), no.~3, 1250--1268.

\bibitem{GelfandMitter1991}
S.~B. Gelfand and S.~K. Mitter,
\newblock Recursive stochastic algorithms for global optimization in $\mathbb R^d$,
\newblock \emph{SIAM J. Control Optim.} \textbf{29} (1991), no.~5, 999--1018.

\bibitem{GrafLuschgy2000}
S.~Graf and H.~Luschgy,
\newblock \emph{Foundations of Quantization for Probability Distributions},
\newblock Lecture Notes in Mathematics, vol.~1730, Springer, Berlin, 2000.

\bibitem{HolleyKusuokaStroock1989}
R.~A. Holley, S.~Kusuoka, and D.~W. Stroock,
\newblock Asymptotics of the spectral gap with applications to the theory of simulated annealing,
\newblock \emph{J. Funct. Anal.} \textbf{83} (1989), no.~2, 333--347.

\bibitem{KitagawaMerigotThibert2019}
J.~Kitagawa, Q.~M\'erigot, and B.~Thibert,
\newblock Convergence of a Newton algorithm for semi-discrete optimal transport,
\newblock \emph{J. Eur. Math. Soc.} \textbf{21} (2019), no.~9, 2603--2651.

\bibitem{KloedenPlaten1992}
P.~E. Kloeden and E.~Platen,
\newblock \emph{Numerical Solution of Stochastic Differential Equations},
\newblock Springer, Berlin, 1992.

\bibitem{LianWangZhang2026}
Y.~Lian, L.~Wang, and K.~Zhang,
\newblock Pointwise regularity for fully nonlinear elliptic equations in general forms,
\newblock arXiv:2012.00324v3, 2026.

\bibitem{MerigotSantambrogioSarrazin2021}
Q.~M\'erigot, F.~Santambrogio, and C.~Sarrazin,
\newblock Non-asymptotic convergence bounds for Wasserstein approximation using point clouds,
\newblock in \emph{Advances in Neural Information Processing Systems 34}, 2021, 12810--12821.

\bibitem{MeynTweedie2009}
S.~P. Meyn and R.~L. Tweedie,
\newblock \emph{Markov Chains and Stochastic Stability},
\newblock 2nd ed., Cambridge University Press, Cambridge, 2009.

\bibitem{PortalesCazellesPauwels2025}
L.~Portales, E.~Cazelles, and E.~Pauwels,
\newblock On the sequential convergence of Lloyd's algorithms,
\newblock \emph{Math. Oper. Res.} \textbf{51} (2025), no.~2, 1120--1138.

\bibitem{RaginskyRakhlinTelgarsky2017}
M.~Raginsky, A.~Rakhlin, and M.~Telgarsky,
\newblock Non-convex learning via stochastic gradient Langevin dynamics: a nonasymptotic analysis,
\newblock in \emph{Proceedings of the 2017 Conference on Learning Theory},
Proceedings of Machine Learning Research, vol.~65, 2017, 1674--1703.

\bibitem{RevuzYor1999}
D.~Revuz and M.~Yor,
\newblock \emph{Continuous Martingales and Brownian Motion},
\newblock 3rd ed., Springer, Berlin, 1999.

\bibitem{TangZhangZhou2022}
W.~Tang, Y.~P. Zhang, and X.~Y. Zhou,
\newblock Exploratory HJB equations and their convergence,
\newblock \emph{SIAM J. Control Optim.} \textbf{60} (2022), no.~6, 3191--3216.

\bibitem{Villani2009}
C.~Villani,
\newblock \emph{Optimal Transport: Old and New},
\newblock Springer, Berlin, 2009.

\bibitem{WangLiWangZhang2026}
T.~Wang, X.~Li, G.~Wang, and Z.~Zhang,
\newblock State-dependent temperature control in Langevin diffusions using numerical exploratory Hamiltonian--Jacobi--Bellman equations,
\newblock arXiv:2603.17934, 2026.

\bibitem{WangZariphopoulouZhou2020}
H.~Wang, T.~Zariphopoulou, and X.~Y. Zhou,
\newblock Reinforcement learning in continuous time and space: a stochastic control approach,
\newblock \emph{J. Mach. Learn. Res.} \textbf{21} (2020), paper no.~198, 1--34.

\end{thebibliography}
\end{document}